\documentclass[11pt,reqno]{amsart}

\newtheorem{thm}{Theorem}[section]
\newtheorem*{thm*}{Theorem}
\newtheorem{lem}[thm]{Lemma}
\newtheorem*{lem*}{Lemma}

\newtheorem{cor}[thm]{Corollary}
\newtheorem{claim}[thm]{Claim}
\newtheorem{prop}[thm]{Proposition}

\theoremstyle{definition}

\newtheorem*{case*}{Case}

\newtheorem{defn}[thm]{Definition}
\newtheorem*{defn*}{Definition}

\newtheorem*{exmp*}{Example}

\newtheorem{hyp}[thm]{Hypothesis}

\newtheorem{step}{Step}\renewcommand{\thestep}{}

\theoremstyle{remark}

\newtheorem{case}{Case}\renewcommand{\thecase}{}

\newtheorem{rmk}[thm]{Remark}
\newtheorem*{rmk*}{Remark}

\makeatletter
\def\alphenumi{
  \def\theenumi{\alph{enumi}}
  \def\p@enumi{\theenumi}
  \def\labelenumi{(\@alph\c@enumi)}}
\makeatother

\makeatletter
\def\thecase{\@arabic\c@case}
\makeatother

\makeatletter
\def\thestep{\@arabic\c@step}
\makeatother

\newcount\hh
\newcount\mm
\mm=\time
\hh=\time
\divide\hh by 60
\divide\mm by 60
\multiply\mm by 60
\mm=-\mm
\advance\mm by \time
\def\hhmm{\number\hh:\ifnum\mm<10{}0\fi\number\mm}

\let\oldmarginpar\marginpar
\renewcommand\marginpar[1]{\-\oldmarginpar[\raggedleft\footnotesize #1]%
{\raggedright\footnotesize #1}}

\newcommand\CC{\mathbb{C}}

\newcommand\HH{\mathbb{H}}

\newcommand\NN{\mathbb{N}}

\newcommand\RR{\mathbb{R}}

\newcommand\cC{{\mathcal{C}}}

\newcommand\sA{{\mathscr{A}}}

\newcommand\sH{{\mathscr{H}}}

\newcommand\sO{{\mathscr{O}}}

\newcommand\sS{{\mathscr{S}}}

\newcommand\eps{\varepsilon}

\newcommand\less{\setminus}

\newcommand\diam{\operatorname{diam}}

\newcommand\dist{\operatorname{dist}}

\DeclareMathOperator{\height}{height}

\DeclareMathOperator{\Int}{int}

\newcommand\loc{\operatorname{loc}}

\newcommand\supp{\operatorname{supp}}

\newcommand\tr{\operatorname{tr}}

\numberwithin{equation}{section}

\usepackage{amssymb}
\usepackage{graphicx}
\usepackage{hyperref}
\usepackage{mathrsfs}
\usepackage[usenames]{color}
\usepackage{url}
\usepackage{verbatim}

\hypersetup{pdftitle={Schauder a priori estimates and regularity of solutions to boundary-degenerate elliptic linear second-order partial differential equations}}
\hypersetup{pdfauthor={Paul M. N. Feehan and Camelia A. Pop}}

\begin{document}

\title[Regularity of solutions to boundary-degenerate elliptic equations]
{Schauder a priori estimates and regularity of solutions to boundary-degenerate elliptic linear second-order partial differential equations}

\author[P. M. N. Feehan]{Paul M. N. Feehan}
\address[PF]{Department of Mathematics, Rutgers, The State University of New Jersey, 110 Frelinghuysen Road, Piscataway, NJ 08854-8019, United States}
\email{feehan@math.rutgers.edu}

\author[C. A. Pop]{Camelia A. Pop}
\address[CP]{Department of Mathematics, University of Pennsylvania, 209 South 33rd Street, Philadelphia, PA 19104-6395, United States}
\email{cpop@math.upenn.edu}

\date{September 4, 2013. Incorporates final galley proof corrections corresponding to published version. To appear in the Journal of Differential Equations, \url{dx.doi.org/10.1016/j.jde.2013.08.012}.}

\begin{abstract}
We establish Schauder a priori estimates and regularity for solutions to a class of boundary-degenerate elliptic linear second-order partial differential equations. Furthermore,  given a $C^\infty$-smooth source function, we prove $C^\infty$-regularity of solutions up to the portion of the boundary where the operator is degenerate. Boundary-degenerate elliptic operators of the kind described in our article appear in a diverse range of applications, including as generators of affine diffusion processes employed in stochastic volatility models in mathematical finance \cite{DuffiePanSingleton2000, Heston1993}, generators of diffusion processes arising in mathematical biology \cite{Athreya_Barlow_Bass_Perkins_2002, Epstein_Mazzeo_annmathstudies}, and the study of porous media \cite{DaskalHamilton1998, Daskalopoulos_Rhee_2003}.
\end{abstract}

%
%
%
%


%



\thanks{PF was partially supported by NSF grant DMS-1059206, the Max Planck Institut f\"ur Mathematik in der Naturwissenschaft, Leipzig, the Max Planck Institut f\"ur Mathematik, Bonn, Germany, and the Department of Mathematics at Columbia University.}

\maketitle
\tableofcontents
\listoffigures

\section{Introduction}
\label{sec:Introduction}
This article continues our development of regularity theory for solutions to the `partial Dirichlet' boundary value
problem\footnote{In the sense that a Dirichlet boundary condition along only a portion of the boundary is required for uniqueness.}
defined by a `boundary-degenerate elliptic' operator.
We use the term `boundary-degenerate elliptic' in this article to clarify the distinction with the term `degenerate elliptic' as used by M. G. Crandall, H. Ishii, and P.-L. Lions in \cite{Crandall_Ishii_Lions_1992} and the operators considered in this article which are locally strictly elliptic on the interior of an open subset but fail to be strictly elliptic along a portion of its boundary.
Boundary-degenerate elliptic operators of the kind explored in our article can arise as generators of affine diffusion processes employed in stochastic volatility models in mathematical finance \cite{DuffiePanSingleton2000, Heston1993}, generators of diffusion processes arising in mathematical biology
\cite{Athreya_Barlow_Bass_Perkins_2002, Epstein_Mazzeo_annmathstudies}, and the analysis of porous media \cite{DaskalHamilton1998, Daskalopoulos_Rhee_2003}, to name just a few applications.

In \cite{Daskalopoulos_Feehan_statvarineqheston}, in addition to other results, P. Daskalopoulos and the first author obtained existence of $H^1$ solutions to a variational equation defined by the Heston operator \cite{Heston1993}. We recall that the Heston operator serves as a useful paradigm for boundary-degenerate elliptic operators arising in mathematical finance. In \cite{Feehan_Pop_regularityweaksoln}, the present authors proved global $C^\alpha_s$-regularity of $H^1$ solutions to the variational equation defined by the Heston operator, while in \cite{Feehan_Pop_higherregularityweaksoln}, we established $\sH^k$ as well as $C^{k,\alpha}_s$ and $C^{k,2+\alpha}_s$ regularity for those solutions, for all integers $k\geq 0$.
(We refer to \cite{Feehan_Pop_higherregularityweaksoln} for the precise definition of the Sobolev spaces $H^1$ and $\sH^k$; the H\"older spaces $C^\alpha_s$, $C^{k,\alpha}_s$, and $C^{k,2+\alpha}_s$ are defined in Section \ref{sec:Preliminaries}.)
However, our $C^{k,\alpha}_s$ and $C^{k,2+\alpha}_s$ regularity results in \cite{Feehan_Pop_higherregularityweaksoln}, although they provide an important stepping stone, are not optimal due to our reliance on variational methods. The purpose of the present article is to prove analogues --- for a broad class of boundary-degenerate elliptic operators --- of the Schauder a priori estimates and regularity results for strictly elliptic operators in \cite[Chapter 6]{GilbargTrudinger}. When coupled with results of \cite{Daskalopoulos_Feehan_statvarineqheston, Feehan_Pop_regularityweaksoln, Feehan_Pop_higherregularityweaksoln}, we immediately obtain existence and $C^{k,2+\alpha}_s$ regularity for solutions to the Dirichlet boundary value problem, defined by a boundary-degenerate elliptic operator, analogous to those expected from the Schauder approach for strictly elliptic operators in \cite[Chapter 6]{GilbargTrudinger}; uniqueness for a broad class of linear second-order boundary-degenerate elliptic operators, with the second-order (or Ventcel) boundary conditions of the kind implied by our choice of Daskalopoulos-Hamilton $C^{k,2+\alpha}_s$ H\"older spaces \cite{DaskalHamilton1998}, is a consequence of the weak maximum principle discussed by the first author in \cite{Feehan_maximumprinciple}.

To describe our results in more detail, suppose $\sO\subseteqq\HH$ is an open subset (possibly unbounded) in the open upper half-space, $\HH := \RR^{d-1}\times\RR_+$, where $d\geq 2$ and $\RR_+ := (0,\infty)$, and $\partial_1\sO := \partial\sO\cap\HH$ is the portion of the boundary, $\partial\sO$, of $\sO$ which lies in $\HH$, and $\partial_0\sO $ is the interior of $\partial\HH\cap\partial\sO$, where $\partial\HH = \RR^{d-1}\times\{0\}$ is the boundary of $\bar\HH := \RR^{d-1}\times\bar\RR_+$ and $\bar\RR_+ := [0,\infty)$. We assume $\partial_0\sO$ is non-empty and consider a linear second-order elliptic differential operator, $A$, on $\sO$ which is degenerate along $\partial_0\sO$. In this article, when the operator $A$ is given by \eqref{eq:defnA}, we prove an a priori interior Schauder estimate and higher-order H\"older regularity up to the boundary portion, $\partial_0\sO$ --- as measured by certain weighted H\"older spaces, $C^{k,2+\alpha}_s(\underline\sO)$
(Definition \ref{defn:DH2spaces}) --- for solutions to the elliptic boundary value problem,
\begin{align}
\label{eq:IntroBoundaryValueProblem}
Au &= f \quad \hbox{on }\sO,
\\
\label{eq:IntroBoundaryValueProblemBC}
u &= g \quad \hbox{on } \partial_1\sO,
\end{align}
where $f:\sO\to\RR$ is a source function and the function $g:\partial_1\sO\to\RR$ prescribes a partial Dirichlet boundary condition. We denote $\underline{\sO} := \sO\cup\partial_0\sO$ throughout our article, while $\bar \sO=\sO\cup\partial\sO$ denotes the usual topological closure of $\sO$ in $\RR^d$. Furthermore, when $f \in C^\infty(\underline\sO)$, we will also show that $u\in C^\infty(\underline\sO)$ (see Corollary \ref{cor:CinftyGlobal}). Since $A$ becomes degenerate along $\partial_0\sO$, such regularity results do not follow from the standard theory for strictly elliptic differential operators \cite{GilbargTrudinger, Krylov_LecturesHolder}.

The boundary-degenerate elliptic operators considered in this article have the form\footnote{The operator $-A$ is the generator of a degenerate-diffusion process with killing.}
\begin{equation}
\label{eq:defnA}
Av :=  -x_d\tr(aD^2v) - b\cdot Dv + cv \quad\hbox{on }\sO, \quad v\in C^\infty(\sO),
\end{equation}
where $x=(x_1,\ldots,x_d)$ are the standard coordinates on $\RR^d$, and the coefficients of $A$ are given by a matrix-valued function, $a=(a^{ij}):\underline\sO\to\sS^+(d)$, a vector field, $b=(b^i):\underline\sO\to\RR^d$, and a function, $c:\underline\sO\to \RR$, where $\sS(d)\subset \RR^{d\times d}$ is the subset of symmetric matrices and $\sS^+(d)\subset \RR^{d\times d}$ is the subset of non-negative definite matrices. We shall call $A$ in \eqref{eq:defnA} an \emph{operator with constant coefficients} if the coefficients $a,b,c$ are constant. Occasionally we shall also need
\begin{equation}
\label{eq:defnA_0}
A_0 v := (A-c)v = -x_d\tr(aD^2v) - b\cdot Dv \quad\hbox{on }\sO, \quad v\in C^\infty(\sO).
\end{equation}
Throughout this article, we shall assume that there is a positive constant, $b_0$, such that
$$
b^d \geq b_0 \quad \hbox{on } \partial_0 \sO.
$$
Because the coefficient, $b^d$, is assumed to obey a positive lower bound along $\partial_0\sO$,
\emph{no boundary condition} need be prescribed for the equation \eqref{eq:IntroBoundaryValueProblem} along $\partial_0\sO$. Indeed, one expects from \cite{DaskalHamilton1998} that the problem \eqref{eq:IntroBoundaryValueProblem}, \eqref{eq:IntroBoundaryValueProblemBC} should be well-posed, given $f\in C^\alpha_s(\underline\sO)$ and $g\in C(\partial_1\sO)$ obeying mild pointwise growth conditions, when we seek solutions in $C^{2+\alpha}_s(\underline\sO)\cap C(\sO\cup\partial_1\sO)$.

In \cite{Feehan_Pop_higherregularityweaksoln}, we proved existence and uniqueness of a solution, $u \in C^{2+\alpha_0}_s(\underline\sO)\cap C(\bar\sO)$ for \emph{some} $\alpha_0=\alpha_0\in(0,1)$, to \eqref{eq:IntroBoundaryValueProblem}, \eqref{eq:IntroBoundaryValueProblemBC} when $\partial_1\sO$ obeys a uniform exterior cone condition with cone $K$, and $A$ is the elliptic Heston operator, and $f\in C^\infty(\underline\sO)\cap C_b(\sO)$ and $g \in C^\infty(\overline{\partial_1\sO})$.  (The H\"older exponent, $\alpha_0$, depends on the coefficients of $A$ and the cone $K$.) In Section \ref{subsec:Summary}, we state the main results of our article and set them in context in Section \ref{subsec:Survey}, where we discuss connections with previous related research by other authors. In Section \ref{subsec:Extensions}, we indicate some extensions of methods and results in our article which we plan to develop in subsequent articles. We provide a guide in Section \ref{subsec:Guide} to the remainder of this article and point out some of the
mathematical difficulties and issues of broader interest.  We refer the reader to Section \ref{subsec:Notation} for our notational conventions.

\subsection{Summary of main results}
\label{subsec:Summary}
Throughout our article, our use of the term `interior' is in the sense intended by \cite{DaskalHamilton1998}, for example, $U\subset\sO$ is an \emph{interior} open subset of an open subset $\sO\subseteqq\HH$ if $\bar U \subset \underline{\sO}$ and by `interior regularity' of a function $u$ on $\sO$, we mean regularity of $u$ up to $\partial_0\sO$ --- see Figure \ref{fig:higher_order_heston_regularity_regions}.

\begin{figure}[htbp]
\centering
\begin{picture}(200,200)(0,0)
\put(0,0){\includegraphics[height=200pt]{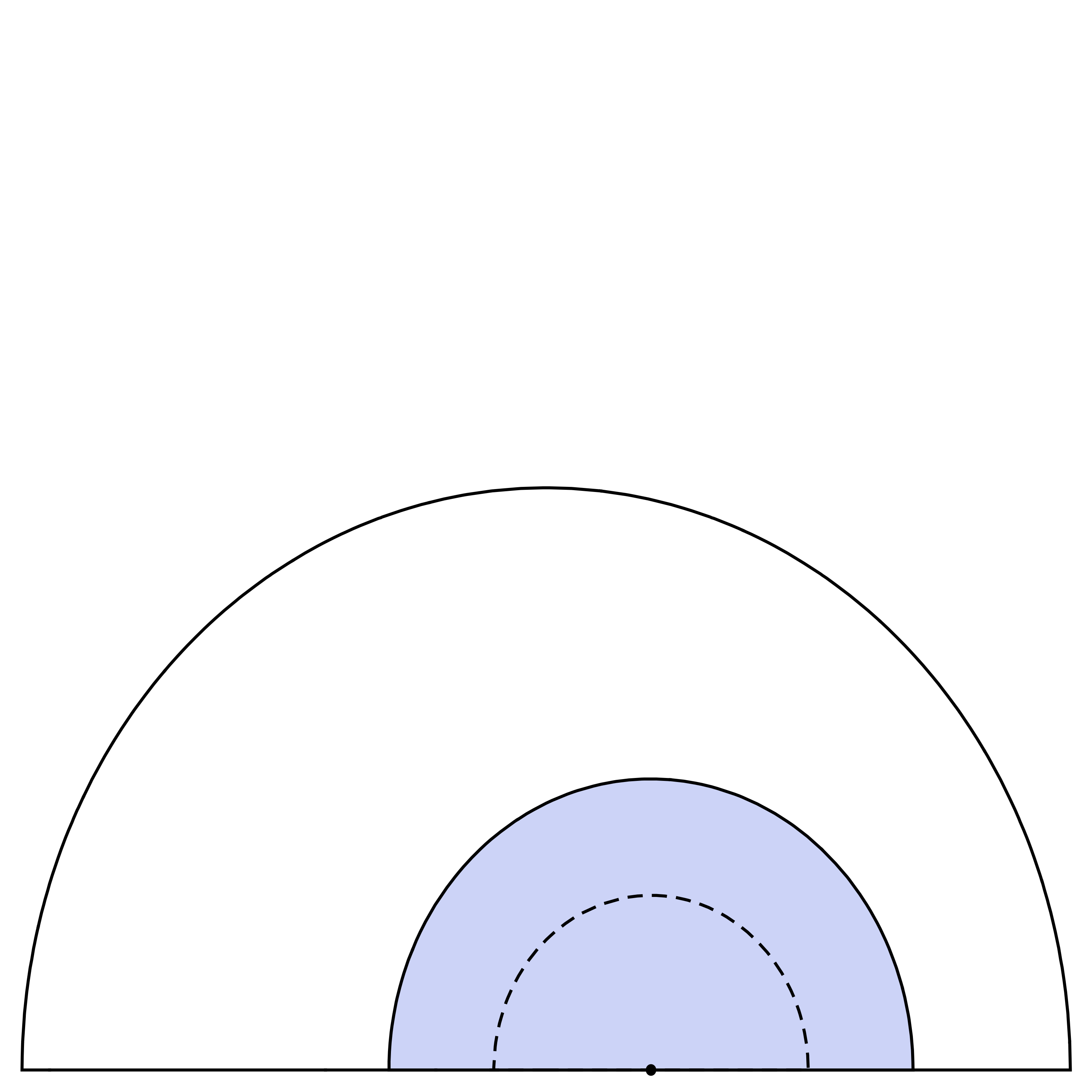}}
\put(12,60){$\scriptstyle \sO$}
\put(35,7){$\scriptstyle \partial_0\sO$}
\put(65,94){$\scriptstyle \partial_1\sO$}
\put(120,61){$\scriptstyle B_{r_0}^+(x^0)$}
\put(115,40){$\scriptstyle B_r^+(x^0)$}
\put(110,7){$\scriptstyle x^0$}
\end{picture}
\caption{Boundaries and regions in Theorem \ref{thm:APrioriSchauderInteriorDomain} and Remark \ref{rmk:APrioriInteriorLocalSchauder}.}
\label{fig:higher_order_heston_regularity_regions}
\end{figure}
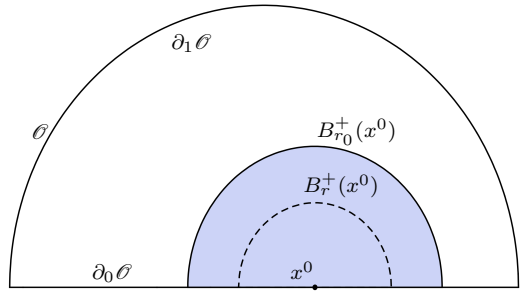

Our first main result is the following analogue of \cite[Theorem I.1.3]{DaskalHamilton1998} (for a related boundary-degenerate parabolic operator \eqref{eq:DHModel} and $d=2$), \cite[Theorem 3.1]{Daskalopoulos_Rhee_2003} (for a related boundary-degenerate parabolic operator \eqref{eq:DHModel} with $d\geq 2$), and \cite[Corollary 6.3 and Problem 6.1]{GilbargTrudinger} (strictly elliptic operator). We refer the reader to Definitions \ref{defn:Calphas}, \ref{defn:DHspaces}, and \ref{defn:DH2spaces} for descriptions of the Daskalopoulos-Hamilton family of $C^{k,\alpha}_s$ and $C^{k,2+\alpha}_s$ H\"older norms and Banach spaces. For any $U\subseteqq \HH$, we denote
\begin{equation}
\label{eq:CoefficientNorms}
\|a\|_{C^{k, \alpha}_s(\bar U)} := \sum_{i,j=1}^d\|a^{ij}\|_{C^{k, \alpha}_s(\bar U)}
\quad\hbox{and}\quad
\|b\|_{C^{k, \alpha}_s(\bar U)} := \sum_{i=1}^d\|b^i\|_{C^{k, \alpha}_s(\bar U)}.
\end{equation}
We then have the

\begin{thm}[A priori interior Schauder estimate]
\label{thm:APrioriSchauderInteriorDomain}
For any $\alpha\in(0,1)$, integer $k\geq 0$, and positive constants $b_0$, $d_0$, $\lambda_0$, $\Lambda$, $\nu$, there is a positive constant, $C=C(\alpha, b_0,d,d_0,k,\lambda_0,\Lambda,\nu)$, such that the following holds. Let $\sO \subset \HH$ be an open subset
with\footnote{If one allows $\height(\sO)=\infty$, one may need to modify our definition of H\"older norms to provide a weight for additional control when $x_d\to\infty$ because the coefficient matrix, $x_da$, for $D^2u$ would be unbounded due to \eqref{eq:Strict_ellipticityDomain}. Weighted H\"older norms of this type were used by the authors in
\cite{Feehan_Pop_mimickingdegen_pde}, for this reason, for the corresponding parabolic operator, $-\partial_t+A$, on $(0,T)\times\HH$.}
$\height(\sO)\leq \nu$ and the coefficients $a,b,c$ of $A$ in \eqref{eq:defnA} belong to $C^{k, \alpha}_s(\underline\sO)$ and obey
\begin{gather}
\label{eq:Coeff_Holder_continuityEstimateDomain}
\|a\|_{C^{k, \alpha}_s(\bar\sO)} + \|b\|_{C^{k, \alpha}_s(\bar\sO)} + \|c\|_{C^{k, \alpha}_s(\bar\sO)} \leq \Lambda,
\\
\label{eq:Strict_ellipticityDomain}
\langle a\xi, \xi\rangle \geq \lambda_0 |\xi|^2\quad\hbox{on } \sO,
\quad \forall\,\xi \in \RR^d,
\\
\label{eq:Coeff_b_dDomain}
b^d \geq b_0 \quad \hbox{on } \partial_0 \sO.
\end{gather}
If $u \in C^{k,2+\alpha}_s(\underline\sO)$ and $\sO'\subset\sO$ is an open subset such that $\dist(\partial_1\sO',\partial_1\sO)\geq d_0$, then
\begin{equation}
\label{eq:APrioriSchauderInteriorDomain}
\|u\|_{C^{k, 2+\alpha}_s(\bar\sO')} \leq C\left(\|A u\|_{C^{k, \alpha}_s(\bar\sO)} + \|u\|_{C(\bar\sO)}\right).
\end{equation}
\end{thm}

\begin{rmk}[A priori interior Schauder estimate]
\label{rmk:APrioriInteriorLocalSchauder}
The case where $k=0$ and the open subset is a half-ball, $\sO=B^+_{r_0}(x^0)$, the coefficients, $a,b$, of $A$ are constant, $c=0$, and $u \in C^\infty(\bar B^+_{r_0}(x^0))$ is given by Corollary \ref{cor:Holder_estimate_any_r_any_gamma_A_0}.
Theorem \ref{thm:Holder_estimate_local} relaxes those conditions to allow $u \in C^{2+\alpha}_s(\underline B^+_{r_0}(x^0))$ and arbitrary $c \in \RR$;
Theorem \ref{thm:Holder_estimate_local_variable_coeff} further relaxes the conditions on $A$ to allow for variable coefficients, $a,b,c$, in $C^\alpha_s(\underline B^+_{r_0}(x^0))$; Theorem \ref{thm:Higher_order_estimate} relaxes the constraint $k=0$ to allow for arbitrary integers $k\geq 0$; finally, Theorem \ref{thm:APrioriSchauderInteriorDomain} is proved in Section \ref{subsec:Variable_coefficients}, where we relax the constraint that $\sO = B^+_{r_0}(x^0)$ and allow for arbitrary open subsets of the form $\sO\subseteqq\RR^{d-1}\times(0,\nu)$.
\end{rmk}

It is considerably more difficult to prove a global a priori estimate for a solution, $u\in C^{k, 2+\alpha}_s(\bar\sO)$, when the intersection, $\overline{\partial_0\sO}\cap\overline{\partial_1\sO}$, is non-empty and we do not consider that problem in this article, but refer the reader to \cite[\S 1.3]{Feehan_Pop_higherregularityweaksoln} for a discussion of this issue. However, the global estimate in Corollary \ref{cor:Global_Schauder_estimate_VariableCoefficients} has useful applications when $\partial_1\sO$ does not meet $\partial_0\sO$. For a constant $\nu>0$, we define the horizontal \emph{slab} (or strip, in dimension two),
\begin{equation}
\label{eq:Slab}
S:=\RR^{d-1}\times(0,\nu),
\end{equation}
and note that $\partial_0 S =\RR^{d-1}\times\{0\}$ and $\partial_1 S =\RR^{d-1}\times\{\nu\}$.

\begin{cor}[A priori global Schauder estimate on a slab]
\label{cor:Global_Schauder_estimate_VariableCoefficients}
For any $\alpha\in(0,1)$, positive constants $b_0$, $\lambda_0$, $\Lambda$, $\nu$, and integer $k\geq 0$, there is a positive constant, $C=C(\alpha, b_0, d, k, \lambda_0,\Lambda,\nu)$, such that the following holds. Suppose the coefficients of $A$ in \eqref{eq:defnA} belong to $C^{k,\alpha}_s(\bar S)$, where $S=\RR^{d-1}\times(0, \nu)$ as in
\eqref{eq:Slab}, and obey
\begin{gather}
\label{eq:Coeff_Holder_continuity_HigherOrderSlab}
\|a\|_{C^{k, \alpha}_s(\bar S)} + \|b\|_{C^{k, \alpha}_s(\bar S)} + \|c\|_{C^{\alpha}_s(\bar S)} \leq \Lambda,
\\
\label{eq:Strict_ellipticity_slab}
\langle a\xi, \xi\rangle \geq \lambda_0 |\xi|^2\quad\hbox{on } S,\quad \forall\,\xi \in \RR^d,
\\
\label{eq:Coeff_b_d_slab}
b^d \geq b_0 \quad \hbox{on } \partial_0 S.
\end{gather}
If $u\in C^{k, 2+\alpha}_s(\bar S)$ and $u = 0$ on $\partial_1 S$, then
\begin{equation}
\label{eq:Global_Schauder_estimate_VariableCoefficients}
\| u\|_{C^{k, 2+\alpha}_s(\bar S)} \leq C\left(\|A u\|_{C^{k, \alpha}_s(\bar S)} + \|u\|_{C(\bar S)}\right),
\end{equation}
and, when $c\geq 0$ on $S$,
\begin{equation}
\label{eq:Global_Schauder_estimate_VariableCoefficients_nonnegc}
\| u\|_{C^{k, 2+\alpha}_s(\bar S)} \leq C \|A u\|_{C^{k, \alpha}_s(\bar S)},
\end{equation}
\end{cor}

\begin{rmk}[A priori global Schauder estimate on a slab]
For an operator, $A$, with constant coefficients, $a,b,c$, an a priori global Schauder estimate on a slab is proved as Corollary \ref{cor:Global_Schauder_estimate_ConstantCoefficients}.
\end{rmk}

The Green's function for an operator $A$ in \eqref{eq:defnA} with constant coefficients can be extracted from Appendix \ref{sec:Existence_smooth_solution}, where we construct explicit $C^\infty$ solutions to $Au=f$ on $\HH$ and
prove the following elliptic analogue of the existence result \cite[Theorem I.1.2]{DaskalHamilton1998} for the initial value problem for a boundary-degenerate parabolic model \eqref{eq:DHModel} on a half-space for the linearization of the porous medium equation \eqref{eq:PorousMediumEquation}.

\begin{thm}[Existence and uniqueness of a $C^\infty(\bar\HH)$ solution on the half-space when $A$ has constant coefficients]
\label{thm:Existence_smooth_solutions_half_space}
Let $A$ be an operator of the form \eqref{eq:defnA} and require that the coefficients, $a,b,c$, are constant with $b^d>0$ and $c>0$. If $f\in C^{\infty}_0(\bar\HH)$, then there is a unique solution, $u \in C^{\infty}(\bar\HH)$, to $Au=f$ on $\HH$.
\end{thm}

Again, it is considerably more difficult to prove existence of a solution, $u$, in $C^{k, 2+\alpha}_s(\bar\sO)$ or $C^{k, 2+\alpha}_s(\underline\sO)\cap C(\bar\sO)$, to \eqref{eq:IntroBoundaryValueProblem}, \eqref{eq:IntroBoundaryValueProblemBC} when the intersection, $\overline{\partial_0\sO}\cap\overline{\partial_1\sO}$, is non-empty. We do not consider that problem in this article either and again refer the reader to \cite[\S 1.3]{Feehan_Pop_higherregularityweaksoln} for a discussion of this issue. However, in the case of a slab, $\partial_1\sO$ does not meet $\partial_0\sO$ and we have an existence result, Theorem \ref{thm:Existence_uniqueness_slabs}, for an operator with variable coefficients. In Section \ref{subsec:Extensions}, we discuss additional existence results which should also follow from Theorems \ref{thm:APrioriSchauderInteriorDomain} and \ref{thm:Existence_smooth_solutions_half_space} when $\partial_0\sO$ is curved and $\partial_1\sO$ is empty.

\begin{thm}[Existence and uniqueness of a $C^{k, 2+\alpha}_s(\bar S)$ solution on a slab $S$]
\label{thm:Existence_uniqueness_slabs}
Let $\alpha\in(0,1)$, let $\nu>0$ and $S=\RR^{d-1}\times(0, \nu)$ be as in \eqref{eq:Slab}, and let $k\geq 0$ be an integer. Let $A$ be an operator as in \eqref{eq:defnA}. If $f$ and the coefficients of $A$ in \eqref{eq:defnA} belong to $C^{k,\alpha}_s(\bar S)$ and obey \eqref{eq:Strict_ellipticity_slab} and \eqref{eq:Coeff_b_d_slab} for some positive constants, $b_0, \lambda_0$, then there is a unique solution, $u \in C^{k, 2+\alpha}_s(\bar S)$, to the boundary value problem,
\begin{align}
\label{eq:Equation_slab}
Au &= f \quad \hbox{on } S,
\\
\label{eq:Equation_slabBC}
u &= 0 \quad \hbox{on } \partial_1 S.
\end{align}
\end{thm}

\begin{rmk}[Existence and uniqueness of a solution on a slab]
For an operator, $A$, with constant coefficients, $a,b,c$, existence and uniqueness of a solution on a slab is proved as Corollary \ref{cor:Existence_solutions_slab}.
\end{rmk}

The preceding existence and uniqueness result on a slab leads to the following analogue of \cite[Theorem 6.17]{GilbargTrudinger} and is proved in Section \ref{subsec:Regularity}.

\begin{thm}[Interior $C^{k,2+\alpha}_s$-regularity]
\label{thm:InteriorRegularityDomain}
For any $\alpha\in(0,1)$ and integer $k\geq 0$, the following holds. Let $\sO\subset\HH$ be an open subset. Assume that the coefficients of $A$ in \eqref{eq:defnA} belong to $C^{k, \alpha}_s(\underline\sO)$ and obey \eqref{eq:Strict_ellipticityDomain} and \eqref{eq:Coeff_b_dDomain} for some positive constants $b_0, \lambda_0$. If $u \in C^2(\sO)$ obeys \footnote{We write $Du, \ x_d D^2 u \in C(\underline\sO)$ as an abbreviation for $u_{x_i}, \ x_d u_{x_ix_j} \in C(\underline\sO)$, for $1 \leq i,j \leq d$ and write $x_dD^2u = 0$ on $\partial_0 \sO$ as an abbreviation for $\lim_{\HH\ni x\to x^0} x_d D^2 u(x) = 0$ for all $x^0 \in \partial_0\sO$.}
\begin{gather}
\label{eq:CondRegularityDomain}
u \in C^1(\underline\sO), \quad x_d D^2 u \in C(\underline\sO), \quad\hbox{and}\quad Au \in C^{k,\alpha}_s(\underline\sO),
\\
\label{eq:VentcelDomain}
x_d D^2 u = 0 \quad\hbox{on }\partial_0\sO,
\end{gather}
then $u \in C^{k,2+\alpha}_s(\underline\sO)$.
\end{thm}

Given Theorem \ref{thm:InteriorRegularityDomain}, one immediately obtains the following boundary-degenerate elliptic analogue of the $C^\infty$-regularity result for the boundary-degenerate linear parabolic model used in the study of the porous medium equation \cite[Theorem I.1.1]{DaskalHamilton1998}.

\begin{cor} [Interior $C^\infty$-regularity]
\label{cor:CinftyGlobal}
Let $\sO\subset\HH$ be an open subset. Assume that the coefficients of $A$ in \eqref{eq:defnA} belong to $C^\infty(\underline\sO)$ and obey \eqref{eq:Strict_ellipticityDomain} and \eqref{eq:Coeff_b_dDomain} for some positive constants $b_0, \lambda_0$. If $u \in C^2(\sO)$ obeys \eqref{eq:CondRegularityDomain} for every integer $k\geq 0$, so $Au \in C^\infty(\underline\sO)$, together with \eqref{eq:VentcelDomain}, then $u \in C^\infty(\underline\sO)$.
\end{cor}

\begin{rmk}[Regularity up to the `non-degenerate boundary']
Regarding the conclusion of Theorem \ref{thm:InteriorRegularityDomain}, standard elliptic regularity results for linear, second-order, strictly elliptic operators \cite[Theorems 6.19]{GilbargTrudinger} also imply, when $k\geq 0$, that $u \in  C^{k+2,\alpha}(\sO\cup\partial_1\sO)$ if $u$ solves \eqref{eq:IntroBoundaryValueProblem}, \eqref{eq:IntroBoundaryValueProblemBC} with $f \in C^{k,\alpha}(\sO\cup\partial_1\sO)$ and $g \in C^{k+2,\alpha}(\sO\cup\partial_1\sO)$, and $\partial_1\sO$ is $C^{k+2,\alpha}$. Because our focus in this article is on regularity of $u$ up to the `degenerate boundary', $\partial_0\sO$, we shall omit further mention of such straightforward generalizations.
\end{rmk}

Finally, we refine our existence results in \cite{Feehan_Pop_higherregularityweaksoln} when $d=2$ for the \emph{Heston} operator,
\begin{equation}
\label{eq:HestonPDE}
Av := -\frac{x_2}{2}\left(v_{x_1x_1} + 2\varrho\sigma v_{x_1x_2} + \sigma^2 v_{x_2x_2}\right)
- \left(c_0-q-\frac{x_2}{2}\right)v_{x_1} - \kappa(\theta-x_2)v_{x_2} + c_0v,
\end{equation}
where $q\in\RR$, $c_0\geq 0$, $\kappa>0$, $\theta>0$, $\sigma\neq 0$, and $\varrho\in (-1,1)$ are constants (their financial interpretation is provided in \cite{Heston1993}), and $v\in C^\infty(\HH)$. In particular, we give analogues of the existence results \cite[Theorems 6.13 and 6.19]{GilbargTrudinger} for the case of the Dirichlet boundary value problem for a strictly elliptic operator.

\begin{thm} [Existence and uniqueness of a $C^{k,2+\alpha}_s$ solution to a partial Dirichlet boundary value problem for the Heston operator]
\label{thm:ExistUniqueCk2+alphasHolderContinuityDomain}
Let $\alpha\in(0,1)$ and let $k\geq 0$ be an integer, let $K$ be a finite right-circular cone, and require that $\partial_1\sO$ obeys a uniform exterior cone condition with cone $K$. If $f \in C^{k,\alpha}_s(\underline\sO)\cap C_b(\sO)$ and
\begin{equation}
\label{eq:cNonnegativeOrPositiveDomain}
\begin{cases} c_0 > 0 &\hbox{if } \height(\sO) = \infty, \\ c_0 \geq 0 &\hbox{if } \height(\sO) < \infty, \end{cases}
\end{equation}
then there is a unique solution,
$$
u \in C^{k,2+\alpha}_s(\underline\sO)\cap C_b(\sO\cup\partial_1\sO),
$$
to the boundary value problem for the Heston operator,
\begin{align}
\label{eq:HestonBoundaryValueProblem}
Au &= f \quad\hbox{on }\sO,
\\
\label{eq:HestonBoundaryValueProblemBC}
u &= 0 \quad\hbox{on }\partial_1\sO.
\end{align}
\end{thm}

\begin{rmk}[Schauder a priori estimates and approach to existence of solutions]
As we explain in \cite[\S 1.3]{Feehan_Pop_higherregularityweaksoln}, the proof of existence of solutions, $u \in C^{k,2+\alpha}_s(\bar\sO)$, to the boundary value problem, \eqref{eq:IntroBoundaryValueProblem}, \eqref{eq:IntroBoundaryValueProblemBC}, given $f\in C^{k,\alpha}_s(\underline\sO)$ and $g \in C(\bar\sO)$, appears considerably more difficult when $\overline{\partial_0\sO}\cap\overline{\partial_1\sO}$ is non-empty because, unlike in \cite{DaskalHamilton1998}, one must consider a priori Schauder estimates and regularity near the `corner' points of the open subset, $\sO\subset\HH$, where the `non-degenerate boundary', $\partial_1\sO$, meets the `degenerate boundary', $\partial_0\sO$.
\end{rmk}

Given an additional geometric hypothesis on $\sO$ near points in $\overline{\partial_0\sO}\cap\overline{\partial_1\sO}$, the property that $u \in C_b(\sO\cup\partial_1\sO)$ in the conclusion of Theorem \ref{thm:ExistUniqueCk2+alphasHolderContinuityDomain} simplifies to $u \in C(\bar\sO)$.

\begin{cor} [Existence and uniqueness of a globally continuous $C^{k,2+\alpha}_s$ solution to a partial Dirichlet boundary value problem for the Heston operator]
\label{cor:ExistUniqueCk2+alphasHolderContinuityDomain}
If in addition to the hypotheses of Theorem \ref{thm:ExistUniqueCk2+alphasHolderContinuityDomain} the open subset, $\sO$, satisfies a uniform exterior and interior cone condition on $\overline{\partial_0\sO}\cap\overline{\partial_1\sO}$ with cone $K$ in the sense of \cite{Feehan_Pop_higherregularityweaksoln}, then $u \in C^{k,2+\alpha}_s(\underline\sO) \cap C(\bar\sO)$.
\end{cor}

\begin{rmk}[Existence of solutions to a partial Dirichlet boundary value problem]
\label{rmk:GeneralizeExistencefromHeston}
By applying the results of this article, Theorem \ref{thm:ExistUniqueCk2+alphasHolderContinuityDomain} is generalized by the first author in \cite{Feehan_classical_perron_elliptic} from the case of the Heston operator $A$ in \eqref{eq:HestonPDE} to an operator $A$ in \eqref{eq:defnA} with $C^{k,2+\alpha}_s$ coefficients and $d\geq 2$. We expect that a similar generalization of Corollary \ref{cor:ExistUniqueCk2+alphasHolderContinuityDomain} should also hold. See Section \ref{subsec:Extensions} for further discussion.
\end{rmk}

\subsection{Connections with previous research}
\label{subsec:Survey}
We provide a brief survey of some related research by other authors on Schauder a priori estimates and regularity theory for solutions to boundary-degenerate elliptic and parabolic partial differential equations most closely related to the results described in our article.

The principal features which distinguish the boundary value problem \eqref{eq:IntroBoundaryValueProblem}, \eqref{eq:IntroBoundaryValueProblemBC}, when the operator $A$ is given by \eqref{eq:defnA}, from the boundary value problems for linear, second-order, strictly elliptic operators in \cite{GilbargTrudinger}, are the degeneracy of $A$ due to the factor, $x_d$, in the coefficient matrix for $D^2u$ and, because $b_0>0$ in \eqref{eq:defnA}, the fact that boundary conditions may be omitted along $x_d=0$ when we seek solutions, $u$, with sufficient regularity up to $x_d=0$.

The literature on degenerate elliptic and parabolic equations is vast, with the well-known articles of E. B. Fabes, C. E. Kenig, and R. P. Serapioni \cite{Fabes_1982, Fabes_Kenig_Serapioni_1982a}, G. Fichera \cite{Fichera_1956, Fichera_1960}, J. J. Kohn and L. Nirenberg \cite{Kohn_Nirenberg_1967}, M. K. V. Murthy and G. Stampacchia \cite{Murthy_Stampacchia_1968, Murthy_Stampacchia_1968corr} and the monographs of S. Z. Levendorski{\u\i} \cite{LevendorskiDegenElliptic} and O. A. Ole{\u\i}nik and E. V. Radkevi{\v{c}} \cite{Oleinik_Radkevic, Radkevich_2009a, Radkevich_2009b}, being merely the tip of the iceberg.

As far as the authors can tell, however, there has been relatively little prior work on a priori Schauder estimates and higher-order H\"older regularity of solutions up to the portion of the domain boundary where the operator becomes degenerate. In this context, the work of P. Daskalopoulos, R. Hamilton, and E. Rhee \cite{DaskalHamilton1998, Daskalopoulos_Rhee_2003, RheeThesis} and of H. Koch \cite{Koch} stands out because of their introduction of the cycloidal metric on the upper-half space, weighted H\"older norms, and weighted Sobolev norms which provide the key ingredients required to unlock the existence, uniqueness, and higher-order regularity theory for solutions to the porous medium equation \eqref{eq:PorousMediumEquation} and the boundary-degenerate parabolic model equation \eqref{eq:DHModel} on the upper half-space given by the linearization of the porous medium equation in suitable coordinates.

Daskalopoulos and Hamilton \cite{DaskalHamilton1998} proved existence and uniqueness of $C^\infty$ solutions, $u$, to the Cauchy problem for the porous medium equation \cite[p. 899]{DaskalHamilton1998} (when $d=2$),
\begin{equation}
\label{eq:PorousMediumEquation}
-u_t + \sum_{i=1}^d (u^m)_{x_ix_i} = 0 \quad\hbox{on }(0,T)\times\RR^d, \quad u(\cdot, 0) = g \quad\hbox{on }\RR^d,
\end{equation}
with constant $m>1$ and initial data, $g\geq 0$, compactly supported in $\RR^d$, together with $C^\infty$-regularity of its free boundary, $\partial\{u>0\}$, provided the initial pressure function is non-degenerate (that is, $D u^{m-1}  \geq a >0$) on boundary of its support at $t=0$. Their analysis is based on their development of existence, uniqueness, and regularity results for the linearization of the porous medium equation near the free boundary and, in particular, their model linear boundary-degenerate operator \cite[p. 901]{DaskalHamilton1998} (generalized from $d=2$ in their article),
\begin{equation}
\label{eq:DHModel}
Au = -x_d\sum_{i=1}^d u_{x_ix_i} - \beta u_{x_d}, \quad u\in C^\infty(\HH),
\end{equation}
where $\beta$ is a positive constant, analogous to the combination of parameters, $2\kappa\theta/\sigma^2$, in \eqref{eq:HestonPDE}, following a suitable change of coordinates \cite[p. 941]{DaskalHamilton1998}.

The same model linear boundary-degenerate operator (for $d\geq 2$), was studied independently by Koch \cite[Equation (4.43)]{Koch} and, in a Habilitation thesis, he obtained existence, uniqueness, and regularity results for solutions to \eqref{eq:PorousMediumEquation} which complement those of Daskalopoulos and Hamilton \cite{DaskalHamilton1998}. Koch employs weighted Sobolev space methods, Moser iteration, and pointwise estimates for the fundamental solution. However, by adapting the approach of Daskalopoulos and Hamilton \cite{DaskalHamilton1998}, we avoid having to rely on difficult pointwise estimates for the fundamental solution for the operator $A$ in \eqref{eq:defnA}. Although tantalizingly explicit ---
see \cite{DuffiePanSingleton2000, Heston1993} for the fundamental solution of the parabolic Heston operator \eqref{eq:HestonPDE} and Appendix \ref{sec:Existence_smooth_solution} for its elliptic analogue --- these kernel functions appear quite intractable for the analysis required to emulate the role of potential theory for the Laplace operator in the traditional development of Schauder theory in \cite{GilbargTrudinger}.

While the Daskalopoulos-Hamilton Schauder theory for boundary-degenerate parabolic operators has been adopted so far by relatively few other researchers, it has also been employed by A. De Simone, L. Giacomelli, H. Kn\"upfer, and F. Otto in \cite{DeSimone_Knupfer_Otto_2006, Giacomelli_Knupfer_Otto_2008, Giacomelli_Knupfer_2010} and by C. L. Epstein and R. Mazzeo in
\cite{Epstein_Mazzeo_annmathstudies}.

\subsection{Extensions and future work}
\label{subsec:Extensions}
We defer to a subsequent article the development of a priori global Schauder $C^{k,2+\alpha}_s(\bar\sO)$ estimates, existence, and regularity theory for solutions, $u$, to the elliptic boundary value problem \eqref{eq:IntroBoundaryValueProblem}, \eqref{eq:IntroBoundaryValueProblemBC} when $f$ and the coefficients, $a,b,c$, of $A$ in \eqref{eq:defnA} belong to $C^{k,\alpha}_s(\bar\sO)$, the boundary data function, $g$, belongs to $C^{k,2+\alpha}_s(\bar\sO)$, and $\sO$ has boundary portion $\partial_1\sO$ of class $C^{k+2,\alpha}$ and $C^{k,2+\alpha}$-transverse to $\partial_0\sO$. For reasons we summarize in \cite[\S 1.3]{Feehan_Pop_higherregularityweaksoln}, the development of global Schauder a priori estimates, regularity, and existence theory appears very difficult when the intersection $\overline{\partial_0\sO}\cap\overline{\partial_1\sO}$ is non-empty.

However, if $\sO\subset\RR^d$ is a bounded open subset and $A$ is an elliptic, linear, second-order partial differential operator which is equivalent to an operator $A^{x_0} $ of the form \eqref{eq:defnA} in local coordinates near \emph{every} point $x^0\in \partial\sO$, then Theorem \ref{thm:APrioriSchauderInteriorDomain} will quickly lead to a \emph{global} $C^{k, 2+\alpha}_s(\bar\sO)$ a priori estimate for $u$ if $\partial\sO=\partial_0\sO$ is of class $C^{k, 2+\alpha}_s$. Moreover,
the method of the proof of \cite[Theorem 6.5.3]{Krylov_LecturesHolder} (or indeed \cite[Theorem II.1.1]{DaskalHamilton1998}) should adapt to give existence of a solution, $u\in C^{k, 2+\alpha}_s(\bar\sO)$, to \eqref{eq:IntroBoundaryValueProblem}, \eqref{eq:IntroBoundaryValueProblemBC}.

In \cite{Feehan_classical_perron_elliptic}, the first author applied Theorems \ref{thm:APrioriSchauderInteriorDomain} and \ref{thm:Existence_uniqueness_slabs} and Corollary \ref{cor:Global_Schauder_estimate_VariableCoefficients} to prove existence of solutions to boundary value problems and obstacle problems for boundary-degenerate elliptic, linear, second-order partial differential operators of the form $A$ in \eqref{eq:defnA} with $C^{k,\alpha}_s(\underline\sO)$ coefficients, $a,b,c$,  on open subsets of the half-space and partial Dirichlet boundary conditions. Those existence results are based on new versions of the classical Perron method \cite[Sections 2.8 and 6.3]{GilbargTrudinger}. Applications of that work include a generalization (see Remark \ref{rmk:GeneralizeExistencefromHeston}) of the existence result in Theorem \ref{thm:ExistUniqueCk2+alphasHolderContinuityDomain}, based purely on the Schauder methods developed in this article and the Perron method in \cite{Feehan_classical_perron_elliptic}, rather than a combination of the variational methods employed in \cite{Daskalopoulos_Feehan_statvarineqheston, Feehan_Pop_regularityweaksoln, Feehan_Pop_higherregularityweaksoln} and the Schauder regularity theory developed in this article.

We expect the a priori interior Schauder estimates that we develop in this article, which are in the style of \cite[Corollary 6.3]{GilbargTrudinger}, to extend to more refined and sharper a priori `global' interior Schauder estimates, in the style of \cite[Theorem 6.2, Lemmas 6.20 and 6.21]{GilbargTrudinger}. Aside from facilitating `rearrangement arguments', we expect such `global' a priori interior Schauder estimates --- relying on a choice of suitable weighted H\"older spaces similar to those employed in \cite[Chapter 6]{GilbargTrudinger} --- to permit the use of the continuity method to prove existence of solutions within a self-contained Schauder framework and this theme will be developed by the authors in a subsequent article.

While our a priori Schauder estimates rely on the specific form of the degeneracy factor, $x_d$, of the operator $A$ in \eqref{eq:defnA} on an open subset of the half-space, we obtained weak and strong maximum principles for a much broader class of boundary-degenerate operators in \cite{Feehan_maximumprinciple}. We hope to extend the a priori Schauder estimates and regularity theory for boundary-degenerate elliptic operators such as
$$
Av =  -\vartheta\tr(aD^2v) - b\cdot Dv + cv \quad\hbox{on }\sO, \quad v\in C^\infty(\sO),
$$
where $\vartheta \in C^\alpha_{\loc}(\bar\sO)$ and $\vartheta>0$ on an open subset $\sO\subset\RR^d$ with non-empty boundary portion $\partial_0\sO = \Int(\{x\in\partial\sO:\vartheta(x)=0\})$.

\subsection{Outline and mathematical highlights of the article}
\label{subsec:Guide}
For the convenience of the reader, we provide a brief outline of the article. In Section \ref{sec:Preliminaries} , we review the construction of the Daskalopoulos-Hamilton-H\"older families of norms and Banach spaces \cite{DaskalHamilton1998}.

In Section \ref{sec:Derivative_estimates_interior_boundary}, we derive a priori local $C^0$ estimates for derivatives of solutions, $u$, to $Au=0$ on half-balls, $B_{r_0}^+(x_0)\subset\HH$, centered at points $x^0 \in \partial\HH$, when $A$ has constant coefficients. However, our method of proof differs significantly from that of Daskalopoulos and Hamilton \cite{DaskalHamilton1998}, who apply a comparison principle for a certain non-linear parabolic operator and which directly uses the fact that this operator is parabolic. We were not able to replace their `parabolic' comparison argument by one which is suitable for the elliptic operators we consider in this article. Instead, we employ a simpler approach using a version of A. Brandt's finite-difference method \cite{Brandt_1969a} to estimate derivatives in directions parallel to $\partial\HH$ and methods of ordinary differential equations to estimate derivatives in the direction orthogonal to $\partial\HH$.

In Section \ref{sec:PolynomialApproximation}, we adapt and slightly streamline the arguments of Daskalopoulos and Hamilton in \cite{DaskalHamilton1998} for their model boundary-degenerate parabolic operator \eqref{eq:DHModel} to the case of our boundary-degenerate elliptic operator \eqref{eq:defnA} and derive a $C^0$ a priori estimate of the remainder of the first-order Taylor polynomial of a function, $u$, on a half-ball, $B_{r_0}^+(x_0)$.

In Section \ref{sec:Schauder_estimates_interior}, we obtain a priori local interior Schauder estimates for a function, $u$, on a ball $B_{r_0}(x^0)\Subset\HH$, where we keep track of the distance between the ball center, $x^0\in\HH$, and the half-space boundary,
$\partial\HH$, again when $A$ has constant coefficients.

In Section \ref{sec:Boundary_Schauder_estimates}, we apply the results of the previous sections to prove our main $C^{2+\alpha}_s$ a priori interior local Schauder estimate (Theorem \ref{thm:Holder_estimate_local}) for an operator $A$ with constant coefficients on a half-ball, $B_{r_0}^+(x_0)$.

In Section \ref{sec:Holder_estimate_local_higherorder}, we prove a $C^{k, 2+\alpha}_s$ a priori interior local Schauder estimate (Theorem \ref{thm:Holder_estimate_local_higherorder}) and a global a priori global Schauder estimate on a slab (Corollary \ref{cor:Global_Schauder_estimate_ConstantCoefficients}), both when $A$ has constant coefficients.

In Section \ref{sec:Variable_coefficients_Higher-order_regularity}, we relax the assumption in the preceding sections that the coefficients of the operator $A$ in \eqref{eq:defnA} are constant and prove a $C^{2+\alpha}_s$ a priori interior local Schauder estimate (Theorem \ref{thm:Holder_estimate_local_variable_coeff}) for a function, $u$, on a half-ball, $B_{r_0}^+(x^0)$ when $A$ has variable coefficients. We then prove a $C^{k, 2+\alpha}_s$ a priori local interior Schauder estimate for arbitrary $k\in \NN$ (Theorem \ref{thm:Higher_order_estimate}) and complete the proofs of Theorem \ref{thm:APrioriSchauderInteriorDomain} and Corollary \ref{cor:Global_Schauder_estimate_VariableCoefficients}. Next, we prove our global $C^{k, 2+\alpha}_s(\bar S)$ existence result on slabs, Theorem \ref{thm:Existence_uniqueness_slabs}, and complete the proofs of our main $C^{k, 2+\alpha}_s$ regularity result, Theorem \ref{thm:InteriorRegularityDomain},
and the $C^{k, 2+\alpha}_s(\underline \sO)$ existence results, Theorem \ref{thm:ExistUniqueCk2+alphasHolderContinuityDomain} and Corollary \ref{cor:ExistUniqueCk2+alphasHolderContinuityDomain}, for solutions to a partial Dirichlet boundary value problem for the Heston operator.

We collect some additional useful results and their proofs in several appendices to this article. In Appendix \ref{sec:MaximumPrinciple}, we prove a weak maximum principle for operators which include those of the form $A$ in \eqref{eq:defnA} with $c \geq 0$ (rather than $c\geq c_0$ for a positive constant $c_0$) when the open subset, $\sO$, is unbounded but has finite height, extending one of the weak maximum principles in \cite{Feehan_maximumprinciple}. In Appendix \ref{sec:Existence_smooth_solution}, we prove Theorem \ref{thm:Existence_smooth_solutions_half_space}. In Appendix \ref{sec:InterpolationInequalities}, we summarize the interpolation inequalities and boundary properties of functions in weighted H{\"o}lder spaces proved in \cite{DaskalHamilton1998} and \cite{Feehan_Pop_mimickingdegen_pde}.

\subsection{Notation and conventions}
\label{subsec:Notation}
In the definition and naming of function spaces, including spaces of continuous functions and H\"older spaces, we follow R. A. Adams \cite{Adams_1975} and alert the reader to occasional differences in definitions between \cite{Adams_1975} and standard references such as D. Gilbarg and N. Trudinger \cite{GilbargTrudinger} or N. V. Krylov \cite{Krylov_LecturesHolder}.

We let $\NN:=\left\{0,1,2,3,\ldots\right\}$ denote the set of non-negative integers. If $U\subset\RR^d$ is any open subset, we let $\bar U$ denote its closure with respect to the Euclidean topology and let $\partial U := \bar U\less U$ denote its topological boundary. For $r>0$ and $x^0\in\RR^d$, we let $B_r(x^0) := \{x\in\RR^d: |x-x^0|<r\}$ denote the open ball with center $x^0$ and radius $r$. We denote $B_r^+(x^0) := B_r(x^0) \cap \HH$ when $x^0\in\partial\HH$. When $x^0$ is the origin, $O\in\RR^d$, we denote $B_r(x^0)$ and $B_r^+(x^0)$ by $B_r$ and $B_r^+$, respectively, for brevity.

If $V\subset U\subset \RR^d$ are open subsets, we write $V\Subset U$ when $U$ is bounded with closure $\bar U \subset V$. By $\supp\zeta$, for any $\zeta\in C(\RR^d)$, we mean the \emph{closure} in $\RR^d$ of the set of points where $\zeta\neq 0$.

We use $C=C(*,\ldots,*)$ to denote a constant which depends at most on the quantities appearing on the parentheses. In a given context, a constant denoted by $C$ may have different values depending on the same set of arguments and may increase from one inequality to the next.

\subsection{Acknowledgments} We are very grateful to the anonymous referee for a careful reading of our manuscript and kind comments. The first author would also like to thank the Max Planck Institut f\"ur Mathematik in der Naturwissenschaft, Leipzig, the Max Planck Institut f\"ur Mathematik, Bonn, and the Department of Mathematics at Columbia University for their generous support for research visits during 2012 and 2013.

\section{Preliminaries}
\label{sec:Preliminaries}
In this section, we review the construction of the Daskalopoulos-Hamilton-H\"older families of norms and Banach spaces \cite{DaskalHamilton1998}.

We first recall the definition of the \emph{cycloidal distance function}, $s(\cdot,\cdot)$, on $\bar\HH$ by
\begin{equation}
\label{eq:Cycloidal_distance}
s(x^1,x^2) := \frac{|x^1-x^2|}{\sqrt{x^1_d+x^2_d+|x^1-x^2|}},\quad\forall\, x^1, x^2 \in \bar\HH,
\end{equation}
where $x^i = (x^i_1,\ldots,x^i_d)$, for $i=1,2$, and $|x^1-x^2|$ denotes the usual Euclidean distance between points $x^1, x^2 \in \RR^d$. Analogues of the cycloidal distance function \eqref{eq:Cycloidal_distance} between points $(t^1,x^1), \ (t^2,x^2) \in [0,\infty)\times\bar\HH$, in the context of parabolic differential equations, were introduced by Daskalopoulos and Hamilton in \cite[p.~901]{DaskalHamilton1998} and Koch in \cite[p.~11]{Koch} for the study of the porous medium equation.

Observe that, by \eqref{eq:Cycloidal_distance},
\begin{equation}
\label{eq:CycloidLessEuclidDistance}
s(x,x^0) \leq |x - x^0|^{1/2}, \quad \forall\, x, x^0 \in \bar\HH.
\end{equation}
The reverse inequality take its simplest form when $x^0\in \partial\HH$, so $x^0_d=0$, in which case the inequalities $x_d\leq |x-x^0|$ and
$$
|x-x^0| = s(x,x^0)\sqrt{x_d + |x-x^0|} \leq s(x,x^0)\sqrt{2|x-x^0|},
$$
give
\begin{equation}
\label{eq:SimpleEuclidLessCycloidDistance}
|x-x^0| \leq 2s^2(x,x^0), \quad \forall\, x\in\bar\HH, \ x^0 \in \partial\HH.
\end{equation}
Following \cite[\S 1.26]{Adams_1975}, for an open subset $U\subset\HH$, we let $C(U)$ denote the vector space of continuous functions on $U$ and let $C(\bar U)$ denote the Banach space of functions in $C(U)$ which are bounded and uniformly continuous on $U$, and thus have unique bounded, continuous extensions to $\bar U$, with norm
$$
\|u\|_{C(\bar U)} := \sup_{U}|u|.
$$
Noting that $U$ may be \emph{unbounded}, we let $C_{\loc}(\bar U)$ denote the linear subspace of functions $u\in C(U)$ such that $u\in C(\bar V)$ for every precompact open subset $V\Subset \bar U$. We let $C_b(U) := C(U)\cap L^\infty(U)$.

Daskalopoulos and Hamilton provide the

\begin{defn}[$C^\alpha_s$ norm and Banach space]
\label{defn:Calphas}
\cite[p. 901]{DaskalHamilton1998}
Given $\alpha \in (0,1)$ and an open subset $U\subset\HH$, we say that $u\in C^\alpha_s(\bar U)$ if $u\in C(\bar U)$ and
$$
\|u\|_{C^\alpha_s(\bar U)} < \infty,
$$
where
\begin{equation}
\label{eq:CalphasNorm}
\|u\|_{C^\alpha_s(\bar U)} := [u]_{C^\alpha_s(\bar U)} + \|u\|_{C(\bar U)},
\end{equation}
and
\begin{equation}
\label{eq:CalphasSeminorm}
[u]_{C^\alpha_s(\bar U)} := \sup_{\begin{subarray}{c}x^1,x^2\in U \\ x^1\neq x^2\end{subarray}}\frac{|u(x^1)-u(x^2)|}{s^\alpha(x^1,x^2)}.
\end{equation}
We say that $u\in C^\alpha_s(\underline{U})$ if $u\in C^\alpha_s(\bar V)$ for all precompact open subsets $V\Subset \underline{U}$, recalling that $\underline{U} := U\cup\partial_0 U$. We let $C^\alpha_{s,\loc}(\bar U)$ denote the linear subspace of functions $u \in C^\alpha_s(U)$ such that $u\in C^\alpha_s(\bar V)$ for every precompact open subset $V\Subset \bar U$.
\end{defn}

It is known that $C^\alpha_s(\bar U)$ is a Banach space \cite[\S I.1]{DaskalHamilton1998} with respect to the norm \eqref{eq:CalphasNorm}.

We shall need the following higher-order weighted H\"older $C^{k,\alpha}_s$ and $C^{k,2+\alpha}_s$ norms and Banach spaces pioneered by Daskalopoulos and Hamilton \cite{DaskalHamilton1998}. We record their definition here for later reference.

\begin{defn}[$C^{k,\alpha}_s$ norms and Banach spaces]
\label{defn:DHspaces}
\cite[p. 902]{DaskalHamilton1998}
Given an integer $k\geq 0$, $\alpha \in (0,1)$, and an open subset $U\subset\HH$, we say that $u\in C^{k,\alpha}_s(\bar U)$ if
$u \in C^k(\bar U)$ and
$$
\|u\|_{C^{k,\alpha}_s(\bar U)} < \infty,
$$
where
\begin{equation}
\label{eq:CkalphasNorm}
\|u\|_{C^{k,\alpha}_s(\bar U)} := \sum_{|\beta|\leq k}\|D^\beta u\|_{C^\alpha_s(\bar U)},
\end{equation}
where $\beta := (\beta_1,\ldots,\beta_d) \in \NN^d$. When $k=0$, we denote $C^{0,\alpha}_s(\bar U) = C^\alpha_s(\bar U)$.
\end{defn}

\begin{defn}[$C^{k,2+\alpha}_s$ norms and Banach spaces]
\label{defn:DH2spaces}
\cite[pp. 901--902]{DaskalHamilton1998}
Given an integer $k\geq 0$, a constant $\alpha \in (0,1)$, and an open subset $U\subset\HH$, we say that $u \in C^{k,2+\alpha}_s(\bar U)$ if $u\in C^{k+1,\alpha}_s(\bar U)$, the derivatives, $D^\beta u$, $\beta\in\NN^d$ with $|\beta| = k+2$, of order $k+2$ are continuous on $U$, and the functions, $x_dD^\beta u$, $\beta\in\NN^d$ with $|\beta|= k+2$, extend continuously up to the boundary, $\partial U$, and those extensions belong to $C^\alpha_s(\bar U)$. We define
\begin{equation}
\label{eq:Ck+2alphasNorm}
\|u\|_{C^{k, 2+\alpha}_s(\bar U)} :=  \|u\|_{C^{k+1,\alpha}_s(\bar U)} + \sum_{|\beta|=k+2}\|x_dD^\beta u\|_{C^{\alpha}_s(\bar U)}.
\end{equation}
We say that\footnote{In \cite[pp. 901--902]{DaskalHamilton1998}, when defining the spaces $C^{k,\alpha}_s(\sA)$ and $C^{k,2+\alpha}_s(\sA)$, it is assumed that $\sA$ is a compact subset of the \emph{closed} upper half-space, $\bar\HH$.} $u\in C^{k,2+\alpha}_s(\underline{U})$ if $u\in C^{k,2+\alpha}_s(\bar V)$ for all precompact open subsets $V\Subset \underline{U}$. When $k=0$, we denote $C^{0,2+\alpha}_s(\bar U) = C^{2+\alpha}_s(\bar U)$.
\end{defn}

For any non-negative integer $k$, we let $C^k_0(\underline{U})$ denote the linear subspace of functions $u\in C^k(U)$ such that $u\in C^k(\bar V)$ for every precompact open subset $V\Subset \underline{U}$ and define $C^{\infty}_0(\underline{U}) := \cap_{k\geq 0} C^k_0(\underline{U})$. Note that we also have $C^{\infty}_0(\underline{U}) = \cap_{k\geq 0}C^{k,\alpha}_s(\underline{U}) = \cap_{k\geq 0}C^{k,2+\alpha}_s(\underline{U})$.

\section{Interior local estimates of derivatives}
\label{sec:Derivative_estimates_interior_boundary}
As in \cite{DaskalHamilton1998}, we begin with the derivation of local estimates of derivatives of solutions on half-balls, $B_r^+(x_0)$, centered at points $x^0 \in \partial\HH$, but the method of the proof differs significantly from the method of the proof in \cite[\S I.4 and I.5]{DaskalHamilton1998}. In \cite{DaskalHamilton1998}, Daskalopoulos and Hamilton apply a comparison principle to a suitably chosen function, defined in terms of the derivatives (see the definitions of $Y$ at the beginning of \cite[\S I.5]{DaskalHamilton1998} and of $X$ in the proof of \cite[Corollary I.5.3]{DaskalHamilton1998}). Their comparison principle directly uses the fact that the operator is parabolic, and we were not able to replace the `parabolic' comparison argument by one which is suitable for the elliptic operators we consider in this article. (The Daskalopoulos-Hamilton approach can be viewed as a variant of the Bernstein method --- see the proof
\cite[Theorem 8.4.4]{Krylov_LecturesHolder} in the case of the heat operator and \cite[Theorem 2.5.2]{Krylov_LecturesHolder} in the case of the Laplace operator.)

Instead, we apply a combination of finite-difference arguments, methods of ordinary differential equations, and, in this section, restrict to the homogeneous version of Eq. \eqref{eq:IntroBoundaryValueProblem} with $f=0$.
We adapt Brandt's finite-difference method \cite{Brandt_1969a} (see also \cite[\S 3.4]{GilbargTrudinger}) to obtain a priori local estimates for $D^{\beta} u$, where $\beta\in\NN^d$ is any multi-index with non-negative integer entries of the form $\beta=(\beta^1, \ldots,\beta^{d-1},0)$. The method of Brandt also uses a comparison principle, but it is applied to finite differences, instead of functions of derivatives of $u$, such as $X$ and $Y$ in \cite[\S I.5]{DaskalHamilton1998}.  Brandt's approach is mentioned by Gilbarg and Trudinger in \cite[p. 47]{GilbargTrudinger} as an alternative to the usual methods for proving a priori interior Schauder estimates such as \cite[Corollary 6.3]{GilbargTrudinger}. We are able to apply the finite-difference estimates method not only on balls $B_r(x^0) \Subset \HH$ as in \cite{Brandt_1969a},
but also on half-balls $B_r^+(x_0) \subset \HH$ centered at points $x^0 \in \partial\HH$ because the degeneracy of the elliptic operator $A$ in \eqref{eq:defnA} along $\partial_0 B_r^+(x_0)$ and the fact that $b^d>0$ along $\partial_0 B_r^+(x_0)$
(see \eqref{eq:bd}) implies that no boundary condition need be imposed along $\partial_0 B_r^+(x_0)$.

In Section \ref{subsec:BasicResultsConstantCoefficients} we summarize the interior local Schauder estimate and regularity results we will prove in Sections \ref{sec:Derivative_estimates_interior_boundary}, \ref{sec:PolynomialApproximation}, \ref{sec:Schauder_estimates_interior}, and \ref{sec:Boundary_Schauder_estimates}. In Section \ref{subsec:Derivative_estimates_tangential}, we develop $C^0$ interior local estimates for derivatives $D^\beta u$ when $\beta_d=0$ and in Section \ref{subsec:Derivative_estimates_normal}, we extend those estimates to case $\beta_d>0$.

\subsection{A priori interior local Schauder estimate and regularity statements in the case of constant coefficients}
\label{subsec:BasicResultsConstantCoefficients}
Throughout Sections \ref{sec:Derivative_estimates_interior_boundary}--\ref{sec:Holder_estimate_local_higherorder}, we further assume the

\begin{hyp}[Constant coefficients and positivity]
\label{hyp:ConstantCoefficients}
The coefficients, $a,b,c$, of the operator $A$ in \eqref{eq:defnA} are constant; there is a positive constant, $\lambda_0$, such that
\footnote{Condition \eqref{eq:Strict_ellipticity_ConstantCoefficents} is first used in the proof of Lemma
\ref{lem:Estimate_v_x}.}
\begin{equation}
\label{eq:Strict_ellipticity_ConstantCoefficents}
\langle a\xi, \xi \rangle \geq \lambda_0|\xi|^2, \quad\forall\, \xi \in \RR^d;
\end{equation}
and
\footnote{Condition \eqref{eq:bd} is required by our weak maximum principle (Lemma \ref{lem:Comparison_principle_A_0} and Corollary \ref{cor:Maximum_principle_A_0}). Our weak maximum principle is in turn required in Section \ref{sec:Derivative_estimates_interior_boundary}; sections \ref{sec:PolynomialApproximation} and \ref{sec:Schauder_estimates_interior} depend on Section \ref{sec:Derivative_estimates_interior_boundary}; and sections \ref{sec:Boundary_Schauder_estimates}, \ref{sec:Holder_estimate_local_higherorder}, and \ref{sec:Variable_coefficients_Higher-order_regularity} each depend on sections \ref{sec:Schauder_estimates_interior} and \ref{sec:Boundary_Schauder_estimates}.}
\begin{equation}
\label{eq:bd}
b^d = b_0 > 0.
\end{equation}
\end{hyp}

The condition \eqref{eq:bd} is first required in the proof of Lemma \ref{lem:Interior_estimates_Holder_seminorms_Du_yD2u}.
When the coefficients of $A$ are constant, we denote
\begin{equation}
\label{eq:LambdaConstantCoefficients}
\Lambda = \sum_{i,j=1}^d|a^{ij}| + \sum_{i=1}^d|b^i| + |c|.
\end{equation}
Our main goal in sections \ref{sec:Derivative_estimates_interior_boundary}, \ref{sec:PolynomialApproximation}, \ref{sec:Schauder_estimates_interior}, and \ref{sec:Boundary_Schauder_estimates} is to prove the following version of Theorems \ref{thm:APrioriSchauderInteriorDomain} and \ref{thm:InteriorRegularityDomain} when $k=0$ and $A$ has constant coefficients and the open subset, $\sO$, is a half-ball, $B_{r_0}^+(x^0)$ with $x^0\in\partial\HH$.

\begin{thm}[A priori interior local Schauder estimate when $A$ has constant coefficients]
\label{thm:Holder_estimate_local}
Assume that $A$ in \eqref{eq:defnA} obeys Hypothesis \ref{hyp:ConstantCoefficients}.
For any $\alpha\in(0,1)$ and constants $r$ and $r_0$ with $0<r<r_0$, there is a positive constant,
$C=C(\alpha, b_0, d,\lambda_0,\Lambda, r_0,r)$, such that the following holds. If $x^0\in\partial\HH$ and $u\in C^{2+\alpha}_s(\underline B^+_{r_0}(x^0))$, then
\begin{equation}
\label{eq:Holder_estimate_local}
\|u\|_{C^{2+\alpha}_s(\bar B^+_r)} \leq C\left(\|A u\|_{C^{\alpha}_s(\bar B^+_{r_0}(x^0))} + \|u\|_{C(\bar B^+_{r_0}(x^0))}\right).
\end{equation}
\end{thm}

Our goal in the remainder of this section is to derive a priori estimates for $Du$ and $x_dD^2u$ on half-balls, $B_r^+(x^0)$, centered at points $x^0 \in \partial\HH$. Because our operator, $A$, is invariant with respect to translations in the variables $(x_1,\ldots,x_{d-1})$ when the coefficients, $a,b,c$, are constant, we can assume without loss of generality that $x^0$ is the origin, $O\in\RR^d$, and write $B_{r_0}^+(x^0)=B_{r_0}^+$ and $B_r^+(x^0)=B_r^+$ in our proof of Theorem \ref{thm:Holder_estimate_local}.

\subsection{Interior local estimates for derivatives in directions parallel to the degenerate boundary}
\label{subsec:Derivative_estimates_tangential}
To derive a priori local estimates for $D^{\beta}u$, for $\beta\in\NN^d$ with $\beta_d=0$, it will be useful to consider the following transformation,
\begin{equation}
\label{eq:Transformation_x_plus_alpha_y}
u(x) =: v(y),\quad x\in\HH,
\end{equation}
where $y = \varphi(x) := x + \xi x_d$ and $\xi=(\xi_1,\ldots,\xi_{d-1},\xi_d) \in \RR^d$. We choose $\xi$ such that
\begin{equation}
\label{eq:Definition_xi}
\xi_i : =-b^i/b^d,\quad \forall\, i\neq d,\quad \xi_d=0,
\end{equation}
where we have used assumption \eqref{eq:bd} that $b^d>0$. Note that $\varphi$ is a diffeomorphism on $\bar\HH$ which restricts to the identity map on $\partial\HH$. We now consider the operator $\widetilde A_0$ defined by
$$
A_0 u (x) =: \widetilde A_0 v (y), \quad x\in \HH,
$$
and by direct calculations we obtain that
\begin{equation}
\label{eq:defnWidetildeA_0}
\widetilde A_0 v = -y_d  \tilde a^{ij} v_{y_iy_j} - \tilde b^i v_{y_i}\quad\hbox{on }\HH,
\end{equation}
where
\begin{equation}
\begin{aligned}
\label{eq:defnCoeffWidetildeA_0}
\tilde a^{ij} &:= a^{ij}+\frac{1}{2}\left(\xi_j a^{id}+\xi_i a^{jd}\right) +\xi_i\xi_j a^{dd}, \quad\forall\, i,j \neq d, \\
\tilde a^{id} = \tilde a^{di}
              &:= a^{id} + \xi_i a^{dd},\quad\forall\, i \neq d,\\
\tilde a^{dd} &:= a^{dd},\\
\tilde b^i    &:= b^i + \xi_i b^d, \quad\forall\, i \neq d,\\
\tilde b^d    &= b^d.
\end{aligned}
\end{equation}
The purpose of the transformation \eqref{eq:Transformation_x_plus_alpha_y} is to ensure that the coefficients $\tilde b^i$ of the partial derivatives with respect to $y_i$ in the definition \eqref{eq:defnWidetildeA_0} of the operator $\widetilde A_0$ are zero when $i\neq d$. The matrix $\tilde a$ is symmetric and positive definite, but now the constant of strict ellipticity depends on $b^i/b^d$, that is, on $b^d$ and $\Lambda$, and on the constant of strict ellipticity, $\lambda_0$, of the matrix $a$.

\begin{lem}[Local estimates for first-order derivatives of $v$ parallel to $\partial\HH$]
\label{lem:Estimate_v_x}
Assume that $A$ in \eqref{eq:defnA} obeys Hypothesis \ref{hyp:ConstantCoefficients}.
Let $0<r<r_0$, and let $v \in C^2(B^+_{r_0}) \cap C(\bar B^+_{r_0})$ obey
\begin{equation*}
\widetilde A_0 v =0 \quad \hbox{on } B^+_{r_0},
\end{equation*}
and assume that $v$ satisfies
\begin{equation}
\label{eq:Condition_estimate_v_x}
Dv,\ y_d D^2v \in C(\underline B^+_{r_0}) \quad\hbox{and}\quad y_d D^2v =0 \quad\hbox{on } \partial_0 B^+_{r_0}.
\end{equation}
Then there is a positive constant, $C=C( b_0, d, \lambda_0, \Lambda, r_0, r)$, such that
\begin{equation*}
\|v_{y_k}\|_{C(\bar B^+_{r})} \leq C \|v\|_{C(\bar B^+_{r_0})}, \quad\forall\, k \neq d.
\end{equation*}
\end{lem}

\begin{proof}
We adapt the finite-difference argument employed by Brandt in \cite{Brandt_1969a} to prove the local estimates for derivatives, $v_{y_k}$, when $k\neq d$. We let $r_2:=(r+{r_0})/2$ and $r_3:=\min\{({r_0}-r)/2, \ 1/2\}$, and consider the $(d+1)$-dimensional cylinder,
$$
\cC := \left\{(y,y_{d+1}) \in \HH\times\RR_+: y \in B^+_{r_2},\ 0< y_{d+1} < r_3\right\}.
$$
We consider the auxiliary function,
\begin{equation*}
\phi(y,y_{d+1}) := \frac{1}{2}\left(v(y+y_{d+1}e_{k}) - v(y-y_{d+1}e_{k})\right), \quad\forall\, (y,y_{d+1}) \in \cC,
\end{equation*}
where $\cC$ is defined above, and $e_k \in \RR^{d}$ is the vector whose coordinates are all zero except for the $k$-th coordinate, which is $1$.
We choose a constant $c_0>0$ small enough, say $c_0=\lambda_0/2$, such that the differential operator,
\[
\widetilde A^1_0 := \widetilde A_0 - c_0 y_d \frac{\partial^2}{\partial y_k^2}
+ c_0y_d\frac{\partial^2}{\partial y_{d+1}^2},
\]
is elliptic on $\HH\times\RR_+$. By the definition of the function $\phi$, we notice that
\begin{equation*}
\widetilde A^1_0 \phi = 0 \quad \hbox{on } \cC,
\end{equation*}
because $\widetilde A_0 v = 0$ on $B^+_{r_0}$. For $y^0 \in \bar B^+_{r}$, we consider the auxiliary function defined on $\cC$,
\begin{equation*}
\psi := C_1 \|v\|_{C(\bar B^+_{r_0})} \left[y_{d+1}(1-y_{d+1}) + C_2 \left(\sum_{i=1}^{d-1}(y_i-y^0_i)^2 + y_d^2(y_d-y^0_d) + y_{d+1}^2 \right)\right],
\end{equation*}
where the positive constants $C_1$, $C_2$ will be suitably chosen below. We want to choose $C_2$ sufficiently small that
\[
\widetilde A^1_0 \psi \geq 0 \quad \hbox{on } \cC.
\]
By direct calculation, we obtain
\begin{align*}
\psi_{y_i} &= 2 C_1 C_2 \|v\|_{C(\bar B^+_{r_0})}(y_i-y_i^0),\quad i=1,2,\ldots,d-1,\\
\psi_{y_iy_i} &= 2 C_1 C_2 \|v\|_{C(\bar B^+_{r_0})},\quad i=1,2,\ldots,d-1,\\
\psi_{y_{d+1}y_{d+1}} &= 2 C_1 \left(C_2-1\right) \|v\|_{C(\bar B^+_{r_0})},\\
\psi_{y_d} &= 2 x_d C_1 C_2 \|v\|_{C(\bar B^+_{r_0})} \left(\frac{3}{2}y_d-y_d^0\right),\\
\psi_{y_d y_d} &= 2C_1 C_2 \|v\|_{C(\bar B^+_{r_0})} \left(3y_d-y_d^0\right),
\end{align*}
and so,
\begin{align*}
\widetilde A^1_0 \psi
&= -y_d  \tilde a^{ij} v_{y_iy_j} - \tilde b^i v_{y_i} - c_0 y_d \left(\psi_{y_{d+1}y_{d+1}} - \psi_{y_ky_k}\right)\\
&= -2y_dC_1 \|v\|_{C(\bar B^+_{r_0})}
\left[C_2\left( \left(\sum_{i=1}^{d-1}\tilde a^{ii} \right)+ \tilde a^{dd}\left(3y_d-y_d^0\right) -c_0 +\tilde b^d\left(\frac{3}{2} y_d-y_d^0\right) \right) + c_0(C_2-1)\right]\\
&\geq -2y_d C_1 \|v\|_{C(\bar B^+_{r_0})}
\left[C_2\left(\left(\sum_{i=1}^{d-1}\tilde a^{ii}\right) + 3r\tilde a^{dd}+ 2r \tilde b^d\right) - c_0\right]
\quad\hbox{on }\cC,
\end{align*}
using the facts that the $\tilde a^{ii}$, for $i=1,\ldots,d$, and $\tilde b^d$ are positive constants, while $\tilde b^i=0$, $i\neq d$, by the transformation \eqref{eq:Transformation_x_plus_alpha_y}, and $y_d <r$. We choose the constant $C_2=C_2(b_0,d,\lambda_0,\Lambda,r_0)$ such that
$$
C_2 \leq c_0\left(\sum_{i=1}^{d-1}\tilde a^{ii} + 3r\tilde a^{dd}+ 2r \tilde b^d\right)^{-1},
$$
so that we have
\begin{align*}
\widetilde A^1_0 \psi \geq 0 \quad \hbox{on } \cC.
\end{align*}
Because $\widetilde A^1_0 \phi = 0$ on $\cC$, the preceding inequality yields
\[
\widetilde A^1_0 \left(\pm \phi - \psi\right) \leq 0 \quad \hbox{on } \cC.
\]
By the definition of the auxiliary function, $\psi$, and using the fact that $y^0\in \bar B^+_{r}$ and $0<y_{d+1}<1/2$, we may choose a positive constant, $C_1=C_1(C_2, r_0, r)$, large enough that
\begin{equation}
\label{eq:Comparison_boundary}
\pm \phi - \psi \leq 0 \quad \hbox{on }\partial_1 \cC.
\end{equation}
The portion $\partial_1 \cC$ of the boundary of $\cC$ consists of the sets
\begin{align*}
\{y_{d+1} =0, \ y \in B^+_{r_2}\}, \quad \{y_{d+1} = r_3,\ y \in B^+_{r_2}\}, \quad
\hbox{and}\quad  \{y_{d+1} \in (0, r_3), \ y \in \partial_1 B^+_{r_2}\}.
\end{align*}
To establish inequality \eqref{eq:Comparison_boundary} along the portion $\{y_{d+1}=0\}$ of the boundary, $\partial_1 \cC$, note that $\phi=0$, and so \eqref{eq:Comparison_boundary} holds on this portion of the boundary since $\psi \geq 0$. For the second portion of the boundary, $\partial_1 \cC$, using the fact that $r_3 \leq 1/2$, we notice that
\begin{align*}
y_{d+1}(1-y_{d+1}) + C_2 \left(\sum_{i=1}^{d-1}(y_i-y^0_i)^2 + y_d^2(y_d-y^0_d) + y_{d+1}^2 \right)
\geq \frac{r_3}{2} \quad \hbox{on }  \{y_{d+1} = r_3,\ y \in B^+_{r_2}\}.
\end{align*}
For the third portion of the boundary, using the fact that $y^0 \in B^+_{r}$ and $y \in B^+_{r_2}$ and $r<r_2$, we see that on $\{y_{d+1} \in (0, r_3), \ y \in \partial_1 B^+_{r_2}\}$ we have
\begin{align*}
y_{d+1}(1-y_{d+1}) + C_2 \left(\sum_{i=1}^{d-1}(y_i-y^0_i)^2 + y_d^2(y_d-y^0_d) + y_{d+1}^2 \right) \geq C_2 (d-1) (r_2-r)^2.
\end{align*}
Therefore, we can find a constant $C_3=C_3(C_2,r_0,r)$ such that
$$
\psi \geq C_1 C_3 \|v\|_{C(B^+_{r})}\quad \hbox{on }  \{y_{d+1} = r_3,\ y \in B^+_{r_2}\} \cup \{y_{d+1} \in (0, r_3), \ y \in \partial_1 B^+_{r_2}\}
$$
We may choose the constant $C_1=C_1(C_3, r_0, r)$ large enough so that $C_1C_3 \geq 1$, and using the definition of $\phi$, we have
$$
\psi \geq |\phi|\quad \hbox{on }  \{y_{d+1} = r_3,\ y \in B^+_{r_2}\} \cup \{y_{d+1} \in (0, r_3), \ y \in \partial_1 B^+_{r_2}\}.
$$
Now, inequality \eqref{eq:Comparison_boundary} follows. By \eqref{eq:Condition_estimate_v_x} we have $\phi \in C(\bar \cC)$, and $D \phi$, $y_d D^2 \phi \in C(\cC \cup \partial_0 \cC)$, and $y_d D^2 \phi =0$ on $\partial_0 \cC$, where $\partial_0 \cC$ is the interior of $\{y_d=0\} \cap \bar\cC$. Since $\psi \in C^{\infty}(\bar\cC)$, we may apply the comparison principle \cite[Theorem 5.1]{Feehan_maximumprinciple} to $\phi$ and $\psi$ on the open subset, $\cC$. We find that $\pm \phi - \psi \leq 0$ on $\cC$, and so by the definition of the function $\phi$, we have, for all $y^0 \in B^+_{r}$ and $y_{d+1} \in (0,r_3)$,
\[
\frac{1}{2y_{d+1}} |v(y^0+y_{d+1} e_k)-v(y_0-y_{d+1} e_k)| \leq C_1 \|v\|_{C(\bar B^+_{r_0})} \left(1-y_{d+1} + C_2 y_{d+1}\right).
\]
The preceding inequality yields
\[
|v_{y_k}(y^0)| \leq C_1 \|v\|_{C(\bar B^+_{r_0})},\quad\forall\, y^0 \in \bar B^+_{r},
\]
for a constant $C_1=C_1(b_0, d,\lambda_0,\Lambda, r_0, r)$, and this concludes the proof.
\end{proof}

\begin{lem}[Local estimates for higher-order derivatives of $v$ parallel to $\partial\HH$]
\label{lem:Estimate_v_k_x}
Assume that $A$ in \eqref{eq:defnA} obeys Hypothesis \ref{hyp:ConstantCoefficients}.
Let $k \in \NN$ and $0<r<{r_0}$. Then there is a constant, $C=C(b_0, d, k, \lambda_0, \Lambda, r_0, r)$, such that for any $v \in C^{\infty}(\bar B^+_{r_0})$ obeying
$$
\widetilde A_0 v = 0 \quad\hbox{on}\quad  B^+_{r_0},
$$
we have
\begin{equation}
\label{eq:Estimate_v_k_x}
\|D^{\beta} v\|_{C(\bar B^+_{r})} \leq C \|v\|_{C(\bar B^+_{r_0})},
\end{equation}
for all multi-indices $\beta=(\beta_1, \ldots, \beta_{d-1},0) \in \NN^d$ such that $|\beta| \leq k$.
\end{lem}

\begin{proof}
Lemma \ref{lem:Estimate_v_x} establishes the result when $|\beta|=1$. We prove the higher-order derivative estimates parallel to $\partial\HH$ by induction. We assume the induction hypothesis: For any $0<r<{r_0}$, there is a constant,
$C_1=C_1(b_0, d, k-1, \lambda_0, \Lambda, r_0, r)$, such that
$$
\|D^{\beta'} v\|_{C(\bar B^+_{r})} \leq C_1 \|v\|_{C(\bar B^+_{r_0})},
$$
for all multi-indices $\beta'=(\beta'_1, \ldots, \beta'_{d-1},0) \in \NN^d$ such that $|\beta'| \leq k-1$.
Since $\widetilde A_0 v = 0$ on $B^+_{r_0}$, we also have that $\widetilde A_0 D^{\beta} v = 0$ on $B^+_{r_0}$, for all multi-indices $\beta$ with $\beta_d=0$. We fix such a multi-index $\beta$. Let $k\in\NN$ be such that $\beta_k\neq 0$, and set $\beta':=\beta-e_k$. We set $r_2:=(r+{r_0})/2$ and apply Lemma \ref{lem:Estimate_v_x} to $D^{\beta'} v$ with $0<r<r_2$ to obtain
\[
\|D^{\beta} v\|_{C(\bar B^+_{r})} \leq C_2 \|D^{\beta'} v\|_{C(\bar B^+_{r_2})},
\]
for some positive constant $C_2=C_2(b_0, d, \lambda_0,\Lambda, r,r_2)$. The conclusion now follows from the preceding estimate and the induction hypothesis applied to $D^{\beta'} v$ with $0<r_2<{r_0}$, since $|\beta'| \leq k-1$.
\end{proof}

From \eqref{eq:Transformation_x_plus_alpha_y}, we have
\begin{equation*}
D^{\beta} u(x) = D^{\beta} v(y),\quad y = x + \xi x_d, \quad x\in\HH,
\end{equation*}
for all $\beta \in \NN^d$ such that $\beta_d=0$. Therefore, Lemmas \ref{lem:Estimate_v_x} and \ref{lem:Estimate_v_k_x} give us the following estimates for $D^{\beta} u$.

\begin{lem}[Local estimates of higher-order derivatives of $u$ parallel to $\partial\HH$]
\label{lem:Estimate_u_k_x}
Assume that $A$ in \eqref{eq:defnA} obeys Hypothesis \ref{hyp:ConstantCoefficients}.
Let $k\in\NN$ and ${r_0}>0$. Then there are positive constants,
$r_1=r_1(b_0,\Lambda, r_0)<r_0$ and $C=C(b_0,d,k,\lambda_0,\Lambda, r_0)$, such that for any function $u\in C^{\infty}(\bar B^+_{r_0})$ solving
\begin{equation}
\label{eq:A_0_of_u_equal_0}
 A_0 u =0 \quad \hbox{on } B^+_{r_0},
\end{equation}
we have, for all $\beta \in \NN^d$ with $\beta_d=0$ and $|\beta| \leq k$,
\begin{equation*}
\|D^{\beta} u\|_{C(\bar B^+_{r_1})} \leq C \|u\|_{C(\bar B^+_{r_0})}.
\end{equation*}
\end{lem}

\begin{proof}
Let $\varphi:\HH\rightarrow\HH$ be the affine transformation defined by $y=\varphi(x):=x+\xi x_d$, 
for $x\in\HH$, where $\xi\in\RR^{d}$ is defined by  \eqref{eq:Definition_xi}. Let $s_0=s_0(b_0,\Lambda,r_0)>0$ be small enough such that $B^+_{s_0}\subset\varphi(B^+_{r_0})$. Then, $v \in C^{\infty}(\bar B^+_{s_0})$ and $\widetilde A_0 v=0$ on $B^+_{s_0}$, since $u \in C^{\infty}(\bar B^+_{r_0})$, and $A_0 u=0$ on $B^+_{r_0}$. Let $s_1={s_0}/2$ and apply Lemma \ref{lem:Estimate_v_k_x} to $v$ with $r$ replaced by $s_1$ and $r_0$ replaced by $s_0$. For any $k \in \NN$, there is a positive constant, $C=C(b_0, d,k,\lambda_0,\Lambda, r_0)$, such that for all $\beta \in \NN^d$ with $\beta_d=0$, we have
\begin{equation}
\label{eq:Estimate_derivative_beta_v}
\|D^{\beta}v\|_{C(\bar B^+_{s_1})} \leq C \|v\|_{C(\bar B^+_{s_0})}.
\end{equation}
We now choose $r_1=r_1(b_0,\Lambda, s_1)$ small enough such that $\varphi(B^+_{r_1})\subset B^+_{s_1}$. Using the fact that $D^{\beta} u(x) = D^{\beta} v(\varphi(x))$, we obtain
\begin{align*}
\|D^{\beta} u\|_{C(\bar B^+_{r_1})}
&\leq \|D^{\beta} v\|_{C(\bar B^+_{s_1})}\quad\hbox{(by the facts that $\varphi(B^+_{r_1})\subset B^+_{s_1}$ and $u(x)=v(\varphi(x))$)}
\\
&\leq C \|v\|_{C(\bar B^+_{s_0})}\quad \hbox{(by \eqref{eq:Estimate_derivative_beta_v})}
\\
&\leq C \|u\|_{C(\bar B^+_{r_0})} \quad\hbox{(by the facts that $B^+_{s_0}\subset\varphi(B^+_{r_0})$ and $u(x)=v(\varphi(x))$)}.
\end{align*}
This concludes the proof.
\end{proof}

\subsection{Interior local estimates for derivatives in the direction orthogonal to the degenerate boundary}
\label{subsec:Derivative_estimates_normal}
We again shall use the affine transformation \eqref{eq:Transformation_x_plus_alpha_y} of coordinates, but now with a different choice of the vector $\xi$, that is
\begin{equation}
\label{eq:Definition_xi_1}
\xi^i := -a^{id}/ a^{dd},\quad\forall\, i\neq d,\quad \xi_d=0,
\end{equation}
and, given a function $u$ on $\HH$, we define the function $w$ by
\begin{equation}
\label{eq:Definition_w}
u(x) =: w(y), \quad y=x+\xi x_d, \quad x\in\HH.
\end{equation}
Then, by analogy with \eqref{eq:defnWidetildeA_0}, we obtain
\begin{equation*}
\bar A_0 w := y_d \bar a^{ij} w_{y_iy_j} + \bar b^i w_{y_i} \quad\hbox{on }\HH,
\end{equation*}
where we notice that $\bar a^{id} =0$ by the choice of the vector $\xi$. Also, we have that $D^{\beta} u(x) = D^{\beta} w(y)$, for all $\beta\in\NN^d$ with $\beta_d=0$. Thus, Lemma \ref{lem:Estimate_u_k_x} applies to $w$, and we obtain a priori local estimates for all derivatives of $w$ parallel to $\partial\HH$.

Next, we derive an a priori local estimate for $w_{y_d}$.

\begin{lem}[Local estimate for $w_{y_d}$]
\label{lem:Estimate_w_y}
Assume that $A$ in \eqref{eq:defnA} obeys Hypothesis \ref{hyp:ConstantCoefficients}.
Let $0<r<{r_0}$. Then there is a positive constant, $C=C(b_0,d,\lambda_0,\Lambda,r_0,r)$, such that for any function $w \in C^{\infty}(\bar B^+_{r_0})$ obeying
\begin{equation}
\label{eq:A_0_w}
\bar A_0 w =0 \quad \hbox{on } B^+_{r_0},
\end{equation}
we have
\begin{equation*}
\|w_{y_d}\|_{C(\bar B^+_{r})} \leq C \|w\|_{C(\bar B^+_{r_0})}.
\end{equation*}
\end{lem}

\begin{proof}
Because $\bar a^{id}=0$, for all $i\neq d$, we can rewrite the equation $\bar A_0 w =0$ on $B^+_{r_0}$ as
\begin{equation*}
y_d w_{y_dy_d} + \theta w_{y_d} = f \quad\hbox{on } B^+_{r_0},
\end{equation*}
where, for simplicity, we denote $\theta:=\bar b^d/\bar a^{dd}>0$, and define $f$ by
\begin{equation*}
f := y_d \sum_{i,j=1}^{d-1}\frac{\bar a^{ij}}{\bar a^{dd}} w_{y_iy_j} + \sum_{i=1}^{d-1}\frac{\bar b^i}{\bar a^{dd}} w_{y_i}\quad\hbox{on } \bar B^+_{r_0}.
\end{equation*}
We can estimate $\|f\|_{C(\bar B^+_{r})}$ in terms of $\|w\|_{C(\bar B^+_{r_0})}$ by applying Lemma \ref{lem:Estimate_u_k_x} to control the supremum norms of $w_{y_i}$ and $w_{y_iy_j}$ on $\bar B^+_{r}$, for all $i,j\neq d$. The preceding ordinary differential equation can be rewritten as
$$
\left(y_d^{\theta} w_{y_d}\right)_{y_d} = y_d^{\theta-1} f\quad\hbox{on }  B^+_{r_0},
$$
and, integrating with respect to $y_d$, we obtain
\[
y_d^{\theta} w_{y_d}(y) = \int_0^{y_d} f(y',s) s^{\theta-1} \,ds,\quad y\in   B^+_{r_0},
\]
where we denote $y=(y',y_d)$, and use the facts that $\theta>0$ and $w_{y_d}\in C(\bar B^+_{r_0})$. Thus, we have
\[
|y_d^{\theta} w_{y_d}(y)| \leq \|f(y',\cdot)\|_{C([0,y_d])} \int_0^{y_d}  s^{\theta-1} \,ds = \frac{1}{\theta} y_d^{\theta} \|f(y',\cdot)\|_{C([0,y_d])}, \quad y\in   B^+_{r_0},
\]
from where it follows, by the definition of $f$, that
\[
|w_{y_d}(y)| \leq C \sum_{ \substack{\beta\in\NN^d \\ \beta_d=0;\ |\beta|\leq 2}}\|D^{\beta}w(y',\cdot)\|_{C([0,y_d])},\quad y\in   B^+_{r_0},
\]
for some constant $C=C(b_0,\lambda_0,\Lambda)$. Now applying Lemma \ref{lem:Estimate_u_k_x} to estimate $D^{\beta}w$ on $B^+_{r}$, for all $0<r<{r_0}$, and for all $\beta \in \NN^d$ with $\beta_d=0$ and $|\beta|\leq 2$, we obtain the supremum estimate for $w_{y_d}$ on $\bar B^+_{r}$ in terms of the supremum estimate of $w$ on $\bar B^+_{r_0}$.
\end{proof}

\begin{lem}[Local estimates for $D^{\beta} D_{y_d} w$ with $\beta_d=0$]
\label{lem:Estimate_D_k_x_D_y_w}
Assume that $A$ in \eqref{eq:defnA} obeys Hypothesis \ref{hyp:ConstantCoefficients}.
Let $k \in \NN$, and let $0<r<{r_0}$. Then there is a constant, $C=C(b_0,d,k,\lambda_0,\Lambda, r_0,r)$, such that for any function $w \in C^{\infty}(\bar B^+_{r_0})$ obeying \eqref{eq:A_0_w} we have
\begin{equation*}
\|D^{\beta} D_{y_d} w\|_{C(\bar B^+_{r})} \leq C \|w\|_{C(\bar B^+_{r_0})},
\end{equation*}
for all $\beta \in \NN^d$ with $\beta_d=0$ and $|\beta| \leq k$.
\end{lem}

\begin{proof}
Since $\bar A_0 w = 0$ on $B^+_{r_0}$, we also have $\bar A_0 D^{\beta} w = 0$ on $B^+_{r_0}$, for all $\beta \in \NN^d$ with $\beta_d=0$. Lemma \ref{lem:Estimate_w_y} then applies with $r$ replaced by $r_2=(r+{r_0})$/2, and gives us
\[
\|D^{\beta} D_{y_d} w\|_{C(\bar B^+_{r})} \leq C_0 \|D^{\beta} w\|_{C(\bar B^+_{r_2})},
\]
where
$C_0=C_0(b_0,d,\lambda_0,\Lambda, r_0,r)$ is a positive constant. Next, we apply Lemma \ref{lem:Estimate_u_k_x} to estimate $D^{\beta} w$ and give a constant $C_1=C_1(b_0,d,k,\lambda_0,\Lambda, r_0,r_2)$ such that
\[
\|D^{\beta} w\|_{C(\bar B^+_{r_2})} \leq C_1 \|w\|_{C(\bar B^+_{r_0})}.
\]
Now combining the preceding two inequalities, we obtain the a priori local estimate for $D^{\beta}D_{y_d} w$.
\end{proof}

\begin{lem}[Local estimate for $w_{y_dy_d}$]
\label{lem:Estimate_w_yy}
Assume that $A$ in \eqref{eq:defnA} obeys Hypothesis \ref{hyp:ConstantCoefficients}.
Let $k \in \NN$ and $0<r<{r_0}$. Then there is a positive constant, $C=C(b_0,d,\lambda_0,\Lambda, r_0,r)$, such that for any function $w \in C^{\infty}(\bar B^+_{r_0})$ obeying \eqref{eq:A_0_w} we have
\begin{equation*}
\|w_{y_dy_d}\|_{C(\bar B^+_{r})} \leq C \|w\|_{C(\bar B^+_{r_0})}.
\end{equation*}
\end{lem}

\begin{proof}
By taking another derivative with respect to $y_d$ in the equation $\bar A_0 w=0$ on $B^+_{r_0}$, we see that $w_{y_d}$ is a solution to
\begin{align*}
y_d \sum_{i,j=1}^d \bar a^{ij} (w_{y_d})_{y_iy_j} + \sum_{i=1}^{d-1}\left(\bar b^i + 2\bar a^{id}\right) (w_{y_d})_{y_i}
+ \left(\bar b^d + \bar a^{dd}\right) (w_{y_d})_{y_d}
&= - \sum_{i,j=1}^{d-1}\bar a^{ij} w_{y_iy_j}.
\end{align*}
Applying the method of the proof of Lemma \ref{lem:Estimate_w_y} with $\theta:=\left(\bar b^d + \bar a^{dd}\right)/\bar a^{dd}$ and
\[
f:=- \sum_{i,j=1}^{d-1}\bar a^{ij} w_{y_iy_j} - y_d \sum_{i,j=1}^{d-1} \bar a^{ij} w_{y_dy_iy_j} - \sum_{i=1}^{d-1}\left(\bar b^i + 2\bar a^{id}\right) w_{y_dy_i},
\]
we obtain
\[
\|w_{y_dy_d}\|_{C(\bar B^+_{r})} \leq
C \sum_{  \substack{\beta\in\NN^d \\ \beta_d=0,1;\ |\beta|\leq 3} } \|D^{\beta} w\|_{C(\bar B^+_{r})},
\]
where $C=C(b_0,d, \lambda_0, \Lambda)$ is a positive constant. We can estimate the supremum norms of $D^{\beta} w$ on $B^+_{r}$ , for all $\beta \in \NN^d$ with $\beta_d=0,1$, in terms of the supremum norm of $w$ on $B^+_{r_0}$ with the aid of Lemmas \ref{lem:Estimate_u_k_x} and \ref{lem:Estimate_D_k_x_D_y_w}. Now, the supremum estimate for $w_{y_dy_d}$ on $B^+_{r}$ follows immediately.
\end{proof}

From the definition \eqref{eq:Definition_w} of $w$, using the fact that $\xi_d=0$, we have
\begin{equation}
\label{eq:Relation_u_w}
\begin{aligned}
u_{x_d}(x)    &= \sum_{k=1}^{d-1}\xi_k w_{y_k}(y) + w_{y_d}(y),\\
u_{x_i}(x)    &= u_{y_i}(y),\quad\forall\, i \neq d,\\
u_{x_dx_d}(x) &= \sum_{k,l=1}^{d-1}\xi_k \xi_l w_{y_k y_l}(y) + 2 \sum_{k=1}^{d-1}\xi_k w_{y_k y_d}(y) + w_{y_dy_d}(y),\\
u_{x_ix_d}(x) &= \sum_{k=1}^{d-1}\xi_k w_{y_iy_k}(y) + w_{y_iy_d}(y),\quad \forall\, i \neq d,\\
u_{x_ix_j}(x) &=  w_{y_iy_j}(y),\quad \forall\, i,j \neq d,
\end{aligned}
\end{equation}
for $x\in\HH$.  Using the preceding identities together with the estimates of Lemmas \ref{lem:Estimate_w_y}, \ref{lem:Estimate_D_k_x_D_y_w} and \ref{lem:Estimate_w_yy}, we obtain

\begin{lem}[Local estimates for second-order derivatives of $u$]
\label{lem:Estimate_y_derivatives_u}
Assume that $A$ in \eqref{eq:defnA} obeys Hypothesis \ref{hyp:ConstantCoefficients}.
Let ${r_0}>0$. Then there are positive constants, $r_1=r_1(b_0,\lambda_0,\Lambda, r_0)<r_0$ and $C=C(b_0,d,\lambda_0,\Lambda, r_0)$, such that for all $u \in C^{\infty}(\bar B^+_{r_0})$ obeying \eqref{eq:A_0_of_u_equal_0}, we have
\begin{align*}
\|D^{\beta} u\|_{C(\bar B^+_{r_1})} &\leq C \|u\|_{C(\bar B^+_{r_0})},
\end{align*}
for all $\beta \in \NN^d$ with $|\beta| \leq 2$.
\end{lem}

\begin{proof}
Let $\varphi:\HH\rightarrow\HH$ be the affine transformation defined by $y=\varphi(x):=x+\xi x_d$, for $x\in\HH$, where $\xi\in\RR^{d}$ is defined by  \eqref{eq:Definition_xi_1}. Let $s_0=s_0(\lambda_0,\Lambda, r_0)>0$ be small enough such that $B^+_{s_0}\subset\varphi(B^+_{r_0})$. Let $s_1=s_1(b_0,\Lambda,s_0)<s_0$ denote the constant $r_1$ given by Lemma \ref{lem:Estimate_u_k_x} applied with $r_0$ replaced by $s_0$. Then, the function $w$ defined by \eqref{eq:Definition_w} has the property that $w \in C^{\infty}(\bar B^+_{s_0})$ and $\widetilde A_0 w=0$ on $B^+_{s_0}$, since $u \in C^{\infty}(\bar B^+_{r_0})$ and $A_0 u=0$ on $B^+_{r_0}$. We apply Lemma \ref{lem:Estimate_w_y}, if $\beta=e_d$, Lemma \ref{lem:Estimate_D_k_x_D_y_w}, if $\beta = e_i +e_d$ and $i\neq d$, and Lemma \ref{lem:Estimate_w_yy}, if $\beta = 2e_d$, to the function $w$ with $r$ replaced by $s_1$ and $r_0$ replaced by $s_0$. We apply Lemma \ref{lem:Estimate_u_k_x}, if $\beta=e_i$ or $\beta=e_i+e_j$, for all $i,j\neq d$, to the function $w$ with $r_1$ replaced by $s_1$ and $r_0$ replaced by $s_0$. Then, for any $k \in \NN$, there is a positive constant,
$C=C(b_0,d,k,\lambda_0,\Lambda, r_0)$, such that for all $\beta \in \NN^d$ with $|\beta| \leq 2$, we have
\begin{equation}
\label{eq:Estimate_derivative_beta_w}
\|D^{\beta}w\|_{C(\bar B^+_{s_1})} \leq C \|w\|_{C(\bar B^+_{s_0})}.
\end{equation}
We now choose $r_1=r_1(b_0,\lambda_0,\Lambda, r_0)$ small enough such that 
$\varphi(B^+_{r_1})\subset B^+_{s_1}$. Using \eqref{eq:Relation_u_w}, we obtain
\begin{align*}
\|D^{\beta} u\|_{C(\bar B^+_{r_1})}
&\leq \|D^{\beta} w\|_{C(\bar B^+_{s_1})}\quad\hbox{(by the facts that $\varphi(B^+_{r_1})\subset B^+_{s_1}$ and $u(x)=w(\varphi(x))$)}
\\
&\leq C \|w\|_{C(\bar B^+_{s_0})}\quad \hbox{(by \eqref{eq:Estimate_derivative_beta_w})}
\\
&\leq C \|u\|_{C(\bar B^+_{r_0})} \quad\hbox{(by the facts that $B^+_{s_0}\subset\varphi(B^+_{r_0})$ and $u(x)=w(\varphi(x))$)}.
\end{align*}
This concludes the proof.
\end{proof}

\section{Polynomial approximation and Taylor remainder estimates}
\label{sec:PolynomialApproximation}
We adapt and slightly streamline the arguments of Daskalopoulos and Hamilton in \cite[\S I.6 and I.7]{DaskalHamilton1998} for their model boundary-degenerate parabolic operators acting on functions $u(t,x)$, for $(t,x) \in \RR_+\times\RR^2$, to the case of our boundary-degenerate elliptic operators acting on functions $u(x)$, for
$x \in \RR^d$. The goal of this section is to derive an estimate of the remainder of the first-order Taylor polynomial of a function $u$ on half-balls centered at points in $\partial\HH$ (Corollary \ref{cor:Remainder_estimate_2}). This result, when combined with the interior Schauder estimates of section Section \ref{sec:Schauder_estimates_interior}, will lead to the full Schauder estimate for a solution on a half-ball centered at point in $\partial\HH$ (Theorem \ref{thm:Holder_estimate_local}). Throughout this section, we continue to assume Hypothesis \ref{hyp:ConstantCoefficients} and so the coefficients, $a,b,c$, of the operator $A$ in \eqref{eq:defnA} and the coefficients, $a,b$, of the operator $A_0$ in \eqref{eq:defnA_0} are constant.

We let $T^P_k v$ denote the \emph{Taylor polynomial of degree $k$} of a smooth function $v$, centered at a point $P \in \RR^d$, and let $R^P_k:=v-T^P_k$ denote the \emph{remainder}. We then have the following analogue of \cite[Theorem I.6.1]{DaskalHamilton1998}.

\begin{prop}[Polynomial approximation]
\label{prop:Polynomial_approximation}
Assume that $A$ in \eqref{eq:defnA} obeys Hypothesis \ref{hyp:ConstantCoefficients}.
There is a positive constant, $C=C(b_0,d,\lambda_0,\Lambda)$, such that for any $r_0>0$, and any function $u \in C^{\infty}(\bar B^+_{r_0})$, there is a polynomial $p$ of degree $1$, such that for any $r \in (0,r_0)$ we have
\begin{equation}
\label{eq:Estimate_with_polynomial}
\|u-p\|_{C(\bar B^+_r)}
\leq
C\left(\frac{r^2}{r_0^2} \|u\|_{C(\bar B^+_{r_0})} + r_0 \|A_0 u\|_{C(\bar B^+_{r_0})}\right).
\end{equation}
\end{prop}

\begin{proof} We first consider the case when $r_0=1$ and then the case when $r_0>0$ is arbitrary.
\begin{step}[$r_0=1$]
We let $f:=A_0 u$ and we choose a smooth, non-negative, cutoff function, $\psi$, such that
\[
\psi \restriction_{B^+_{1/2}} \equiv 1
\quad\hbox{and}\quad
\psi\restriction_{\HH\less B^+_1} \equiv 0.
\]
We fix a constant $\nu>1$, and let $S=\RR^{d-1}\times(0,\nu)$ as in \eqref{eq:Slab}. By Theorem \ref{thm:Existence_smooth_solutions_slab}, there is a unique solution, $u_1\in C^{\infty}(\bar S)$,  to
\begin{align*}
\begin{cases}
A_0 u_1 = \psi f&\quad\hbox{on } S,\\
u_1(\cdot, \nu) = 0 &\quad\hbox{on } \RR^{d-1}.
\end{cases}
\end{align*}
Then, by setting $u_2:=u-u_1$, we see that $u_2 \in C^{\infty}(\bar B^+_{r_0})$ and satisfies $A_0 u_2 = (1-\psi) f$ on $B^+_{r_0}$. Notice that the definition of the functions $u_1$ and $u_2$ differs from that of their analogues, $h$ and $f-h$, in the proof of \cite[Theorem I.6.1]{DaskalHamilton1998}.
The reason for this change is that the zeroth-order coefficient in the definition of $A_0$ is zero, and so uniqueness of $C^{\infty}(\bar\HH)$ solutions to the equation $A_0 u=f$ on $\HH$ does not hold since we may add any constant to a solution, $u$. Since $u=u_1+u_2$, we have
\begin{equation}
\label{eq:Estimate_u_minus_Taylor_polynomial_u_2}
\|u-T^0_1 u_2\|_{C(\bar B^+_r)} \leq \|u_2-T^0_1 u_2\|_{C(\bar B^+_r)}+ \|u_1\|_{C(\bar B^+_r)}.
\end{equation}
By the Mean Value Theorem, we know that
\[
\|u_2-T^0_1 u_2\|_{C(\bar B^+_r)} \leq C r^2 \|D^2 u_2 \|_{C(\bar B^+_r)},
\]
where $C=C(d)$. Because $A_0 u_2 =0$ on $B^+_{1/2}$, we may apply Lemma \ref{lem:Estimate_y_derivatives_u} to $u_2$ with $r=1/2$. Then there are constants, $r_1=r_1(d)$ and $C=C(b_0,d,\lambda_0,\Lambda)$, such that for any $r\in(0,r_1)$ we have
$$
\|D^2 u_2 \|_{C(\bar B^+_r)} \leq C  \|u_2 \|_{C(\bar B^+_{1/2})},
$$
from where it follows that
\begin{equation*}
\|u_2-T^0_1 u_2\|_{C(\bar B^+_r)} \leq C r^2 \|u_2 \|_{C(\bar B^+_1)}.
\end{equation*}
Corollary \ref{cor:Maximum_principle_A_0} gives the estimate
\begin{equation}
\label{eq:Estimate_sup_u_2}
\|u_1\|_{C(\RR^{d-1}\times(0,\nu))} \leq C \|\psi f\|_{C(\RR^{d-1}\times(0,\nu))} \leq \|f\|_{C(\bar B^+_1)},
\end{equation}
where the second inequality follows because the support of $\psi$ is contained in $\bar B^+_1$. Since $u=u_1+u_2$, we have
$$
\|u_2 \|_{C(\bar B^+_1)} \leq \|u \|_{C(\bar B^+_1)} + \|u_1 \|_{C(\bar B^+_1)},
$$
and so, combining the preceding two inequalities,
\[
\|u_2 \|_{C(\bar B^+_1)} \leq \|u \|_{C(\bar B^+_1)}+\|f\|_{C(\bar B^+_1)}.
\]
Thus, we have proved that
\begin{equation*}
\|u_2-T^0_1 u_2\|_{C(\bar B^+_r)} \leq C r^2 \|u\|_{C(\bar B^+_1)} + C\|f\|_{C(\bar B^+_1)},\quad\forall\, r\in(0,r_1).
\end{equation*}
When $r\in [r_1, 1)$, we have,  for all $x\in\bar B^+_r$,
\begin{align*}
|u_2(x)-T^0_1 u_2(x)|
&\leq d|Du_2(0)| r + |u_2(x)| + |u_2(0)|\\
& \leq C r^2 \|u\|_{C(\bar B^+_1)}\quad\hbox{(by Lemma \ref{lem:Estimate_y_derivatives_u} and the fact that $r_1 \leq r$)},
\end{align*}
where $C=C(d)$ is a positive constant. Combining the cases $0<r<r_1$ and $r_1\leq r<1$, we obtain
\begin{equation*}
\|u_2-T^0_1 u_2\|_{C(\bar B^+_r)} \leq C r^2 \|u\|_{C(\bar B^+_1)} + C\|f\|_{C(\bar B^+_1)},\quad\forall\, r\in(0,1),
\end{equation*}
for a constant $C=C(b_0,d, \lambda_0,\Lambda)$. The preceding estimate together with the identity $u=u_1+u_2$ and \eqref{eq:Estimate_sup_u_2} show that
$$
\|u-T^0_1 u_2\|_{C(\bar B^+_r)}
\leq
C\left(r^2 \|u\|_{C(\bar B^+_s)} + \|A_0 u\|_{C(\bar B^+_s)}\right),
$$
and so, the conclusion \eqref{eq:Estimate_with_polynomial} follows with $p=T^0_1 u_2$, in the special case when $r_0=1$.
\end{step}

\begin{step}[Arbitrary $r_0>0$]
When $r_0>0$ is arbitrary, we use rescaling. We let $\tilde u(x) := u(r_0x)$, for all $x\in B^+_1$, and we see that $(A_0\tilde u)(x) = r_0 (A_0 u)(r_0x)$. Notice that the rescaling property $(A_0\tilde u)(x) = r_0 (A_0 u)(r_0x)$ does not hold in this form if the zeroth-order coefficient of $A_0$ is non-zero.

We apply the preceding step to $\tilde u$ with $r$ replaced by $r/r_0$. Then, there is a polynomial $\tilde p$ such that
\[
\|\tilde u-\tilde p\|_{C(\bar B^+_{r/r_0})}
\leq
C\left(\frac{r^2}{r_0^2} \|\tilde u\|_{C(\bar B^+_1)} +  \|A_0 \tilde u\|_{C(\bar B^+_1)}\right),
\]
which is equivalent to
\[
\|u-p\|_{C(\bar B^+_r)}
\leq
C\left(\frac{r^2}{r_0^2} \|u\|_{C(\bar B^+_{r_0})} +  r_0 \|A_0 u\|_{C(\bar B^+_{r_0})}\right),
\]
where we set $p(x):=\tilde p(x/r_0)$. We notice that the polynomial $p$ depends on $r_0$, but not on $r$.
\end{step}
The proof of Proposition \ref{prop:Polynomial_approximation} is now complete.
\end{proof}

Proposition \ref{prop:Polynomial_approximation} is used to obtain the following analogue of \cite[Theorem I.7.1]{DaskalHamilton1998}.

\begin{prop}
\label{prop:Consequence_polynomial_approximation}
Assume that $A$ in \eqref{eq:defnA} obeys Hypothesis \ref{hyp:ConstantCoefficients}.
For any $\alpha \in (0,1)$, there is a positive constant, $S=S(b_0,d,\lambda_0,\Lambda)$, such that for any $u \in C^{\infty}(\bar B^+_1)$ with $T^0_1 u=0$, we have
\begin{equation}
\label{eq:Consequence_polynomial_approximation}
\sup_{0<r \leq 1} \frac{\|u\|_{C(\bar B^+_r)}}{r^{1+\alpha}}
\leq
S\left(\|u\|_{C(\bar B^+_1)} + \sup_{0<r \leq 1} \frac{\|A_0 u\|_{C(\bar B^+_r)}}{r^{\alpha}} \right).
\end{equation}
\end{prop}

\begin{proof}
Because $T^0_1 u=0$ and $u \in C^{\infty}(\bar B^+_1)$, it follows that the quantity on the left-hand side of the inequality \eqref{eq:Consequence_polynomial_approximation} is finite. In addition, the fact that
$T^0_1 u=0$ implies $A_0 u (0)=0$, and so we also have
\[
\sup_{0<r \leq 1} \frac{\|A_0 u\|_{C(\bar B^+_r)}}{r^{\alpha}} <\infty.
\]
Let $r_{*}\in(0,1]$ be such that
\[
\sup_{0<r \leq 1} \frac{\|u\|_{C(\bar B^+_r)}}{r^{1+\alpha}} = \frac{\|u\|_{C(\bar B^+_{r_{*}})}}{r_{*}^{1+\alpha}},
\]
and we define for convenience,
\begin{equation}
\label{defn:Definition_Q}
Q:= \|u\|_{C(\bar B^+_1)} + \sup_{0<r \leq 1} \frac{\|A_0 u\|_{C(\bar B^+_r)}}{r^{\alpha}}.
\end{equation}
We let $S$ (depending on $u$) be such that
\begin{equation}
\label{defn:Definition_S}
\frac{\|u\|_{C(\bar B^+_{r_{*}})}}{r_{*}^{1+\alpha}} = SQ.
\end{equation}
It is sufficient to find an upper bound on $S$, independent of $u$, to give the conclusion \eqref{eq:Consequence_polynomial_approximation}.

Let $q$ and $s$ be positive constants such that $0<q<r_{*}<s\leq 1$. We apply Proposition \ref{prop:Polynomial_approximation} to $u$ with $r$ replaced by $q$ and $r_{*}$ and $r_0$ replaced by $s$. Then, we can find a degree-one polynomial, $p$, such that
\begin{align}
\label{eq:Application_1_polynomial_approximation}
\|u-p\|_{C(\bar B^+_q)}
&\leq
C\left(\frac{q^2}{s^2} \|u\|_{C(\bar B^+_s)} + s \|A_0 u\|_{C(\bar B^+_s)}\right),
\\
\label{eq:Application_2_polynomial_approximation}
\|u-p\|_{C(\bar B^+_{r_*})}
&\leq
C\left(\frac{r_*^2}{s^2} \|u\|_{C(\bar B^+_s)} + s \|A_0 u\|_{C(\bar B^+_s)}\right).
\end{align}
But
\[
\|p\|_{C(\bar B^+_{r_*})}
\leq C \frac{r_{*}}{q} \|p\|_{C(\bar B^+_q)},
\]
for some positive constant $C=C(d)$. We can then estimate
\begin{align*}
\|p\|_{C(\bar B^+_{r_*})}
&\leq C\frac{r_{*}}{q} \left(\|u-p\|_{C(\bar B^+_q)} + \|u\|_{C(\bar B^+_q)}\right),
\end{align*}
and using \eqref{eq:Application_1_polynomial_approximation} and the fact that $q<r_{*}$, we obtain
\begin{align}
\label{eq:Estimate_sup_p_on_ball_r_star}
\|p\|_{C(\bar B^+_{r_*})}
&\leq C\frac{r_{*}}{q}  \left(\frac{r_{*}^2}{s^2} \|u\|_{C(\bar B^+_s)} + s \|A_0 u\|_{C(\bar B^+_s)} + \|u\|_{C(\bar B^+_q)} \right).
\end{align}
From
\[
\|u\|_{C(\bar B^+_{r_*})} \leq \|u-p\|_{C(\bar B^+_{r_*})} + \|p\|_{C(\bar B^+_{r_*})},
\]
and \eqref{eq:Application_2_polynomial_approximation} and \eqref{eq:Estimate_sup_p_on_ball_r_star}, we see that
\begin{align*}
\|u\|_{C(\bar B^+_{r_*})}
&\leq C \left(\frac{r_*^2}{s^2} \|u\|_{C(\bar B^+_s)} + s \|A_0 u\|_{C(\bar B^+_s)}
+ \frac{r_{*}^3}{qs^2} \|u\|_{C(\bar B^+_s)} + \frac{r_{*}s}{q} \|A_0 u\|_{C(\bar B^+_s)} + \frac{r_{*}}{q}\|u\|_{C(\bar B^+_q)}\right)
\\
&\leq C \left(\frac{r_*^2}{s^2} \|u\|_{C(\bar B^+_s)} + \frac{r_{*}}{q}\|u\|_{C(\bar B^+_q)}+ \frac{r_{*}s}{q} \|A_0 u\|_{C(\bar B^+_s)} \right),
\end{align*}
where we have used the fact that $q < r_{*} <s$ to obtain the last inequality. We divide by $r_{*}^{1+\alpha}$ and find that
\begin{align*}
\frac{\|u\|_{C(\bar B^+_{r_*})}}{r_{*}^{1+\alpha}}
&\leq C \left( \left(\frac{r_{*}}{s}\right)^{1-\alpha} \frac{\|u\|_{C(\bar B^+_s)}}{s^{1+\alpha}}
+\left(\frac{q}{r_{*}}\right)^{\alpha} \frac{\|u\|_{C(\bar B^+_q)}}{q^{1+\alpha}}
+\frac{s}{q} \left(\frac{s}{r_{*}}\right)^{\alpha} \frac{\|A_0 u\|_{C(\bar B^+_s)}}{s^{\alpha}}
\right).
\end{align*}
From the preceding inequality and definitions \eqref{defn:Definition_Q} of $Q$ and \eqref{defn:Definition_S} of $S$, we deduce that
\begin{align*}
SQ \leq C \left(\left(\frac{r_{*}}{s}\right)^{1-\alpha} + \left(\frac{q}{r_{*}}\right)^{\alpha}\right) SQ + C \frac{s}{q} \left(\frac{s}{r_{*}}\right)^{\alpha} Q,
\end{align*}
By choosing $r_{*}/s$ and $q/r_{*}$ small enough, we obtain a bound on $S$ depending only on $C=C(b_0,d,\lambda_0,\Lambda)$. Hence, the estimate \eqref{eq:Consequence_polynomial_approximation} now follows.
\end{proof}

We apply Proposition \ref{prop:Consequence_polynomial_approximation} to $R^0_1 u:= u - T^0_1 u$. Note that $A_0 T^0_1 u = (A_0 u) (0)$ and so
\[
A_0\left(u - T^0_1 u\right) = A_0 u -(A_0 u)(0) = R^0_0 A_0 u,
\]
because $x_dD^2u=0$ on $\partial\HH$ and the zeroth-order coefficient of $A_0$ is zero. Thus, Proposition \ref{prop:Consequence_polynomial_approximation} yields the following analogue of \cite[Corollary I.7.2]{DaskalHamilton1998}.

\begin{cor}
\label{cor:Consequence_polynomial_approximation_2}
Assume that $A$ in \eqref{eq:defnA} obeys Hypothesis \ref{hyp:ConstantCoefficients}.
For any $\alpha \in (0,1)$, there is a positive constant, $S=S(b_0,d,\lambda_0,\Lambda)$, such that for any $u \in C^{\infty}(\bar B^+_1)$ we have
\begin{equation*}
\sup_{0<r \leq 1} \frac{\|R^0_1 u\|_{C(\bar B^+_r)}}{r^{1+\alpha}}
\leq
S\left(\|R^0_1 u\|_{C(\bar B^+_1)} + \sup_{0<r \leq 1} \frac{\|R^0_0 A_0 u\|_{C(\bar B^+_r)}}{r^{\alpha}} \right).
\end{equation*}
\end{cor}

Using the inequality \eqref{eq:SimpleEuclidLessCycloidDistance},
$$
|x| \leq 2s^2(x, 0), \quad \forall\, x\in \HH,
$$
where we recall that the cycloidal distance function, $s(x^1,x^2)$ for all $x^1,x^2\in\bar\HH$, is given by \eqref{eq:Cycloidal_distance}, we see that there is a positive constant, $C=C(\alpha,d)$, such that
\begin{align*}
\sup_{0<r\leq 1} \frac{\|R^0_0 A_0 u\|_{C(\bar B^+_r)}}{r^{\alpha}}
&\leq C \sup_{0<r\leq 1}
\left\|\frac{A_0 u(x)-A_0u(0)}{s^{2\alpha}(x,0)}\right\|_{C(\bar B^+_r)}\\
&\leq C\left[A_0u\right]_{C^{2\alpha}_s(\bar B^+_1)}.
\end{align*}
Therefore, Corollary \ref{cor:Consequence_polynomial_approximation_2} gives us the following partial analogue of \cite[Corollary I.7.5]{DaskalHamilton1998}.

\begin{cor}
\label{cor:Consequence_polynomial_approximation_3}
Assume that $A$ in \eqref{eq:defnA} obeys Hypothesis \ref{hyp:ConstantCoefficients}.
For any $\alpha \in (0,1)$, there is a positive constant, $C=C(\alpha,b_0,d,\lambda_0,\Lambda)$, such that for any $u \in C^{\infty}(\bar B^+_1)$ and $0<r\leq 1$, we have
\begin{equation*}
\|R^0_1 u\|_{C(\bar B^+_r)}
\leq
C r^{1+\alpha/2} \left(\|R^0_1 u\|_{C(\bar B^+_1)} + \left[A_0 u\right]_{C^{\alpha}_s(\bar B^+_1)}\right).
\end{equation*}
\end{cor}

Next, we improve the estimate in Corollary \ref{cor:Consequence_polynomial_approximation_3} with the following analogue of \cite[Theorems I.7.3 and I.7.6]{DaskalHamilton1998}.

\begin{prop}
\label{prop:Taylor_polynomial_estimate}
Assume that $A$ in \eqref{eq:defnA} obeys Hypothesis \ref{hyp:ConstantCoefficients}.
For any $\alpha \in (0,1)$, there is a positive constant, $C=C(\alpha,b_0,d,\lambda_0,\Lambda)$, such that for any $u \in C^{\infty}(\bar B^+_1)$, we have
\begin{equation*}
\|T^0_1 u\|_{C(\bar B^+_1)}
\leq
C \left(\|u\|_{C(\bar B^+_1)} + \left[A_0 u\right]_{C^{\alpha}_s(\bar B^+_1)}\right).
\end{equation*}
\end{prop}

\begin{proof}
Because $T^0_1$ is a degree-one polynomial, there is a positive constant, $C=C(d)$, such that
\[
\|T^0_1 u\|_{C(\bar B^+_1)} \leq \frac{C}{r} \|T^0_1 u\|_{C(\bar B^+_r)},\quad\forall\, r \in(0,1].
\]
By Corollary \ref{cor:Consequence_polynomial_approximation_3}, we have, for all $r \in(0,1]$,
\begin{align*}
\|R^0_1 u\|_{C(\bar B^+_r)}
\leq
C r^{1+\alpha/2} \left(\|u\|_{C(\bar B^+_1)} + \|T^0_1 u\|_{C(\bar B^+_1)} + \left[A_0 u\right]_{C^{\alpha}_s(\bar B^+_1)}\right).
\end{align*}
By combining the preceding two inequalities, we find that
\begin{align*}
\|T^0_1 u\|_{C(\bar B^+_1)}
&\leq \frac{C}{r} \left(\|u\|_{C(\bar B^+_1)} + \|R^0_1 u\|_{C(\bar B^+_1)}\right)\\
&\leq \left(\frac{C}{r} + C r^{\alpha/2} \right) \|u\|_{C(\bar B^+_1)} + C r^{\alpha/2} \|T^0_1 u\|_{C(\bar B^+_1)}+ C r^{\alpha/2}\left[A_0 u\right]_{C^{\alpha}_s(\bar B^+_1)}.
\end{align*}
By choosing $r$ small enough so that $Cr^{\alpha/2}\leq1/2$, we obtain the conclusion.
\end{proof}

Proposition \ref{prop:Taylor_polynomial_estimate} implies the following special case ($r=1$) of \cite[Corollary I.7.8]{DaskalHamilton1998}.

\begin{cor}
\label{cor:u_x_u_y_0_0_remainder_estimates}
Assume that $A$ in \eqref{eq:defnA} obeys Hypothesis \ref{hyp:ConstantCoefficients}.
For any $\alpha \in (0,1)$, there is a positive constant, $C=C(\alpha,b_0,d,\lambda_0,\Lambda)$, such that for any $u \in C^{\infty}(\bar B^+_1)$, we have
\begin{equation*}
\|R^0_1 u\|_{C(\bar B^+_1)}
\leq
C \left(\|u\|_{C(\bar B^+_1)} + \left[A_0 u\right]_{C^{\alpha}_s(\bar B^+_1)}\right).
\end{equation*}
\end{cor}

Corollaries \ref{cor:Consequence_polynomial_approximation_3} and \ref{cor:u_x_u_y_0_0_remainder_estimates}  yield the following analogue of \cite[Corollary I.7.8]{DaskalHamilton1998}.

\begin{cor}
\label{cor:Remainder_estimate_2}
Assume that $A$ in \eqref{eq:defnA} obeys Hypothesis \ref{hyp:ConstantCoefficients}.
For any $\alpha \in (0,1)$, there is a positive constant, $C=C(\alpha,b_0,d,\lambda_0,\Lambda)$, such that for any $u \in C^{\infty}(\bar B^+_1)$ and $0<r\leq 1$, we have
\begin{equation*}
\|R^0_1 u\|_{C(\bar B^+_r)}
\leq
C r^{1+\alpha/2} \left(\|u\|_{C(\bar B^+_1)} + \left[A_0 u\right]_{C^{\alpha}_s(\bar B^+_1)}\right).
\end{equation*}
\end{cor}

\section{Schauder estimates away from the degenerate boundary}
\label{sec:Schauder_estimates_interior}
In this section, we use a scaling argument to obtain elliptic Schauder estimates away from the degenerate boundary analogous to the parabolic versions of those estimates in \cite[\S I.8]{DaskalHamilton1998}. Our argument is shorter because we only aim to obtain the estimates in Lemma \ref{lem:Interior_estimates_Holder_seminorms_Du_yD2u} and Corollary \ref{cor:Interior_remainder_estimate}. Even though these estimates are weaker than their analogues \cite[Corollary I.8.7]{DaskalHamilton1998} and \cite[Corollary I.8.8]{DaskalHamilton1998}, respectively, they are sufficient to obtain the full Schauder estimate \eqref{eq:Holder_estimate_local} in Theorem \ref{thm:Holder_estimate_local}. The estimate \eqref{eq:Holder_estimate_local} is proved using a combination of the Schauder estimate on balls $B_r(x_0) \Subset \HH$ which we prove in this section, and the results of Section \ref{sec:Boundary_Schauder_estimates}.
The proof of Proposition \ref{prop:Estimate_D2u_at_Q_r} uses Corollary \ref{cor:Interior_remainder_estimate}, which
is derived from Lemma \ref{lem:Interior_estimates_Holder_seminorms_Du_yD2u}. We have encountered a similar situation in the proof of H\"older continuity along $\partial\HH$ of a weak solution to the Heston elliptic equation in \cite[Theorem 1.11]{Feehan_Pop_regularityweaksoln}. Throughout this section, we continue to assume Hypothesis \ref{hyp:ConstantCoefficients} and so the coefficients, $a,b,c$, of the operator $A$ in \eqref{eq:defnA} and the coefficients, $a,b$, of the operator $A_0$ in \eqref{eq:defnA_0} are constant.

For any $r>0$, we let $Q_r$ denote the point $r e_d \in \HH$. We have the following analogue of \cite[Corollary I.8.7]{DaskalHamilton1998}.

\begin{lem}
\label{lem:Interior_estimates_Holder_seminorms_Du_yD2u}
Assume that $A$ in \eqref{eq:defnA} obeys Hypothesis \ref{hyp:ConstantCoefficients}.
For any $\alpha\in(0,1)$ and positive constants $\mu$ and $\lambda$ such that $0<\mu<\lambda<1$, there is a positive constant $C=C(\alpha,d, \lambda,\lambda_0,\Lambda,\mu)$, such that the following holds. For any function $u \in C^{\infty}(\bar B_{\lambda r}(Q_r))$, we have
\begin{equation}
\label{eq:Interior_estimates_seminorm}
\begin{aligned}
\left[Du\right]_{C^{\alpha}_s(\bar B_{\mu r}(Q_r))}
+ \left[x_dD^2 u\right]_{C^{\alpha}_s(\bar B_{\mu r}(Q_r))}
&\leq
C\left(\frac{1}{r^{1+\alpha/2}} \|u\|_{C(\bar B_{\lambda r}(Q_r))}\right.\\
&\qquad\left. + \frac{1}{r^{\alpha/2}} \|A_0 u\|_{C(\bar B_{\lambda r}(Q_r))}
+ \left[A_0 u\right]_{C^{\alpha}_s(\bar B_{\lambda r}(Q_r))}\right).
\end{aligned}
\end{equation}
\end{lem}

\begin{rmk}
The estimate in \cite[Corollary I.8.7]{DaskalHamilton1998} does not contain the term $\|A_0 u\|_{C(\bar B_{\lambda r}(Q_r))}$ appearing on the right-hand side of our interior estimate \eqref{eq:Interior_estimates_seminorm}. However, our estimate is sufficient to give the Schauder estimate \eqref{eq:Holder_estimate_local} in our Theorem \ref{thm:Holder_estimate_local}.
\end{rmk}

\begin{proof}[Proof of Lemma \ref{lem:Interior_estimates_Holder_seminorms_Du_yD2u}]
The result follows by rescaling. We denote $x=(ry',r+ry_d)\in\HH$, where we recall that we denote $y=(y',y_d)\in\HH=\RR^{d-1}\times\RR_+$, and define
\[
v(y) = u(x),\quad \forall\, y \in B_{\lambda}.
\]
By the hypothesis $u\in C^{\infty} (\bar B_{\lambda r}(Q_r))$, it follows that $v\in C^{\infty}(\bar B_{\lambda})$  and $v$ is a solution to the strictly elliptic equation,
\[
\frac{(1+y_d)}{2}a^{ij}v_{y_iy_j}(y)+b^iv_{y_i}(y)=r\tilde f(y), \quad\forall\, y=(y',y_d) \in B_{\lambda},
\]
where $\tilde f(y):=f(ry',r+ry_d)$, for all $y \in B_{\lambda}$, and $f:=A_0 u$. By the interior Schauder estimates \cite[Theorem 7.1.1]{Krylov_LecturesHolder}, there is a positive constant, $C=C(\alpha,d,\lambda,\lambda_0,\Lambda,\mu)$, such that
\begin{equation}
\label{eq:Interior_Schauder_estimate_D2v}
\|D^2 v\|_{C^{\alpha}(\bar B_{\mu})}
\leq C\left(\|v\|_{C(B_{\lambda})}+r\|\tilde f\|_{C^{\alpha}(\bar B_{\lambda})} \right).
\end{equation}
By direct calculation, we obtain
\begin{equation}
\label{eq:Estimates_after_rescaling}
\begin{aligned}
\|v\|_{C(\bar B_{\lambda})} &= \|u\|_{C(\bar B_{\lambda r}(Q_r))},\\
\|\tilde f\|_{C(\bar B_{\lambda})} &= \|f\|_{C(\bar B_{\lambda r}(Q_r))},\\
[\tilde f]_{C^{\alpha}(\bar B_{\lambda})} &\leq C r^{\alpha/2}\left[f\right]_{C^{\alpha}_s(\bar B_{\lambda r}(Q_r))},
\end{aligned}
\end{equation}
where $C=C(\alpha)$. To see the last inequality, recall that $x = (ry',r+ry_d)$, for all $(y',y_d) \in B_{\lambda}$. For any $y^i \in B_{\lambda}$, for $i=1,2$, we have
\begin{align*}
\frac{|\tilde f(y^1)-\tilde f(y^2)|}{|y^1-y^2|^{\alpha}} = \frac{|f(x^1)-f(x^2)|}{s^\alpha(x^1,x^2)} \frac{s^\alpha(x^1,x^2)}{|y^1-y^2|^{\alpha}}.
\end{align*}
By \eqref{eq:Cycloidal_distance}, we notice that
\begin{align*}
\frac{s(x^1,x^2)}{|y^1-y^2|} = \frac{r|y^1-y^2|}{\sqrt{r(2+y^1_d+y^2_d+|y^1-y^2|)}} \frac{1}{|y^1-y^2|} \leq \sqrt{\frac{r}{2}},
\end{align*}
and so, by letting $C=2^{-\alpha/2}$, we obtain
$$
[\tilde f]_{C^{\alpha}(\bar B_{\lambda})} \leq C r^{\alpha/2}\left[f\right]_{C^{\alpha}_s(\bar B_{\lambda r}(Q_r))}.
$$
We also claim that
\begin{equation}
\label{eq:Second_order_derivative_seminorm}
\left[x_d D^2 u\right]_{C^{\alpha}_s(\bar B_{\mu r}(Q_r))} \leq C r^{-(1+\alpha/2)}\|D^2v\|_{C^{\alpha}(\bar B_{\mu})},
\end{equation}
for a constant $C=C(\alpha,d)$. To establish \eqref{eq:Second_order_derivative_seminorm}, we only need to consider quotients of the form
$$
\frac{|x^1_d D^2 u(x^1)- x_d^2 D^2 u(x^2)|}{s^\alpha(x^1,x^2)},
$$
where $x^1,x^2 \in B_{\mu r}$, and all their coordinates coincide, except for the $i$-th one, where $i=1,\ldots,d$. We only consider the case when $i=d$, as all the other cases, $i=1,\ldots,d-1$, follow in the same way. Recall that we denote $x=(ry',r+ry_d)$, for all $y \in B_{\mu}$. We obtain
\begin{align*}
\frac{|x^1_d D^2 u(x^1)- x_d^2 D^2 u(x^2)|}{s^\alpha(x^1,x^2)}
&\leq \frac{|x^1_d -x^2_d|}{s^\alpha(x^1,x^2)} |D^2 u(x^1)|
+x^2_d\frac{|D^2 u(x^1)- D^2 u(x^2)|}{s^\alpha(x^1,x^2)}.
\end{align*}
Using the definition of the cycloidal distance function \eqref{eq:Cycloidal_distance}, and the fact that $D^2 u(x) = r^{-2} D^2v(y)$, for all $x \in B_{\mu r}$, we see that
\begin{align*}
\frac{|x^1_d D^2 u(x^1)- x_d^2 D^2 u(x^2)|}{s^\alpha(x^1,x^2)}
&\leq \left(2 x_d^1 + 2 x_d^2\right)^{\alpha/2}|x^1_d -x^2_d|^{1-\alpha} \frac{1}{r^2}|D^2 v(y^1)|\\
&\quad+x^2_d \frac{1}{r^2}\frac{|D^2 v(y^1)- D^2 v(y^2)|}{|y^1-y^2|^{\alpha}}\frac{|y^1-y^2|^{\alpha}}{s^\alpha(x^1,x^2)} \\
&\leq 2^{\alpha} r^{-(1+\alpha/2)} \|D^2 v\|_{C(\bar B_{\mu})} \\
&\quad + r^{-1}
\left[D^2 v\right]_{C^{\alpha}(\bar B_{\mu})}\frac{|y^1-y^2|^{\alpha}}{s^\alpha(x^1,x^2)},
\end{align*}
where we used the fact that $x^i_d \leq r$, for all $x^1, x^2 \in B_{\mu r}$. We also have by \eqref{eq:Cycloidal_distance},
\begin{align*}
\frac{|y^1-y^2|}{s(x^1,x^2)} = \frac{r^{-1}|x^1_d-x^2_d|}{|x^1_d-x^2_d|} \sqrt{x^1_d+x^2_d+|x^1_d-x^2_d|} \leq C r^{-1/2},
\end{align*}
which implies that
\begin{align*}
\frac{|x^1_d D^2 u(x^1)- x_d^2 D^2 u(x^2)|}{s^\alpha(x^1,x^2)}
&\leq C r^{-(1+\alpha/2)} \|D^2 v\|_{C^{\alpha}_s(\bar B_{\mu})} ,
\end{align*}
for a constant $C=C(\alpha)$. Now, the inequality \eqref{eq:Second_order_derivative_seminorm} follows immediately.

Using the estimates \eqref{eq:Estimates_after_rescaling} and  \eqref{eq:Second_order_derivative_seminorm}, it follows by \eqref{eq:Interior_Schauder_estimate_D2v} that
\begin{align*}
\left[x_d D^2 u\right]_{C^{\alpha}_s(\bar B_{\mu r}(Q_r))}
&\leq C\left( r^{-(1+\alpha/2)} \|u\|_{C(\bar B_{\lambda r}(Q_r))} + r^{-\alpha/2} \|A_0 u\|_{C(\bar B_{\lambda r}(Q_r))}  + \left[A_0 u\right]_{C^{\alpha}_s(\bar B_{\lambda r}(Q_r))}\right),
\end{align*}
where we substituted $A_0 u$ for $f$.

To obtain the estimate for the H\"older seminorm of $Du$, we proceed by analogy with the argument for $x_d D^2 u$.
\end{proof}

We have the following analogue of \cite[Corollary I.8.8]{DaskalHamilton1998}.

\begin{cor}
\label{cor:Interior_remainder_estimate}
Assume that $A$ in \eqref{eq:defnA} obeys Hypothesis \ref{hyp:ConstantCoefficients}.
For any $\alpha\in(0,1)$ and positive constants $\mu$ and $\lambda$ such that $0<\mu<\lambda<1$, there is a positive constant $C=C(\alpha,d,\lambda,\lambda_0,\Lambda,\mu)$,  such that for any function $u \in C^{\infty}(B_r(Q_r))$ we have
\begin{equation*}
\begin{aligned}
\|R^{Q_r}_2 u\|_{C(\bar B_{\mu r}(Q_r))}
&\leq
C\left(\|u\|_{C(\bar B_{\lambda r}(Q_r))}
+ r^{1+\alpha/2}\left[A_0 u\right]_{C^{\alpha}_s(\bar B_{\lambda r}(Q_r))}
+ r \|A_0 u\|_{C(\bar B_{\lambda r}(Q_r))}
\right).
\end{aligned}
\end{equation*}
\end{cor}

\begin{proof}
As in the case of the inequality preceding \cite[Corollary I.8.8]{DaskalHamilton1998}, we have
\[
\|R^{Q_r}_2 u\|_{C^(\bar B_{\mu r}(Q_r))}
\leq C r^{1+\alpha/2} \left[x_d D^2 u\right]_{C^{\alpha}_s(\bar B_{\mu r}(Q_r))},
\]
for a constant $C=C(d)$. Thus, the conclusion follows from Lemma \ref{lem:Interior_estimates_Holder_seminorms_Du_yD2u} and the preceding inequality.
\end{proof}

\section{Schauder estimates near the degenerate boundary}
\label{sec:Boundary_Schauder_estimates}
In this section, we use the results of the previous sections to prove our main a priori interior local Schauder estimate (Theorem \ref{thm:Holder_estimate_local}) for the operator $A$ on half-balls centered at points in the `degenerate boundary', $\partial\HH$. Throughout this section, we continue to assume Hypothesis \ref{hyp:ConstantCoefficients}, and so the coefficients, $a,b,c$, of the operator $A$ in \eqref{eq:defnA} and the coefficients, $a,b$, of the operator $A_0$ in \eqref{eq:defnA_0} are constant.

We begin with an analogue of \cite[Theorem I.9.1]{DaskalHamilton1998}.

\begin{prop}
\label{prop:Estimate_D2u_at_Q_r}
Assume that $A$ in \eqref{eq:defnA} obeys Hypothesis \ref{hyp:ConstantCoefficients}.
For any $\alpha\in(0,1)$, there is a constant, $C=C(\alpha,b_0,d,\lambda_0,\Lambda)$, such that the following holds. For any function $u \in C^{\infty}(\bar B^+_1)$ and any $r \in (0, 1/2]$, we have
\begin{equation*}
|D^2 u(Q_r)| \leq C r^{\alpha/2-1}\left(\|u\|_{C(\bar B^+_1)} + \left[A_0 u\right]_{C^{\alpha}_s(\bar B^+_1)}\right).
\end{equation*}
\end{prop}

\begin{proof}
We choose $\mu=1/4$ and $\lambda=1/2$ in Corollary \ref{cor:Interior_remainder_estimate}. We consider the points $Q_r:=r e_d$ and $P:=O \in \RR^d$. Let $p:=T^{Q_r}_2 u - T^P_1 u$, where we recall that $T^{Q_r}_2 u$ is the second-degree Taylor polynomial of $u$ at $Q_r$, and $T^P_1 u$ is the first-degree Taylor polynomial of $u$ at $P$. Then, we also have that $p:=R^P_1 u-R^{Q_r}_2 u$, where we recall that $R^{Q_r}_2 u$ is the remainder of the second-degree Taylor polynomial of $u$ at $Q_r$, and $R^P_1 u$ is the remainder of the first-degree Taylor polynomial of $u$ at $P$. There is a positive constant, $C=C(d,\mu)$, such that
\[
|D^2 p| \leq \frac{C}{r^2} \|p\|_{C(\bar B_{\mu r}(Q_r))},
\]
which implies, from the definition of $p$, that
\begin{equation}
\label{eq:Bound_D2u}
|D^2 u(Q_r)| \leq \frac{C}{r^2} \|R^P_1 u-R^{Q_r}_2 u\|_{C(\bar B_{\mu r}(Q_r))}.
\end{equation}
Corollary \ref{cor:Remainder_estimate_2} applied to $R^P_1 u$ gives
\begin{equation}
\label{eq:Bound_RP1u}
\|R^P_1 u\|_{C(\bar B^+_r)} \leq C r^{1+\alpha/2} \left(\|u\|_{C(\bar B^+_1)} +
\left[ A_0 u\right]_{C^{\alpha}_s(\bar B^+_1)}\right),
\end{equation}
and the interior Schauder estimate in Corollary \ref{cor:Interior_remainder_estimate} applied to $R^P_1 u$ yields
\begin{equation*}
\begin{aligned}
\|R^{Q_r}_2 R^P_1 u\|_{C(\bar B_{\mu r}(Q_r))}
&\leq
C\left(\|R^P_1 u\|_{C(\bar B_{\lambda r}(Q_r))}\right.\\
&\quad + \left. r \|A_0 R^P_1 u\|_{C(\bar B_{\lambda r}(Q_r))}
+ r^{1+\alpha/2}\left[A_0 R^P_1 u\right]_{C^{\alpha}_s(\bar B_{\lambda r}(Q_r))}\right),
\end{aligned}
\end{equation*}
for a constant $C=C(\alpha,d,\lambda,\lambda_0,\Lambda,\mu)$. We notice that $A_0 R^P_1 u = A_0 u - (A_0 u)(P)$, from where it follows that
\begin{equation}
\label{eq:Sup_norm_A0RP1u}
\|A_0 R^P_1 u\|_{C(\bar B_{\lambda r}(Q_r))} \leq C r^{\alpha/2 } \left[A_0 u\right]_{C^{\alpha}_s(\bar B^+_1)},
\end{equation}
using the fact \eqref{eq:CycloidLessEuclidDistance} that $s(x^1,x^2) \leq |x^1-x^2|^{1/2}$, for all $x^1,x^2\in\bar\HH$, and also that
\begin{equation}
\label{eq:Holder_seminorm_A0RP1u}
\left[A_0 R^P_1 u\right]_{C^{\alpha}_s(\bar B_{\lambda r}(Q_r))} = \left[A_0 u\right]_{C^{\alpha}_s(\bar B_{\lambda r}(Q_r))}.
\end{equation}
The preceding three inequalities, together with \eqref{eq:Bound_RP1u}, give us the inequality,
\begin{equation*}
\|R^{Q_r}_2 R^P_1 u\|_{C(\bar B_{\mu r}(Q_r))} \leq C r^{1+\alpha/2} \left(\|u\|_{C(\bar B^+_1)} +
\left[ A_0 u\right]_{C^{\alpha}_s(\bar B^+_1)}\right),
\end{equation*}
where we used the fact that $B_{\mu r}(Q_r) \subset B_1$, when $0 <r \leq 1$, for all $0<\mu<1$.
Notice that $R^{Q_r}_2 u=R^{Q_r}_2 R^P_1 u$, and so the preceding estimate becomes
\begin{equation*}
\|R^{Q_r}_2  u\|_{C(\bar B_{\mu r}(Q_r))} \leq C r^{1+\alpha/2} \left(\|u\|_{C(\bar B^+_1)} +
\left[ A_0 u\right]_{C^{\alpha}_s(\bar B^+_1)}\right).
\end{equation*}
The conclusion now follows from the preceding estimate, and inequalities \eqref{eq:Bound_RP1u} and \eqref{eq:Bound_D2u}.
\end{proof}

Note that the definition \eqref{eq:Cycloidal_distance} of the cycloidal distance function gives
\begin{equation*}
s((x',x_d),(x',0))=\sqrt{x_d/2},\quad\forall\, (x',x_d)\in\HH,
\end{equation*}
and hence, via Proposition \ref{prop:Estimate_D2u_at_Q_r}, we obtain the following analogues of \cite[Theorems I.9.3 and I.9.4]{DaskalHamilton1998}.

\begin{cor}
\label{cor:Estimate_D2u}
Assume that $A$ in \eqref{eq:defnA} obeys Hypothesis \ref{hyp:ConstantCoefficients}.
For any $\alpha\in(0,1)$, there is a constant, $C=C(\alpha,b_0,d,\lambda_0,\Lambda)$, such that for all
$x_d \in (0, 1/2]$ and $x' \in \RR^{d-1}$, and any function $u \in C^{\infty}(\bar B^+_1(x',0))$, we have
\begin{equation*}
|x_dD^2 u(x',x_d)| \leq C s^\alpha\left((x',x_d),(x',0)\right)\left(\|u\|_{C(\bar B^+_1(x',0))} + \left[A_0 u\right]_{C^{\alpha}_s(\bar B^+_1(x',0))}\right).
\end{equation*}
\end{cor}

\begin{cor}
\label{cor:Estimate_difference_first_order_derivatives}
Assume that $A$ in \eqref{eq:defnA} obeys Hypothesis \ref{hyp:ConstantCoefficients}.
For any $\alpha\in(0,1)$, there is a constant, $C=C(\alpha,b_0,d,\lambda_0,\Lambda)$, such that for all
$x_d \in (0, 1/2]$ and $x' \in \RR^{d-1}$, and any function $u \in C^{\infty}(\bar B^+_1(x',0))$, we have
\begin{equation*}
|D u(x',x_d) - D u(x',0) | \leq C s^\alpha\left((x',x_d),(x',0)\right)\left(\|u\|_{C(\bar B^+_1(x',0))} + \left[A_0 u\right]_{C^{\alpha}_s(\bar B^+_1(x',0))}\right).
\end{equation*}
\end{cor}

\begin{proof}
Following the proof of \cite[Theorem I.9.4]{DaskalHamilton1998}, using Proposition \ref{prop:Estimate_D2u_at_Q_r} and translation-invariance with respect to $x'\in\RR^{d-1}$ to obtain the second inequality, we have
\begin{align*}
|D u(x',x_d)-D u(x',0) |
&\leq \int_0^{x_d} |Du_{x_d}(x',t)|\, dt \\
&\leq C \left(\|u\|_{C(\bar B^+_1(x',0))} + \left[A_0 u\right]_{C^{\alpha}_s(\bar B^+_1(x',0))}\right) \int_0^{x_d} t^{\alpha/2-1} \,dt\\
&= C \left(\|u\|_{C(\bar B^+_1(x',0))} + \left[A_0 u\right]_{C^{\alpha}_s(\bar B^+_1(x',0))}\right) x_d^{\alpha/2}.
\end{align*}
Using the fact that $s((x',x_d),(x',0))=\sqrt{x_d/2}$, we obtain the conclusion.
\end{proof}

Next, we use Lemma \ref{lem:Interior_estimates_Holder_seminorms_Du_yD2u}
(for estimates away from $\partial\HH$) and the Taylor remainder estimates in
Corollary \ref{cor:Remainder_estimate_2} (for estimates near $\partial\HH$) to prove the following analogue of \cite[Theorem I.9.5]{DaskalHamilton1998}.

\begin{prop}
\label{prop:Estimate_Holder_seminorm_Du_D2u_interior}
Assume that $A$ in \eqref{eq:defnA} obeys Hypothesis \ref{hyp:ConstantCoefficients}.
Let $\alpha\in(0,1)$, and $0<r\leq 1/4$, and $0<\mu<1$. Then there is a constant, $C=C(\alpha,d,\lambda_0,\Lambda,\mu)$, such that for any function $u\in C^{\infty}(\bar B^+_1)$, we have
\begin{equation*}
\left[Du\right]_{C^{\alpha}_s(\bar B_{\mu r}(Q_r))} + \left[x_dD^2u\right]_{C^{\alpha}_s(\bar B_{\mu r}(Q_r))}
\leq C\left(\|u\|_{C(\bar B^+_1)}+\left[A_0 u\right]_{C^{\alpha}_s(\bar B^+_1)}\right).
\end{equation*}
\end{prop}

\begin{proof}
For convenience, we denote $v:= R^P_1 u=u-T^P_1 u$. We notice that
\[
[Du]_{C^{\alpha}_s(\bar B_{\mu r}(Q_r))} = [Dv]_{C^{\alpha}_s(\bar B_{\mu r}(Q_r))}
\quad\hbox{and}\quad
[x_dD^2 u]_{C^{\alpha}_s(\bar B_{\mu r}(Q_r))} = [x_dD^2 v]_{C^{\alpha}_s(\bar B_{\mu r}(Q_r))},
\]
and hence we only need to estimate $[Dv]_{C^{\alpha}_s(\bar B_{\mu r}(Q_r))}$ and $[x_dD^2 v]_{C^{\alpha}_s(\bar B_{\mu r}(Q_r))}$. The proof is similar to the proof of Proposition \ref{prop:Estimate_D2u_at_Q_r}. The interior Schauder estimates in Lemma \ref{lem:Interior_estimates_Holder_seminorms_Du_yD2u} applied to $v$ with $\lambda=(1+\mu)/2$ yield
\begin{align*}
\left[D v\right]_{C^{\alpha}_s(\bar B_{\mu r}(Q_r))}
+ \left[x_dD^2 v\right]_{C^{\alpha}_s(\bar B_{\mu r}(Q_r))}
&\leq
C\left(\frac{1}{r^{1+\alpha/2}} \|v\|_{C(\bar B_{\lambda r}(Q_r))}\right.\\
&\qquad\left. + \frac{1}{r^{\alpha/2}} \|A_0 v\|_{C(\bar B_{\lambda r}(Q_r))}
+ \left[A_0 v\right]_{C^{\alpha}_s(\bar B_{\lambda r}(Q_r))}\right),
\end{align*}
for some constant $C=C(\alpha,d,\lambda_0,\Lambda,\mu)$. The conclusion now follows from the preceding estimate and inequalities
\eqref{eq:Bound_RP1u} applied on $B_{\lambda r}(Q_r)$ instead of $B^+_r$ (notice that $B_{\lambda r}(Q_r) \subset B^+_{1/2}$, since $0<r\leq 1/4$), together with \eqref{eq:Sup_norm_A0RP1u} and \eqref{eq:Holder_seminorm_A0RP1u}.
\end{proof}

Next, we have the following analogue of \cite[Theorems I.9.7 and I.9.8]{DaskalHamilton1998}.

\begin{prop}
\label{prop:Estimate_Holder_seminorm_Du_D2u_boundary}
Assume that $A$ in \eqref{eq:defnA} obeys Hypothesis \ref{hyp:ConstantCoefficients}.
For $\alpha\in(0,1)$, there are constants $\gamma=\gamma(d)\in (0,1)$ and $C=C(\alpha,b_0,d,\lambda_0,\Lambda)$ such that, for any function $u\in C^{\infty}(\bar B^+_1)$, we have
\begin{equation}
\label{eq:Estimate_Holder_seminorm_Du_D2u_boundary}
\left[Du\right]_{C^{\alpha}_s(\bar B^+_{\gamma})} + \left[x_dD^2u\right]_{C^{\alpha}_s(\bar B^+_{\gamma})}
\leq C\left(\|u\|_{C(\bar B^+_1)}+\left[A_0 u\right]_{C^{\alpha}_s(\bar B^+_1)}\right).
\end{equation}
\end{prop}

\begin{proof}
We combine the arguments of the proofs of \cite[Theorems I.9.7 and I.9.8]{DaskalHamilton1998}. Let $x^i\in B^+_{\gamma}$, for $i=1,2$, where $\gamma$ will be fixed below. We may assume without loss of generality that $x_d^1 \geq x_d^2$. We consider two cases.

\begin{case}[$x^1$ and $x^2$ close together relative to their distance from $\partial\HH$] 
If $|x^1-x^2| \leq x_d^1/4$, then $x^2 \in B_{x_d^1/4}(x^1)$, and the estimate \eqref{eq:Estimate_Holder_seminorm_Du_D2u_boundary} follows if we assume $0<\gamma \leq 1/2$ and apply Proposition \ref{prop:Estimate_Holder_seminorm_Du_D2u_interior} with $\mu=1/4$ and $r=x^1_d$.
\end{case}

\begin{case}[$x^1$ and $x^2$ farther apart relative to their distance from $\partial\HH$]
We next consider the case when
\begin{equation}
\label{eq:Second_case}
|x^1-x^2| > x_d^1/4.
\end{equation}
Writing $x=(\bar x,x_d)\in\RR^{d-1}\times\RR_+$, we define the points,
\begin{align*}
x^3 := (\bar x^1,0)\quad\hbox{and}\quad x^4 := (\bar x^2,0),\\
x^5 := (\bar x^1,r)\quad\hbox{and}\quad x^6 := (\bar x^2,r),
\end{align*}
where the positive constant $r$ will be chosen below. Notice that when \eqref{eq:Second_case} holds, we have
$$
s(x^1, x^2) \geq \frac{1}{8} \sqrt{x_d^1},
$$
by the definition \eqref{eq:Cycloidal_distance} of the cycloidal distance function. By the definition of the points $x^i$, for $i=3,4$, and the fact that $s((x',x_d),(x',0))=\sqrt{x_d/2}$,  we see that
\begin{equation}
\label{eq:Distance_inequalities}
\begin{aligned}
s(x^1, x^2) &\geq 8 s(x^1, x^3),\\
s(x^1, x^2) &\geq 8 s(x^2, x^4)\quad\hbox{(since $x_d^1\geq x_d^2$)}.
\end{aligned}
\end{equation}
Let $v$ denote $Du$ or $x_d D^2 u$, and consider the difference
\begin{equation}
\label{eq:Decompose_v}
\begin{aligned}
v(x^1)-v(x^2) &= \left(v(x^1)-v(x^3)\right) + \left(v(x^3)-v(x^5)\right) + \left(v(x^5)-v(x^6)\right)\\
&\quad + \left(v(x^6)-v(x^4)\right) + \left(v(x^4)-v(x^2)\right).
\end{aligned}
\end{equation}
Using the distance inequalities \eqref{eq:Distance_inequalities}, we find that
\begin{align*}
\frac{\left|v(x^1)-v(x^3)\right|}{s^\alpha(x^1,x^2)} \leq 8^{\alpha} \frac{\left|v(x^1)-v(x^3)\right|}{s^\alpha(x^1,x^3)},\\
\frac{\left|v(x^2)-v(x^4)\right|}{s^\alpha(x^1,x^2)} \leq 8^{\alpha} \frac{\left|v(x^2)-v(x^4)\right|}{s^\alpha(x^2,x^4)}.
\end{align*}
By Corollary \ref{cor:Estimate_D2u}, if $v=x_d D^2 u$, and Corollary \ref{cor:Estimate_difference_first_order_derivatives}, if $v=Du$, we obtain
\begin{equation}
\label{eq:Term_1_v_decomposition}
\frac{\left|v(x^1)-v(x^3)\right|}{s^\alpha(x^1,x^2)} + \frac{\left|v(x^2)-v(x^4)\right|}{s^\alpha(x^1,x^2)}
\leq C\left(\|u\|_{C(\bar B^+_1)}+\left[A_0 u\right]_{C^{\alpha}_s(\bar B^+_1)}\right),
\end{equation}
for a constant $C=C(\alpha,b_0,d,\lambda_0,\Lambda)$.

We now let $r:=B s^2(x^1,x^2)$, where the constant $B$ will be chosen below. Using the fact that $s(x^3,x^5)=\sqrt{r/2}$ and definition of $x^i$, for $i=3,5$, we obtain
\begin{align*}
\frac{\left|v(x^3)-v(x^5)\right|}{s^\alpha(x^1,x^2)} &= \left(B/2\right)^{\alpha/2}\frac{\left|v(x^3)-v(x^5)\right|}{s^\alpha(x^3,x^5)}.
\end{align*}
Because $x^i\in B^+_{\gamma}$, for $i=1,2$, and due to the inequality \eqref{eq:CycloidLessEuclidDistance}, we can choose the constant $B:=1/(4\gamma)$ such that
$$
r=Bs^2(x^1,x^2)\leq B |x^1-x^2| \leq B \gamma \leq 1/4.
$$
We apply Corollary \ref{cor:Estimate_D2u}, when $v=x_dD^2u$, and Corollary \ref{cor:Estimate_difference_first_order_derivatives}, when $v=Du$, to obtain
\begin{equation}
\label{eq:Term_2_v_decomposition}
\frac{\left|v(x^3)-v(x^5)\right|}{s^\alpha(x^1,x^2)}
\leq C  \left(B/2\right)^{\alpha/2} \left(\|u\|_{C(\bar B^+_1)}+\left[A_0 u\right]_{C^{\alpha}_s(\bar B^+_1)}\right).
\end{equation}
The inequality,
\begin{equation}
\label{eq:Term_3_v_decomposition}
\frac{\left|v(x^4)-v(x^6)\right|}{s^\alpha(x^1,x^2)}
\leq C  \left(B/2\right)^{\alpha/2} \left(\|u\|_{C(\bar B^+_1)}+\left[A_0 u\right]_{C^{\alpha}_s(\bar B^+_1)}\right),
\end{equation}
follows by the same argument used to obtain the estimate \eqref{eq:Term_2_v_decomposition}.

Using \eqref{eq:Second_case} and the assumption $x^1_d \geq x^2_d$, we see that
$$
|\bar x^1- \bar x^2| \leq |x^1-x^2| \leq \frac{3}{2}\frac{|x^1-x^2|^2}{x^1_d+x^2_d+|x^1-x^2|} = \frac{3}{2} s^2(x^1,x^2).
$$
Recalling that $B=1/(4\gamma)$ and $r=B s^2(x^1,x^2)$, we have
\begin{align*}
|\bar x^1-\bar x^2|\leq \frac{3}{2B} Bs^2(x^1,x^2)
\leq 6\gamma r.
\end{align*}
Next, we choose $\gamma=1/24$, and so
$$
|\bar x^1-\bar x^2|\leq r/4.
$$
for all $x^i=(x^i_1,\cdots,x^i_d) \in B^+_{\gamma}$, for $i=1,2$. Because $|\bar x^1-\bar x^2|\leq r/4$, we may apply Proposition \ref{prop:Estimate_Holder_seminorm_Du_D2u_interior}, with $\mu=1/4$, to obtain
\begin{equation*}
\frac{\left|v(x^5)-v(x^6)\right|}{s^\alpha(x^5,x^6)}
\leq C  \left(\|u\|_{C(\bar B^+_1)}+\left[A_0 u\right]_{C^{\alpha}_s(\bar B^+_1)}\right).
\end{equation*}
Again using the definition $r:=B s^2(x^1,x^2)$, we notice that
\begin{align*}
s(x^5,x^6) &\leq \frac{|\bar x^1-\bar x^2|}{\sqrt{2r}} \leq \frac{\sqrt{3}}{4} s(x^1,x^2),
\end{align*}
and so the preceding two inequalities yield
\begin{equation}
\label{eq:Term_4_v_decomposition}
\frac{\left|v(x^5)-v(x^6)\right|}{s^{\alpha}(x^1,x^2)}
\leq C  \left(\|u\|_{C(\bar B^+_1)}+\left[A_0 u\right]_{C^{\alpha}_s(\bar B^+_1)}\right).
\end{equation}
Combining the estimates \eqref{eq:Term_1_v_decomposition}, \eqref{eq:Term_2_v_decomposition}, \eqref{eq:Term_3_v_decomposition} and \eqref{eq:Term_4_v_decomposition} gives us the estimate \eqref{eq:Estimate_Holder_seminorm_Du_D2u_boundary}, when condition \eqref{eq:Second_case} holds.
\end{case}
The conclusion now follows from the two cases we considered.
\end{proof}

By analogy with \cite[Corollary I.9.9]{DaskalHamilton1998}, we have

\begin{prop}
\label{prop:Holder_estimate_A_0}
Assume that $A$ in \eqref{eq:defnA} obeys Hypothesis \ref{hyp:ConstantCoefficients}.
For any $\alpha\in(0,1)$, there are positive constants, $\gamma=\gamma(d)\in (0,1)$ and $C=C(\alpha,b_0,d,\lambda_0,\Lambda)$, such that the following holds. If $u\in C^{\infty}(\bar B^+_1)$, then
\begin{equation}
\label{eq:Holder_estimate}
\|u\|_{C^{2+\alpha}_s(\bar B^+_{\gamma})} \leq C\left(\|u\|_{C(\bar B^+_1)}+\left[A_0 u\right]_{C^{\alpha}_s(\bar B^+_1)}\right).
\end{equation}
\end{prop}

\begin{proof}
Let $\gamma=\gamma(d)\in (0,1)$ be as in Proposition \ref{prop:Estimate_Holder_seminorm_Du_D2u_boundary}. The bound on $x_dD^2 u$ follows from Corollary \ref{cor:Estimate_D2u}. Proposition \ref{prop:Estimate_Holder_seminorm_Du_D2u_boundary} gives us the estimate \eqref{eq:Holder_estimate} for the $C^{\alpha}_s(\bar B^+_{\gamma})$ H\"older seminorms of $Du$ and $x_dD^2 u$. We only need to establish the bound on $Du$, namely that there is a constant $C=C(\alpha,b_0,d,\lambda_0,\Lambda)$, such that
\begin{equation}
\label{eq:Holder_norm_Du}
\|Du\|_{C(\bar B^+_{\gamma})} \leq C\left(\|u\|_{C(\bar B^+_1)}+\left[A_0 u\right]_{C^{\alpha}_s(\bar B^+_1)}\right).
\end{equation}
We follow the argument of \cite[p. 932]{DaskalHamilton1998}. Let $x^0\in \bar B^+_{\gamma}$ be such that $|Du(x^0)| = \|Du\|_{C(\bar B^+_{\gamma})}$. Then by Proposition \ref{prop:Estimate_Holder_seminorm_Du_D2u_boundary} we have, for all $x \in \bar B^+_{\gamma}$,
\[
|Du(x)-Du(x^0)| \leq C_0\left(\|u\|_{C(\bar B^+_1)}+\left[A_0 u\right]_{C^{\alpha}_s(\bar B^+_1)}\right),
\]
for a constant $C_0=C_0(\alpha,b_0,d,\lambda_0,\Lambda)$.

Let $N \geq 2$ be a positive integer such that
\[
\|Du\|_{C(\bar B^+_{\gamma})} \geq N C_0 \left(\|u\|_{C(\bar B^+_1)}+\left[A_0 u\right]_{C^{\alpha}_s(\bar B^+_1)}\right).
\]
Estimate \eqref{eq:Holder_norm_Du} will follow if we can find an upper bound on $N$, independent of $u$. The preceding two inequalities give
\[
|Du(x)| \geq (N-1) C_0\left(\|u\|_{C(\bar B^+_1)}+\left[A_0 u\right]_{C^{\alpha}_s(\bar B^+_1)}\right), \quad \forall\, x \in \bar B^+_{\gamma},
\]
and the Mean Value Theorem yields
\[
|u(x)-u(x^0)| \geq |x-x^0|(N-1) C_0\left(\|u\|_{C(\bar B^+_1)}+\left[A_0 u\right]_{C^{\alpha}_s(\bar B^+_1)}\right), \quad \forall\, x, x^0 \in \bar B^+_{\gamma}.
\]
Choosing $x \in B^+_{\gamma}$ such that $|x-x^0|\geq \gamma/2$, we obtain a contradiction with \eqref{eq:Estimate_Holder_seminorm_Du_D2u_boundary} if $N$ is too large. Thus, \eqref{eq:Holder_norm_Du} follows.
\end{proof}

We have the following corollary of Proposition \ref{prop:Holder_estimate_A_0}:

\begin{cor}
\label{cor:Holder_estimate_any_r_A_0}
Assume that $A$ in \eqref{eq:defnA} obeys Hypothesis \ref{hyp:ConstantCoefficients}.
For any $\alpha\in(0,1)$,  there are positive constants, $\gamma=\gamma(d)\in(0,1)$ and $C=C(\alpha, b_0,d,\lambda_0,\Lambda)$, such that for any $r>0$ the following holds. If $u\in C^{\infty}(\bar B^+_r)$, then
\begin{equation}
\label{eq:Holder_estimate_any_r_A_0}
\|u\|_{C^{2+\alpha}_s(\bar B^+_{\gamma r})} \leq C r^{-(1+\alpha/2)}\left(\|u\|_{C(\bar B^+_{r})}+\left[A_0 u\right]_{C^{\alpha}_s(\bar B^+_{r})}\right).
\end{equation}
\end{cor}

\begin{proof}
Let $\gamma=\gamma(d)\in (0,1)$ be as in Proposition \ref{prop:Estimate_Holder_seminorm_Du_D2u_boundary}. We set $v(x):=u(rx)$, for all $x \in B^+_1$. The estimates in Proposition \ref{prop:Holder_estimate_A_0} applied to $v$ give us
\begin{equation*}
\|u\|_{C^{2+\alpha}_s(\bar B^+_{\gamma r})} \leq C \left(\|u\|_{C(\bar B^+_{r})}+\left[A_0 u\right]_{C^{\alpha}_s(\bar B^+_{r})}\right),
\end{equation*}
where $C=C(\alpha, b_0, d,\lambda_0,\Lambda,r)$. The dependence of the constant $C$ on $r$ follows as in the proof of Lemma \ref{lem:Interior_estimates_Holder_seminorms_Du_yD2u}, and so we obtain \eqref{eq:Holder_estimate_any_r_A_0}.
\end{proof}

We now generalize Corollary \ref{cor:Holder_estimate_any_r_A_0} to allow for any $\gamma \in (0,1)$ and make explicit the dependence of the constant $C$ appearing in \eqref{eq:Holder_estimate_any_r_A_0} on $r$ and $\gamma$.

\begin{cor}
\label{cor:Holder_estimate_any_r_any_gamma_A_0}
If $\alpha\in(0,1)$ and $r_0>0$, then there are positive constants, $p=p(\alpha)$ and $C=C(\alpha, b_0,d,\lambda_0,\Lambda,r_0)$, such that, for any $r\in (0,r_0)$ and $\gamma\in(0,1)$, the following holds. If $u\in C^{\infty}(\bar B^+_r)$, then
\begin{equation}
\label{eq:Holder_estimate_any_r_any_gamma_A_0}
\|u\|_{C^{2+\alpha}_s(\bar B^+_{\gamma r})} \leq C ((1-\gamma)r)^{-p}\left(\|u\|_{C(\bar B^+_{r})} + 
\left\|A_0 u\right\|_{C^{\alpha}_s(\bar B^+_{r})}\right).
\end{equation}
\end{cor}

Corollary \ref{cor:Holder_estimate_any_r_any_gamma_A_0} is proved at the end of this section.

\begin{proof}[Proof of Theorem \ref{thm:Holder_estimate_local}]
We combine the localization procedure in the proof of \cite[Theorem 8.11.1]{Krylov_LecturesHolder} with Corollary \ref{cor:Holder_estimate_any_r_any_gamma_A_0}. We divide the proof into two steps. Set $R:=(r+r_0)/2$.

\setcounter{step}{0}
\begin{step}[A priori estimate for $u\in C^{\infty}(\underline B^+_{r_0})$]
Consider the sequence of radii, $\{r_n\}_{n\geq 1}\subset [r, R)$, defined by $r_1 := r$ and
\begin{equation}
\label{eq:r_n}
r_n:=r+(R-r)\sum_{k=1}^{n-1} \frac{1}{2^k},\quad\forall\, n \geq 2.
\end{equation}
Denote $B_n:=B^+_{r_n}(x^0)$, for all $n \geq 1$. Let $\{\varphi_n\}_{n \geq 1}$ be a sequence of $C^\infty_0(\bar\HH)$ cutoff functions such that, for all $n \geq 1$, we have $0\leq \varphi_n\leq 1$ with $\varphi_n=1$ on $B_n$ and $\varphi_n=0$ outside $B_{n+1}$. Let
\begin{equation}
\label{eq:alpha_n}
\alpha_n:=\|u\varphi_n\|_{C^{2+\alpha}_s(\bar B_n)},\quad\forall\, n\geq 1.
\end{equation}
By applying the estimate \eqref{eq:Holder_estimate_any_r_any_gamma_A_0} to $u\varphi_n$ with $r=r_{n+1}$ and $\gamma = r_n/r_{n+1}$, we obtain
\begin{align*}
\alpha_n &\leq C (r_{n+1}-r_n)^{-p}\left(\|u\varphi_n\|_{C(\bar B_{n+1})}+
\|A_0 u\|_{C^{\alpha}_s(\bar B_{n+1})}\right)\\
         &\leq C (R-r)^{-p}2^{(n-1)p}\left(\|u\varphi_n\|_{C(\bar B_{n+1})}+\|A u\|_{C^{\alpha}_s(\bar B_{n+1})}
         + \|u\varphi_{n+1}\|_{C^{\alpha}_s(\bar B_{n+1})}\right),
\end{align*}
where the last inequality follows from the fact that $A=A_0+c$ by \eqref{eq:defnA} and \eqref{eq:defnA_0} and employing \eqref{eq:r_n}. The interpolation inequalities (Lemma \ref{lem:InterpolationIneqS}) give, for any $\eps>0$,
$$
\|u\varphi_{n+1}\|_{C^{\alpha}_s(\bar B_{n+1})} \leq \eps \alpha_{n+1}+ C\eps^{-m} \|u\varphi_{n+1}\|_{C(B_{n+1})} \quad\hbox{(by \eqref{eq:alpha_n})},
$$
where $C=C(\alpha,d,R)$ and $m=m(\alpha,d)$ are positive constants independent of $\eps$. Choosing $\eps:=\delta C^{-1} 2^{-(n-1)p} (R-r)^{p}$, we obtain, for all $\delta>0$,
\begin{align*}
\alpha_n &\leq  \delta\alpha_{n+1} +  C\left((R-r)^{-p} 2^{(n-1)p} + \delta^{-m} (R-r)^{-p(m+1)} 2^{(n-1)(m+1)p}\right) \\
&\quad \times \left(\|A u\|_{C^{\alpha}_s(\bar B^+_{R}(x^0))}
         + \|u\varphi_{n+1}\|_{C^{\alpha}_s(\bar B^+_{R}(x^0))}\right),
\end{align*}
and now the estimate \eqref{eq:Holder_estimate_local} follows as in the proofs of \cite[Theorem 8.11.1]{Krylov_LecturesHolder} or
\cite[Theorem 3.8]{Feehan_Pop_mimickingdegen_pde}.
\end{step}

\begin{step}[A priori estimate for $u\in C^{2+\alpha}_s(\underline B^+_{r_0}(x^0))$]
Choose a sequence $\{u_n\}_{n \in \NN}\subset C^{\infty}(\bar B^+_R(x^0))$ such that $u_n\to u$ in $C^{2+\alpha}_s(\bar B^+_R(x^0))$ as $n\to\infty$. Applying the estimate  \eqref{eq:Holder_estimate_local} to each $u_n$ and then taking the limit as $n\to\infty$, yields the a priori estimate \eqref{eq:Holder_estimate_local} for
$u\in C^{2+\alpha}_s(\underline B^+_{r_0}(x^0))$.
\end{step}
This concludes the proof of Theorem \ref{thm:Holder_estimate_local}.
\end{proof}

To prove Corollary \ref{cor:Holder_estimate_any_r_any_gamma_A_0}, we make use of the a priori Schauder estimates for strictly elliptic operators. The statement of Proposition \ref{prop:Schauder_unif_elliptic} is the same as that of \cite[Corollary 6.3]{GilbargTrudinger} except that in the estimate \eqref{eq:Schauder_unif_elliptic}, the dependence of the constant $N_2$ on the constant of uniform ellipticity and the $C^{\alpha}$ norm of the coefficients, is made explicit in \eqref{eq:N_1}, \eqref{eq:p}, and \eqref{eq:N_2}; a close parabolic analogue is given by \cite[Proposition  3.13]{Feehan_Pop_mimickingdegen_pde}.

\begin{prop}[Quantitative a priori Schauder estimate for a strictly elliptic operator]
\label{prop:Schauder_unif_elliptic}
Let $\alpha\in (0,1)$, $\delta$, $K$, and $\rho$ be positive constants. Then there are positive constants,
\begin{align}
\label{eq:N_1}
N_1&=N_1(\alpha,d,\rho),
\\
\label{eq:p}
p&=p(\alpha),
\\
\label{eq:N_2}
N_2&=N_1(1+\delta^{-p}+K^p),
\end{align}
such that the following holds. Let $\Omega \subset \RR^d$ be an open subset and
$$
\bar Av: = -\tr(\bar a D^2v) - \langle \bar b, Dv \rangle + \bar c v, \quad\forall\, v \in C^2(\Omega),
$$
where the components of $\bar a:\Omega\to\sS^+(d)$ and $\bar b:\Omega\to\RR^d$ and the function $c:\Omega\to\RR$ belong to $C^{\alpha}(\Omega)$, and obey
\begin{gather*}
\langle\bar a\xi, \xi\rangle \geq \delta |\xi|^2 \quad\hbox{on }\Omega, \quad\forall\, \xi\in\RR^d,
\\
\|\bar a\|_{C^{\alpha}(\bar\Omega)} + \|\bar b\|_{C^{\alpha}(\bar\Omega)} + \|\bar c\|_{C^{\alpha}(\bar\Omega)} \leq K.
\end{gather*}
If $\Omega'\subset\Omega$ is an open subset such that $\dist(\partial\Omega', \partial\Omega)\geq \rho$ and $u\in C^{2+\alpha}(\Omega)$, then
\begin{align}
\label{eq:Schauder_unif_elliptic}
\|u\|_{C^{2+\alpha}(\bar\Omega')} \leq N_2 \left(\|\bar A u\|_{C^{\alpha}(\bar\Omega)}+\|u\|_{C(\bar\Omega)}\right).
\end{align}
\end{prop}

\begin{proof}
We begin with an elliptic analogue of \cite[Lemma 3.11]{Feehan_Pop_mimickingdegen_pde}:

\begin{claim}
\label{claim:Schauder_unif_elliptic_2}
Assume the hypotheses of Proposition \ref{prop:Schauder_unif_elliptic} and, in addition, that $\bar a$ is constant, $\bar b =0$, and $\bar c = 0$, and $\Omega=\RR^d$. If $u\in C^{2+\alpha}(\bar\RR^d)$, then
\begin{equation}
\label{eq:Schauder_unif_elliptic_2}
\|u\|_{C^{2+\alpha}(\bar\RR^d)} \leq N_2 \left(\|\bar Au\|_{C^{\alpha}(\bar\RR^d)}+\|u\|_{C(\bar\RR^d)}\right).
\end{equation}
\end{claim}

\begin{proof}
By \cite[Theorem 3.6.1]{Krylov_LecturesHolder}, there is a positive constant, $C'=C'(\alpha,d,\delta,K)$, such that any $u\in C^{2+\alpha}(\bar\RR^d)$ obeys
\begin{equation*}
\left[D^2u\right]_{C^{\alpha}(\bar\RR^d)} \leq C' \left(\|\bar Au\|_{C^{\alpha}(\bar\RR^d)}+\|u\|_{C(\bar\RR^d)}\right).
\end{equation*}
From the interpolation inequalities \cite[Theorem 3.2.1]{Krylov_LecturesHolder}, we have that for any $\eps>0$, these is a positive constant, $C''=C''(\alpha,d,\eps)$, such that
$$
\|u\|_{C^2(\bar\RR^d)} \leq \eps \left[D^2u\right]_{C^{\alpha}(\bar\RR^d)}+ C'' \|u\|_{C(\bar\RR^d)}.
$$
Combining the preceding two inequalities, we obtain that there is a positive constant, $C=C(\alpha,d,\delta,K)$, such that
\begin{align}
\label{eq:Schauder_unif_elliptic_1}
\|u\|_{C^{2+\alpha}(\bar\RR^d)} \leq C \left(\|\bar Au\|_{C^{\alpha}(\bar\RR^d)}+\|u\|_{C(\bar\RR^d)}\right).
\end{align}
Using an argument completely analogous to that employed in the proof of \cite[Lemma 3.11]{Feehan_Pop_mimickingdegen_pde} to derive \cite[Inequality (3.82)]{Feehan_Pop_mimickingdegen_pde}, we can refine \eqref{eq:Schauder_unif_elliptic_1} by finding constants as in \eqref{eq:N_1}, \eqref{eq:p}, and \eqref{eq:N_2} such that $u\in C^{2+\alpha}(\bar\RR^d)$ satisfies \eqref{eq:Schauder_unif_elliptic_2}; in that argument, we need only replace the cylinder $(0,T)\times\RR^d$ in
\cite{Feehan_Pop_mimickingdegen_pde} by the space $\RR^d$ here and to replace the use of \cite[Theorem 9.2.1]{Krylov_LecturesHolder} by the inequality \eqref{eq:Schauder_unif_elliptic_1}.
\end{proof}

Next, we allow the coefficients of $\bar A$ in Claim \ref{claim:Schauder_unif_elliptic_2} to be arbitrary functions in $C^\alpha(\RR^d)$ obeying the hypotheses of Proposition \ref{prop:Schauder_unif_elliptic} and give an elliptic analogue of \cite[Proposition 3.12]{Feehan_Pop_mimickingdegen_pde}:

\begin{claim}
\label{claim:Schauder_unif_elliptic_2_variable_coefficients}
Assume the hypotheses of Proposition \ref{prop:Schauder_unif_elliptic} and, in addition, that $\Omega=\RR^d$. If $u\in C^{2+\alpha}(\bar\RR^d)$, then $u$ obeys \eqref{eq:Schauder_unif_elliptic_2}.
\end{claim}

\begin{proof}
We apply the argument used to prove \cite[Proposition 3.12]{Feehan_Pop_mimickingdegen_pde}, but replace the role of \cite[Lemma 3.11]{Feehan_Pop_mimickingdegen_pde} by that of Claim \ref{claim:Schauder_unif_elliptic_2}, 
and the role of the classical parabolic H\"older interpolation inequalities \cite[Theorem 8.8.1]{Krylov_LecturesHolder} by the classical elliptic H\"older interpolation inequalities \cite[Theorem 3.2.1]{Krylov_LecturesHolder}.
In the proof of \cite[Proposition 3.12]{Feehan_Pop_mimickingdegen_pde}, we used the classical parabolic weak maximum principle estimate \cite[Corollary 8.1.5]{Krylov_LecturesHolder} to eliminate the need to add the
$C([0,T]\times\bar\RR^d)$-norm of the function $u$ on the right-hand side of \cite[Inequality (3.86)]{Feehan_Pop_mimickingdegen_pde}. In our elliptic estimate \eqref{eq:Schauder_unif_elliptic_2}, we do not need to use a maximum principle estimate because we allow the $C(\bar\RR^d)$-norm of the function $u$ to appear on the right-hand side of \eqref{eq:Schauder_unif_elliptic_2}.
\end{proof}

Now consider an arbitrary open subset $\Omega \subset \RR^d$ and $u\in C^{2+\alpha}(\Omega)$. We cover $\Omega'$ by a countable set of balls, $\{B_{\rho/4}(x^n)\}_{n\in\NN}$, such that $\{x^n\}_{n\in\NN} \subset \Omega'$ and
\begin{align}
\label{eq:Cover}
\Omega'\subset\bigcup_{n\in\NN} B_{\rho/4}(x^n) \subset \bigcup_{n\in\NN} B_{\rho/2}(x^n)\subset \Omega.
\end{align}
Applying the localization argument used to prove \cite[Proposition 3.13]{Feehan_Pop_mimickingdegen_pde} (more specifically, \cite[Proposition Theorem 3.8]{Feehan_Pop_mimickingdegen_pde}, whose proof is in turn an adaptation of the proof of \cite[Theorem 8.11.1]{Krylov_LecturesHolder})
but replacing the role of \cite[Proposition 3.12]{Feehan_Pop_mimickingdegen_pde} with that of Claim \ref{claim:Schauder_unif_elliptic_2_variable_coefficients}, we find that $u\in C^{2+\alpha}(\bar B_{\rho/2}(x^n))$ satisfies
$$
\|u\|_{C^{2+\alpha}(\bar B_{\rho/4}(x^n))} \leq N_2 \left(\|\bar Au\|_{C^{\alpha}(\bar B_{\rho/2}(x^n))}+\|u\|_{C(\bar B_{\rho/2}(x^n))}\right),
\quad\forall\, n\in\NN.
$$
From the preceding inequality and the inclusion relations \eqref{eq:Cover}, we obtain inequality \eqref{eq:Schauder_unif_elliptic} and this completes the proof of Proposition \ref{prop:Schauder_unif_elliptic}.
\end{proof}

We can now give the

\begin{proof}[Proof of Corollary \ref{cor:Holder_estimate_any_r_any_gamma_A_0}]
Let $\gamma_d\in (0,1)$ denote the constant produced by Corollary \ref{cor:Holder_estimate_any_r_A_0}. We consider two cases, where $\gamma \in (0,1)$ in the hypotheses of Corollary \ref{cor:Holder_estimate_any_r_any_gamma_A_0} obeys either $0<\gamma\leq\gamma_d$ or $\gamma>\gamma_d$; clearly we only need to consider the second case, since in the first case the conclusion of Corollary \ref{cor:Holder_estimate_any_r_any_gamma_A_0} is implied by Corollary \ref{cor:Holder_estimate_any_r_A_0}. Our proof of the inequality \eqref{eq:Holder_estimate_any_r_any_gamma_A_0} uses a covering and rescaling argument.
Let
\begin{equation}
\label{eq:defn_t}
t := (1-\gamma)r/2,
\end{equation}
and divide the half-ball, $B^+_{\gamma r}$, into the two regions,
$$
U_1:= B^+_{\gamma r} \cap \left(\RR^{d-1} \times (0, \gamma_d t/2)\right)
\quad\hbox{and}\quad
U_2:= B^+_{\gamma r} \less U_1 = B^+_{\gamma r} \cap \left(\RR^{d-1} \times [\gamma_d t/2, \infty)\right).
$$
We cover $U_1$ by a finite number of half-balls, $B^+_{\gamma_d t}(x^n)$, centered at points $x^n\in \partial_0 B^+_r(x^0)$, and we apply the estimate \eqref{eq:Holder_estimate_any_r_A_0} to obtain
$$
\|u\|_{C^{2+\alpha}_s(\bar B^+_{\gamma_d t}(x^n))}
\leq C t^{-(1+\alpha/2)}\left(\|u\|_{C(\bar B^+_{r})}+\left[A_0 u\right]_{C^{\alpha}_s(\bar B^+_{r})}\right),
$$
and thus
\begin{equation}
\label{eq:Estimate_near_boundary}
\|u\|_{C^{2+\alpha}_s(\bar B^+_{\gamma_d t}(x^n))}
\leq C ((1-\gamma)r)^{-(1+\alpha/2)}\left(\|u\|_{C(\bar B^+_{r})}+\left[A_0 u\right]_{C^{\alpha}_s(\bar B^+_{r})}\right),
\end{equation}
where $C=C(\alpha,b_0,d,\lambda_0,\Lambda)$ and we used the fact that that $t=(1-\gamma)r/2$.

In the region $U_2$, the operator $A_0$ is strictly elliptic (because $x_d\geq \gamma_d t/2>0$). Using a rescaling argument and Proposition \ref{prop:Schauder_unif_elliptic}, we next show that there are positive constants, $C=C(\alpha,d,\lambda_0,\Lambda,r_0)$ and $p=p(\alpha)$, such that
\begin{align}
\label{eq:Estimate_away_boundary}
\|u\|_{C^{2+\alpha}_s(\bar U_2)}
&\leq C ((1-\gamma)r)^{-p}\left(\|u\|_{C(\bar B^+_{r})}+\left\|A_0 u\right\|_{C^{\alpha}_s(\bar B^+_{r})}\right).
\end{align}
To prove inequality \eqref{eq:Estimate_away_boundary}, we apply the rescaling $x=(ty',ty_d)\in\HH$, where we recall that we denote $y=(y',y_d)\in\HH=\RR^{d-1}\times\RR_+$. Notice that the rescaling $x=(ty',ty_d)$ transforms $B^+_{\gamma r}$ into $B^+_{2\gamma(1-\gamma)^{-1}}$, and the set $U_2$ becomes
$$
\Omega':= B^+_{2\gamma(1-\gamma)^{-1}}\backslash \left(\RR^{d-1} \times (0, \gamma_d /2)\right).
$$
We let
$$
\Omega:= B^+_{2(1-\gamma)^{-1}}\backslash \left(\RR^{d-1} \times (0, \gamma_d /4)\right),
$$
and we define
\[
v(y) := u(x),\quad \forall\, y \in \Omega.
\]
By the hypothesis that $u\in C^{\infty}(\bar B^+_r)$, it follows that $v\in C^{\infty}(\bar \Omega)$ and $v$ is a solution to the strictly elliptic equation,
\[
\hat A_0v(y)
:=
-\frac{y_d}{2}a^{ij}v_{y_iy_j}(y) - b^iv_{y_i}(y) = t\tilde f(y), \quad\forall\, y=(y',y_d) \in \Omega,
\]
where $\tilde f(y):=f(ty',ty_d)$, for all $y \in \Omega$, and $f:=A_0 u$. We apply Proposition \ref{prop:Schauder_unif_elliptic} to $v$ with $\bar A$ replaced by
$\hat A_0$ on the open subset $\Omega$. We notice that $\hat A_0$ is a strictly elliptic operator on $\Omega$, because for all $y=(y',y_d)\in \Omega$ we have $y_d \geq \gamma_d/4$, and that in the notation of Proposition \ref{prop:Schauder_unif_elliptic}, we have
\begin{gather*}
\delta:=\lambda_0\gamma_d/4,
\\
K=2(1-\gamma)^{-1}\Lambda,
\\
\rho:=\dist(\partial\Omega', \partial\Omega)\geq \gamma_d/4,
\end{gather*}
where the identification of $K$ uses the facts that $y_d < 2(1-\gamma)^{-1}$ for all $y\in\Omega$ and $\gamma\in(0,1)$. From Proposition \ref{prop:Schauder_unif_elliptic} and the fact that $\gamma_d=\gamma_d(\alpha)$ we obtain positive constants, $N_1 = N_1(\alpha,d,\rho) = N_1(\alpha,d)$
and $p=p(\alpha)$, such that
$$
\|v\|_{C^{2+\alpha}(\bar \Omega')}
\leq N_1\left(1 + (\lambda_0\gamma_d/4)^p + (1-\gamma)^{-p}(2\Lambda)^p\right)\left(\|\tilde f\|_{C^{\alpha}(\bar \Omega)} + \|v\|_{C(\bar \Omega)}\right),
$$
and hence
\begin{equation}
\label{eq:Interior_Schauder_estimate_D2v_1}
\|v\|_{C^{2+\alpha}(\bar \Omega')}
\leq C_1(1-\gamma)^{-p}\left(\|\tilde f\|_{C^{\alpha}(\bar \Omega)} + \|v\|_{C(\bar \Omega)}\right),
\end{equation}
for a positive constant $C_1=C_1(\alpha,d,\lambda_0,\Lambda)$. We next show that the preceding estimate implies inequality \eqref{eq:Estimate_away_boundary}.
We claim that
\begin{subequations}
\label{eq:Estimates_after_rescaling_1}
\begin{align}
\label{eq:Estimates_after_rescaling_1_v_C0}
\|v\|_{C(\bar \Omega)} &\leq \|u\|_{C(\bar B_r^+)},
\\
\label{eq:Estimates_after_rescaling_1_tildef_C0}
\|\tilde f\|_{C(\bar \Omega)} &= \|f\|_{C(\bar B_r^+)},
\\
\label{eq:Estimates_after_rescaling_1_tildef_Halpha}
[\tilde f]_{C^{\alpha}(\bar \Omega)} &\leq C ((1-\gamma)r)^{\alpha/2}\left[f\right]_{C^{\alpha}_s(\bar B_r^+)},
\end{align}
\end{subequations}
where $C=C(\alpha,d)$. Inequalities \eqref{eq:Estimates_after_rescaling_1_v_C0} and \eqref{eq:Estimates_after_rescaling_1_tildef_C0} are immediate by direct calculation. For the inequality \eqref{eq:Estimates_after_rescaling_1_tildef_Halpha}, recall that $x = (ty',ty_d)$, for all $(y',y_d) \in \Omega$. For any $y^i \in \Omega$, with $i=1,2$, we have
\begin{align*}
\frac{|\tilde f(y^1)-\tilde f(y^2)|}{|y^1-y^2|^{\alpha}} = \frac{|f(x^1)-f(x^2)|}{s^\alpha(x^1,x^2)} \frac{s^\alpha(x^1,x^2)}{|y^1-y^2|^{\alpha}}.
\end{align*}
By \eqref{eq:Cycloidal_distance} and \eqref{eq:defn_t}, we see that
\begin{align*}
\frac{s(x^1,x^2)}{|y^1-y^2|} &= \frac{t|y^1-y^2|}{\sqrt{t(y^1_d+y^2_d+|y^1-y^2|)}} \frac{1}{|y^1-y^2|}
\\
                             &\leq \frac{\sqrt{t}}{\sqrt{\gamma_d/2}}
                             \quad\hbox{(using the fact that $y_d \geq \gamma_d/4$, for all $y=(y',y_d)\in\Omega$)},
\end{align*}
and this, choosing $C=(\gamma_d/2)^{-\alpha/2}$, gives \eqref{eq:Estimates_after_rescaling_1_tildef_Halpha}. We also claim that
\begin{subequations}
\label{eq:Rescaling_derivatives_1}
\begin{align}
\label{eq:Rescaling_derivatives_1_Du_C0}
\|D u\|_{C(\bar U_2)} &\leq C(1-\gamma)^{-1} r^{-1}\|Dv\|_{C(\bar \Omega')},
\\
\label{eq:Rescaling_derivatives_1_xdD2u_C0}
\|x_d D^2 u\|_{C(\bar U_2)} &\leq C(1-\gamma)^{-2} r^{-1}\|D^2v\|_{C(\bar \Omega')},
\\
\label{eq:Rescaling_derivatives_1_u_Calpha_s}
\left[u\right]_{C^{\alpha}_s(\bar U_2)} &\leq
Cr^{\alpha/2}\|v\|_{C^{\alpha}(\bar \Omega')},
\\
\label{eq:Rescaling_derivatives_1_Du_Calpha_s}
\left[D u\right]_{C^{\alpha}_s(\bar U_2)} &\leq
C(1-\gamma)^{-1} r^{-1+\alpha/2}\|Dv\|_{C^{\alpha}(\bar \Omega')},
\\
\label{eq:Rescaling_derivatives_1_xdD2u_Calpha_s}
\left[x_d D^2 u\right]_{C^{\alpha}_s(\bar U_2)} &\leq C(1-\gamma)^{-(2+\alpha)} r^{-(1+\alpha/2)}\|D^2v\|_{C^{\alpha}(\bar \Omega')},
\end{align}
\end{subequations}
for a constant $C=C(\alpha)$. Inequalities \eqref{eq:Rescaling_derivatives_1_Du_C0} and \eqref{eq:Rescaling_derivatives_1_xdD2u_C0} follow by direct calculation.
We shall only give the details of the proof of inequality \eqref{eq:Rescaling_derivatives_1_xdD2u_Calpha_s}, as the justifications for inequalities \eqref{eq:Rescaling_derivatives_1_u_Calpha_s} and \eqref{eq:Rescaling_derivatives_1_Du_Calpha_s} are very similar. To establish \eqref{eq:Rescaling_derivatives_1_xdD2u_Calpha_s}, we only need to consider quotients of the form
$$
\frac{|x^1_d D^2 u(x^1)- x_d^2 D^2 u(x^2)|}{s^\alpha(x^1,x^2)},
$$
where $x^1,x^2 \in U_2$, and all their coordinates coincide, except for the $i$-th one, where $i=1,\ldots,d$. Furthermore, we shall only consider the case when $i=d$, as all the other cases, $i=1,\ldots,d-1$, follow in the same way. We have
\begin{align*}
\frac{|x^1_d D^2 u(x^1)- x_d^2 D^2 u(x^2)|}{s^\alpha(x^1,x^2)}
&\leq \frac{|x^1_d -x^2_d|}{s^\alpha(x^1,x^2)} |D^2 u(x^1)|
+x^2_d\frac{|D^2 u(x^1)- D^2 u(x^2)|}{s^\alpha(x^1,x^2)}.
\end{align*}
By using the fact that $D^2 u(x) = t^{-2} D^2v(y)$, for all $x\in B_{\gamma r}^+$, recalling that $x=(ty',ty_d)$, for all $y \in \Omega'$, and recalling the definition of the cycloidal distance function \eqref{eq:Cycloidal_distance}, we obtain
\begin{align*}
\frac{|x^1_d D^2 u(x^1)- x_d^2 D^2 u(x^2)|}{s^\alpha(x^1,x^2)}
&\leq \left(x^1_d+x^2_d+|x^1_d-x^2_d|\right)^{\alpha/2}|x^1_d -x^2_d|^{1-\alpha} \frac{1}{t^2}|D^2 v(y^1)|
\\
&\quad + x^2_d \frac{1}{t^2}\frac{|y^1-y^2|^{\alpha}}{s^\alpha(x^1,x^2)}\frac{|D^2 v(y^1)- D^2 v(y^2)|}{|y^1-y^2|^{\alpha}}.
\end{align*}
But the definition of the cycloidal distance function \eqref{eq:Cycloidal_distance} and the fact that $x^i_d \leq r$, for all $x^1, x^2 \in B_{\gamma r}^+$, gives
\begin{align*}
\frac{|y^1-y^2|}{s(x^1,x^2)} &= \frac{2(1-\gamma)^{-1}r^{-1}|x^1_d-x^2_d|}{|x^1_d-x^2_d|} \sqrt{x^1_d+x^2_d+|x^1_d-x^2_d|}
\\
&\leq 4(1-\gamma)^{-1}r^{-1/2}.
\end{align*}
Combining the preceding inequalities with the definition \eqref{eq:defn_t} of $t=(1-\gamma)r/2$ yields
\begin{align*}
\frac{|x^1_d D^2 u(x^1)- x_d^2 D^2 u(x^2)|}{s^\alpha(x^1,x^2)}
&\leq 2^{\alpha+2} (1-\gamma)^{-2}r^{-(1+\alpha/2)}
\|D^2 v\|_{C(\bar\Omega')}
\\
&\quad + 4(1-\gamma)^{-2}r^{-1}
\left(4(1-\gamma)^{-1}r^{-1/2}\right)^\alpha
\left[D^2 v\right]_{C^{\alpha}(\bar\Omega')}.
\end{align*}
Therefore, noting that $\gamma \in (0,1)$,
\begin{align*}
\frac{|x^1_d D^2 u(x^1)- x_d^2 D^2 u(x^2)|}{s^\alpha(x^1,x^2)}
&\leq C (1-\gamma)^{-(2+\alpha)}r^{-(1+\alpha/2)}
\|D^2 v\|_{C^{\alpha}(\bar\Omega')},
\end{align*}
for a constant $C=C(\alpha)$. The inequality \eqref{eq:Rescaling_derivatives_1_xdD2u_Calpha_s} follows immediately.

Using inequalities \eqref{eq:Estimates_after_rescaling_1} and  \eqref{eq:Rescaling_derivatives_1}, it follows by \eqref{eq:Interior_Schauder_estimate_D2v_1} that estimate \eqref{eq:Estimate_away_boundary} holds. Estimate \eqref{eq:Holder_estimate_any_r_any_gamma_A_0} follows by combining \eqref{eq:Estimate_near_boundary} and \eqref{eq:Estimate_away_boundary}.
\end{proof}

\section{Higher-order a priori Schauder estimates for operators with constant coefficients}
\label{sec:Holder_estimate_local_higherorder}
In this section, we prove a higher-order version of Theorem \ref{thm:Holder_estimate_local}, our basic a priori local interior Schauder estimate, and a global a priori global Schauder estimate on a slab (Corollary \ref{cor:Global_Schauder_estimate_ConstantCoefficients}), both when $A$ has constant coefficients. Throughout this section, we continue to assume Hypothesis \ref{hyp:ConstantCoefficients} and so the coefficients, $a,b,c$, of the operator $A$ in \eqref{eq:defnA} and the coefficients, $a,b$, of the operator $A_0$ in \eqref{eq:defnA_0} are constant.

\begin{thm}[Higher-order a priori local interior Schauder estimate when $A$ has constant coefficients]
\label{thm:Holder_estimate_local_higherorder}
Assume the hypotheses of Theorem \ref{thm:Holder_estimate_local} and let $k\in\NN$. If $u\in C^{k,2+\alpha}_s(\underline B^+_{r_0}(x^0))$, then
\begin{equation}
\label{eq:Holder_estimate_local_higherorder}
\|u\|_{C^{k, 2+\alpha}_s(\bar B^+_r)} \leq C\left(\|A u\|_{C^{k, \alpha}_s(\bar B^+_{r_0}(x^0))} + \|u\|_{C(\bar B^+_{r_0}(x^0))}\right),
\end{equation}
where $C$ now also depends on $k$.
\end{thm}

\begin{proof}
Choose $r_1 := (r+r_0)/2 \in (r, r_0)$. For any multi-index $\beta\in\NN^d$ with $|\beta| = \beta_1+\cdots+\beta_d\leq k$, direct calculation yields
\begin{equation}
\label{eq:ADbetaCommutator}
D^\beta Av = A_{(\beta_d)}D^\beta v - \beta_d\sum_{i,j \neq d} a^{ij}D^{\beta_0+(\beta_d-1)e_d}v_{x_ix_j}, \quad v \in C^\infty(\HH),
\end{equation}
where we write $\beta_0 := \beta - \beta_d e_d$ and, for $l\in\NN$,
\begin{align*}
A_{(l)}v := -x_da^{ij}v_{x_ix_j} - \sum_{i\neq d}\left(b^i + 2l a^{id}\right) v_{x_i} + \left(b^d + la^{dd}\right)v_{x_d} + cv.
\end{align*}
Note that $A_{(0)}=A$. To prove \eqref{eq:Holder_estimate_local_higherorder}, we see by Definition \ref{defn:DH2spaces} that it suffices to establish
\begin{equation}
\label{eq:Holder_estimate_local_higherorder_Dbetau}
\|D^\beta u\|_{C^{2+\alpha}_s(\bar B_r^+)} \leq C\left(\|Au\|_{C^{k,\alpha}_s(\bar B_{r_0}^+(x^0))} + \|u\|_{C(\bar B_{r_0}^+(x^0))}\right),
\end{equation}
for any multi-index $\beta\in\NN^d$ with $|\beta|\leq k$, where $C$ has the dependencies given in our hypotheses.

Theorem \ref{thm:Holder_estimate_local} yields \eqref{eq:Holder_estimate_local_higherorder} when $k=0$. Therefore, as an induction hypothesis for $k$, we assume that \eqref{eq:Holder_estimate_local_higherorder} holds with $k$ replaced by any $l\in\NN$ in the range $0 \leq l\leq k-1$ and we seek to prove \eqref{eq:Holder_estimate_local_higherorder_Dbetau} and hence \eqref{eq:Holder_estimate_local_higherorder} by induction on $l$ when $|\beta|=k$.

We first consider the case $\beta_d=0$, so $AD^\beta v = D^\beta Av$. Then
\begin{align*}
\|D^\beta u\|_{C^{2+\alpha}_s(\bar B_r^+)} &\leq C\left(\|AD^\beta u\|_{C^\alpha_s(\bar B_{r_1}^+(x^0))} + \|D^\beta u\|_{C(\bar B_{r_1}^+(x^0))}\right)
\quad\hbox{(by \eqref{eq:Holder_estimate_local}}
\\
&\leq C\left(\|D^\beta Au\|_{C^\alpha_s(\bar B_{r_1}^+(x^0))} + \|D^\beta u\|_{C(\bar B_{r_1}^+(x^0))}\right)
\quad\hbox{(by \eqref{eq:ADbetaCommutator})}
\\
&\leq C\left(\|Au\|_{C^{k,\alpha}_s(\bar B_{r_1}^+(x^0))} + \|u\|_{C^{k,\alpha}_s(\bar B_{r_1}^+(x^0))}\right)
\quad\hbox{(by Definition \ref{defn:DHspaces})}
\\
&\leq C\left(\|Au\|_{C^{k,\alpha}_s(\bar B_{r_1}^+(x^0))} + \|u\|_{C^{k-1, 2+\alpha}_s(\bar B_{r_1}^+(x^0))}\right)
\quad\hbox{(by Definition \ref{defn:DH2spaces})}
\\
&\leq C\left(\|Au\|_{C^{k,\alpha}_s(\bar B_{r_1}^+(x^0))} + \|Au\|_{C^{k-1,\alpha}_s(\bar B_{r_0}^+(x^0))} + \|u\|_{C(\bar B_{r_0}^+(x^0))}\right),
\end{align*}
where the final inequality follows by induction on $l$ and the a priori Schauder estimate \eqref{eq:Holder_estimate_local_higherorder} with $k$ replaced by $l=k-1$ (and $r$ replaced by $r_1$). Since $r_1<r_0$, we can combine terms and obtain \eqref{eq:Holder_estimate_local_higherorder_Dbetau} in the case $\beta_d=0$.

Now we consider the case $0\leq \beta_d \leq k$ and argue by induction on $\beta_d$. As an induction hypothesis for $\beta_d$, we assume that \eqref{eq:Holder_estimate_local_higherorder_Dbetau} holds when $0\leq \beta_d\leq k-1$. For $\beta_d$ in the range $1\leq \beta_d \leq k$ (and thus $|\beta_0|\leq k-1$), we have
\begin{align*}
{}&\|D^\beta u\|_{C^{2+\alpha}_s(\bar B_r^+)}
\\
&\quad\leq C\left(\|A_{(\beta_d)}D^\beta u\|_{C^\alpha_s(\bar B_{r_1}^+(x^0))} + \|D^\beta u\|_{C(\bar B_{r_1}^+(x^0))}\right)
 \quad\hbox{(by \eqref{eq:Holder_estimate_local})}
\\
&\quad\leq C\left(\|D^\beta Au\|_{C^\alpha_s(\bar B_{r_1}^+(x^0))} + \sum_{i,j\neq d}\|D^{\beta_0+(\beta_d-1)e_d}u_{x_ix_j}\|_{C^\alpha_s(\bar B_{r_1}^+(x^0))} + \|D^\beta u\|_{C(\bar B_{r_1}^+(x^0))}\right)
\\
&\quad\leq C\left(\|Au\|_{C^{k, \alpha}_s(\bar B_{r_1}^+(x^0))} + \max_{i\neq d}\|D^{\beta_0+e_i+(\beta_d-1)e_d}u\|_{C^{1,\alpha}_s(\bar B_{r_1}^+(x^0))} + \|u\|_{C^{k, \alpha}_s(\bar B_{r_1}^+(x^0))}\right),
\end{align*}
where the penultimate inequality follows from \eqref{eq:ADbetaCommutator} and the final inequality by Definition \ref{defn:DHspaces} of our H\"older norms. Because $C^{2+\alpha}_s(\bar B_{r_1}^+(x^0))\hookrightarrow C^{1,\alpha}_s(\bar B_{r_1}^+(x^0))$ by Definitions \ref{defn:DHspaces} and \ref{defn:DH2spaces}, we see that
\begin{align*}
\|D^\beta u\|_{C^{2+\alpha}_s(\bar B_r^+)}
&\leq C\left(\|Au\|_{C^{k, \alpha}_s(\bar B_{r_1}^+(x^0))} + \max_{i\neq d}\|D^{\beta_0+e_i+(\beta_d-1)e_d}u\|_{C^{2+\alpha}_s(\bar B_{r_1}^+(x^0))} + \|u\|_{C^{k, \alpha}_s(\bar B_{r_1}^+(x^0))}\right)
\\
&\leq C\left(\|Au\|_{C^{k, \alpha}_s(\bar B_{r_1}^+(x^0))} + \|Au\|_{C^{k,\alpha}_s(\bar B_{r_0}^+(x^0))} + \|u\|_{C(\bar B_{r_0}^+(x^0))} + \|u\|_{C^{k, \alpha}_s(\bar B_{r_1}^+(x^0))}\right)
\\
&\qquad\hbox{(by induction on $\beta_d$ and \eqref{eq:Holder_estimate_local_higherorder_Dbetau} since $\beta_d-1\leq k-1$)}
\\
&\leq C\left(\|Au\|_{C^{k,\alpha}_s(\bar B_{r_0}^+(x^0))} + \|u\|_{C^{k,\alpha}_s(\bar B_{r_1}^+(x^0))} + \|u\|_{C(\bar B_{r_0}^+(x^0))}\right)
\quad\hbox{(since $r_1<r_0$)}
\\
&\leq C\left(\|Au\|_{C^{k,\alpha}_s(\bar B_{r_0}^+(x^0))} + \|u\|_{C^{k-1, 2+\alpha}(\bar B_{r_1}^+(x^0))} + \|u\|_{C(\bar B_{r_0}^+(x^0))}\right)
\\
&\leq C\left(\|Au\|_{C^{k,\alpha}_s(\bar B_{r_1}^+(x^0))} + \|Au\|_{C^{k-1,\alpha}_s(\bar B_{r_0}^+(x^0))} + \|u\|_{C(\bar B_{r_0}^+(x^0))}\right),
\end{align*}
where the penultimate inequality follows from the embedding $C^{k-1, 2+\alpha}_s(\bar B_{r_1}^+(x^0))\hookrightarrow C^{k,\alpha}_s(\bar B_{r_1}^+(x^0))$ implied by Definitions \ref{defn:DHspaces} and \ref{defn:DH2spaces} and the final inequality follows by induction on $l$ and the a priori Schauder estimate \eqref{eq:Holder_estimate_local_higherorder} with $k$ replaced by $l=k-1$ (and $r$ replaced by $r_1$). Again, since $r_1<r_0$, we can combine terms and obtain \eqref{eq:Holder_estimate_local_higherorder_Dbetau} in this case too.
\end{proof}

Let $\nu>0$ and let $S = \RR^{d-1} \times (0,\nu)$, as in \eqref{eq:Slab}. Theorem \ref{thm:Holder_estimate_local_higherorder} together with a priori estimates for strictly elliptic operators in \cite[\S 6]{GilbargTrudinger} now imply the following global Schauder estimate on a slab.

\begin{cor}[A priori global Schauder estimate on a slab when $A$ has constant coefficients]
\label{cor:Global_Schauder_estimate_ConstantCoefficients}
Assume that $A$ in \eqref{eq:defnA} obeys Hypothesis \ref{hyp:ConstantCoefficients}.
For $\alpha\in(0,1)$, constant $\nu > 0$, and $k\in\NN$, there is a positive constant, $C=C(\alpha,b_0,d,k,\lambda_0,\Lambda,\nu)$, such that the following holds. If $u\in C^{k, 2+\alpha}_s(\bar S)$ and $u = 0$ on $\partial_1 S$, then
\begin{equation}
\label{eq:Global_Schauder_estimate_constant_coeff_l}
\| u\|_{C^{k, 2+\alpha}_s(\bar S)} \leq C \left(\|A u\|_{C^{k, \alpha}_s(\bar S)} + \|u\|_{C(\bar S)}\right),
\end{equation}
and, when $c\geq 0$,
\begin{equation}
\label{eq:Global_Schauder_estimate_constant_coeff_l_nonnegc}
\| u\|_{C^{k, 2+\alpha}_s(\bar S)} \leq C \|A u\|_{C^{k, \alpha}_s(\bar S)}.
\end{equation}
\end{cor}

\begin{proof}
Let $r:= \nu/2$, and let $\{x^n\}_{n \in \NN} \subset \partial \HH$ be a sequence of points such that
$$
\RR^{d-1} \times (0,r/4) \subset \bigcup_{n \in \NN} B^+_{r/2} (x^n).
$$
Using the a priori interior local Schauder estimate \eqref{eq:Holder_estimate_local} on each half-ball $B^+_r(x^n)$, we obtain
$$
\| u\|_{C^{k, 2+\alpha}_s(\bar B^+_{r/2}(x^n))} \leq C \left(\|A u\|_{C^{k, \alpha}_s(\bar B^+_r(x^n))} + \|u\|_{C(\bar B^+_r(x^n))}\right).
$$
By applying a standard covering argument to the slab, $S_0:=\RR^{d-1}\times(0, r)$, we find that
$$
\| u\|_{C^{k, 2+\alpha}_s(\bar S_0)} \leq C \left(\|A u\|_{C^{k, \alpha}_s(\bar S)} + \|u\|_{C(\bar S)}\right).
$$
By \cite[Lemma 6.5 and Problem 6.2]{GilbargTrudinger} and a similar covering argument, there is a constant $\delta>0$ such that, if $S_1:=\RR^{d-1}\times(\nu-\delta,\nu)$, we have
$$
\| u\|_{C^{k, 2+\alpha}_s(\bar S_1)} \leq C \left(\|A u\|_{C^{k, \alpha}_s(\bar S)} + \|u\|_{C(\bar S)}\right).
$$
Setting $S_2:=\RR^{d-1}\times(r/4,\nu-\delta/2)$ and now applying \cite[Corollary 6.3 and Problem 6.1]{GilbargTrudinger} and a covering argument, we obtain
$$
\| u\|_{C^{k, 2+\alpha}_s(\bar S_2)} \leq C \left(\|A u\|_{C^{k, \alpha}_s(\bar S)} + \|u\|_{C(\bar S)}\right).
$$
By combining the preceding three estimates, we obtain \eqref{eq:Global_Schauder_estimate_constant_coeff_l} and by appealing to Corollary \ref{cor:Maximum_principle_A_0}, we obtain \eqref{eq:Global_Schauder_estimate_constant_coeff_l_nonnegc}.
\end{proof}

\section{A priori Schauder estimates, global existence, and regularity for operators with variable coefficients}
\label{sec:Variable_coefficients_Higher-order_regularity}
In Section \ref{subsec:Variable_coefficients}, we relax the condition in Hypothesis \ref{hyp:ConstantCoefficients} that the coefficients, $a,b,c$, of the operator $A$ in \eqref{eq:defnA} are constant, which we assumed in Sections \ref{sec:Derivative_estimates_interior_boundary}--\ref{sec:Holder_estimate_local_higherorder}, to prove a generalization (Theorem \ref{thm:Holder_estimate_local_variable_coeff}) of our $C^{2+\alpha}_s$ a priori Schauder estimate (Theorem \ref{thm:Holder_estimate_local}) from the case of constant coefficients, $a,b,c$, to the case of variable coefficients. We then prove Theorem \ref{thm:Higher_order_estimate}, extending the preceding $C^{2+\alpha}_s$ a priori Schauder estimate to a $C^{k, 2+\alpha}_s$ a priori Schauder estimate for arbitrary $k\in \NN$. This allows us to complete the proofs of Theorem \ref{thm:APrioriSchauderInteriorDomain} and Corollary \ref{cor:Global_Schauder_estimate_VariableCoefficients}. In Section \ref{subsec:Regularity}, we prove our global $C^{k, 2+\alpha}_s(\bar S)$ existence result on slabs, $S$, and hence a $C^{k, 2+\alpha}_s(\underline B_{r_0}^+(x^0))$-regularity result, Theorem \ref{thm:Higher_order_regularity}, on half-balls, $B_{r_0}^+(x^0)$. We conclude the section with the proofs of Theorems \ref{thm:InteriorRegularityDomain} and \ref{thm:ExistUniqueCk2+alphasHolderContinuityDomain}, and Corollary \ref{cor:ExistUniqueCk2+alphasHolderContinuityDomain}.

\subsection{A priori Schauder estimates for operators with variable coefficients}
\label{subsec:Variable_coefficients}
We begin with a generalization of Theorem \ref{thm:Holder_estimate_local} to the case of variable coefficients.

\begin{thm}[A priori interior local Schauder estimate when $A$ has variable coefficients]
\label{thm:Holder_estimate_local_variable_coeff}
Let $\alpha \in (0,1)$ and let $b_0$, $\lambda_0$, $\Lambda$, $r_0$ be positive constants. Suppose that the coefficients $a^{ij}$, $b^i$, and $c$ of $A$ in \eqref{eq:defnA} belong to $C^{\alpha}_s(\underline B_{r_0}^+(x^0))$, where $x^0 \in \partial\HH$, and obey
\begin{gather}
\label{eq:Coeff_Holder_continuity}
\|a\|_{C^{\alpha}_s(\bar B_{r_0}^+(x^0))} + \|b\|_{C^{\alpha}_s(\bar B_{r_0}^+(x^0))} + \|c\|_{C^{\alpha}_s(\bar B_{r_0}^+(x^0))} \leq \Lambda,
\\
\label{eq:Coeff_b_d}
b^d \geq b_0 \quad \hbox{on } \partial_0 B_{r_0}^+(x^0),
\\
\label{eq:Strict_ellipticity}
\langle a\xi, \xi\rangle \geq \lambda_0 |\xi|^2\quad\hbox{on } \underline B_{r_0}^+(x^0),\quad \forall\,\xi \in \RR^d,
\end{gather}
Then, for all $r\in(0,r_0)$, there is a positive constant $C=C(\alpha,b_0,d,\lambda_0,\Lambda,r_0,r)$ such that, for any function\footnote{It is enough to require $u \in C^{2+\alpha}_s(\underline B_{r_0}^+(x^0))$ since the estimate trivially holds if $\|A u\|_{C^{\alpha}_s(\bar B_{r_0}^+(x^0))}$ and $\|u\|_{C(\bar B_{r_0}^+(x^0))}$ are not finite.} $u \in C^{2+\alpha}_s(\underline B_{r_0}^+(x^0))$, we have
\begin{equation}
\label{eq:Schauder_estimate_variable_coeff}
\| u\|_{C^{2+\alpha}_s(\bar B^+_{r}(x^0))} \leq C \left(\|A u\|_{C^{\alpha}_s(\bar B_{r_0}^+(x^0))} + \|u\|_{C(\bar B_{r_0}^+(x^0))}\right).
\end{equation}
\end{thm}

\begin{proof}
We use the a priori interior local Schauder estimate \eqref{eq:Holder_estimate_local} for the operator with constant coefficients (given by Theorem \ref{thm:Holder_estimate_local}) and the interpolation inequalities for the H\"older norms defined by the cycloidal metric (Lemma \ref{lem:InterpolationIneqS}),
the method of freezing coefficients as in the proofs of \footnote{This method is also employed in the proof of \cite[Theorem 6.2]{GilbargTrudinger}, but Gilbarg and Trudinger employ a family of `global' interior H\"older norms (which we do not develop in this article) which allows a rearrangement argument.}
\cite[Theorem 7.1.1]{Krylov_LecturesHolder} (elliptic case), \cite[Theorem 8.11.1]{Krylov_LecturesHolder} (parabolic case), and, in particular, \cite[Theorem 3.8]{Feehan_Pop_mimickingdegen_pde} for the parabolic version of our elliptic operator \eqref{eq:defnA} to obtain \eqref{eq:Schauder_estimate_variable_coeff}.
\end{proof}

We now generalize Corollary \ref{cor:Global_Schauder_estimate_ConstantCoefficients} to the case of variable coefficients when $u$ has compact support in a slab.

\begin{prop}[Higher-order a priori global Schauder estimate for compactly supported functions on a slab when $A$ has variable coefficients]
\label{prop:Higher_order_estimate_global_slab_compactsupport}
Let $\alpha\in(0,1)$ and $b_0$, $\lambda_0$, $\Lambda$, $\nu$ be positive constants and $k\in\NN$. Suppose $S=\RR^{d-1}\times(0,\nu)$ as in \eqref{eq:Slab} and the coefficients $a,b,c$ of $A$ in \eqref{eq:defnA} belong to $C^{k,\alpha}_s(\bar S)$ and obey \eqref{eq:Coeff_Holder_continuity_HigherOrderSlab}, \eqref{eq:Strict_ellipticity_slab}, and \eqref{eq:Coeff_b_d_slab}. Then there are positive constants, $C=C(\alpha,b_0,d,k,\lambda_0,\Lambda,\nu)$ and $\delta=\delta(\alpha,b_0,d,k,\lambda_0,\Lambda,\nu)<\nu/2$, such that the following holds. If $u\in C^{k,2+\alpha}_s(\bar S)$ has \emph{compact support} in $\bar S$ with $\diam(\supp u)\leq\delta$ and $u=0$ on $\partial_1 S$, then
\begin{equation}
\label{eq:Higher_order_estimate_global_slab_compactsupport}
\| u\|_{C^{k,2+\alpha}_s(\bar S)} \leq C\left(\|A u\|_{C^{k,\alpha}_s(\bar S)} + \|u\|_{C(\bar S)}\right),
\end{equation}
and, when $c\geq 0$ on $S$,
\begin{equation}
\label{eq:Higher_order_estimate_global_slab_compactsupport_cnonneg}
\| u\|_{C^{k,2+\alpha}_s(\bar S)} \leq C\|A u\|_{C^{k,\alpha}_s(\bar S)}.
\end{equation}
\end{prop}

\begin{proof}
Fix $x^0\in S\cap \supp u$ and let $A_{x^0}$ denote the operator with constant coefficients $a(x^0)$, $b(x^0)$, $c(x^0)$. By applying \eqref{eq:Global_Schauder_estimate_constant_coeff_l} for the operator $A_{x_0}$ with constant coefficients, we obtain
$$
\|u\|_{C^{k, 2+\alpha}_s(\bar S)} \leq C_0\left(\|A_{x_0}u\|_{C^{k, \alpha}_s(\bar S))} + \|u\|_{C(\bar S)}\right),
$$
and hence
\begin{equation}
\label{eq:Ck2+alphasu_diffbound}
\|u\|_{C^{k, 2+\alpha}_s(\bar S)} \leq C_0\left(\|Au\|_{C^{k, \alpha}_s(\bar S)} + \|(A-A_{x_0})u\|_{C^{k, \alpha}_s(\bar S)} + \|u\|_{C(\bar S)}\right),
\end{equation}
where $C_0$ has the dependencies stated for the constant $C$ in the estimate \eqref{eq:Global_Schauder_estimate_constant_coeff_l}.

For any $x^1, x^2 \in \supp u$, the cycloidal distance-function bound \eqref{eq:CycloidLessEuclidDistance} and our hypothesis on $\supp u$ imply that $s(x^1,x^2) \leq |x^1-x^2|^{1/2} \leq \delta^{1/2}$, for some $\delta\in(0,\nu/2)$ to be selected later. We first consider the case $\supp u \subset B_{2\delta}^+(y^0)$ for some $y^0\in\partial_0 S$. We further restrict to the case $k=0$ initially. Observe that
$$
(A-A_{x_0})u = -x_d\tr((a-a(x^0))D^2u) - (b-b(x^0))\cdot Du + (c-c(x^0))u.
$$
We consider in turn each of the three terms appearing in our expression for $(A-A_{x_0})u$. From Definition \ref{defn:DHspaces},
$$
\|(b-b(x^0))\cdot Du\|_{C^{\alpha}_s(\bar S)} = \|(b-b(x^0))\cdot Du\|_{C(\bar S)} + [(b-b(x^0))\cdot Du]_{C^{\alpha}_s(\bar S)}.
$$
The coefficient bounds \eqref{eq:Coeff_Holder_continuity_higherorder} ensure that
$$
\|(b-b(x^0))\cdot Du\|_{C(\bar S)} \leq \left(\|b\|_{C(\bar S)} + |b(x^0)|\right)\|Du\|_{C(\bar S)} \leq 2\Lambda\|Du\|_{C(\bar S)},
$$
while the interpolation inequality \eqref{eq:InterpolationIneqS2} yields, for some $m=m(\alpha,d)$ and $C_1=C_1(\alpha,d,\delta)$ (because $\diam(\supp u) = \delta$) and any $\eps\in(0,1)$,
\begin{equation}
\label{eq:Du_C0_Interpolation}
\|Du\|_{C(\bar S)} \leq \eps\|u\|_{C^{2+\alpha}_s(\bar S)} + C_1\eps^{-m}\|u\|_{C(\bar S)},
\end{equation}
and thus, combining \eqref{eq:Du_C0_Interpolation} with the preceding inequality, yields
\begin{equation}
\label{eq:bbx0Du_C0_Interpolation}
\|(b-b(x^0))\cdot Du\|_{C(\bar S)} \leq 2\eps\Lambda\|u\|_{C^{2+\alpha}_s(\bar S)} + C_1\Lambda\eps^{-m}\|u\|_{C(\bar S)}.
\end{equation}
Writing, for $x^1,\ x^2 \in S\cap \supp u$,
\begin{align*}
{}&\frac{(b(x^1)-b(x^0))\cdot Du(x^1) - (b(x^2)-b(x^0))\cdot Du(x^2)}{s^\alpha(x^1, x^2)}
\\
&= \frac{(b(x^1)-b(x^2))}{s^\alpha(x^1, x^2)}\cdot Du(x^1) + (b(x^1)-b(x^0))\cdot \frac{\left(Du(x^1) - Du(x^2)\right)}{s^\alpha(x^1, x^2)},
\end{align*}
we obtain
$$
[(b-b(x^0))\cdot Du]_{C^{\alpha}_s(\bar S)} \leq [b]_{C^{\alpha}_s(\bar S)}\left(\|Du\|_{C(\bar S)}
+ s^\alpha(x^1, x^0) [Du]_{C^{\alpha}_s(\bar S)}\right).
$$
Since $\diam(\supp u) = \delta$ and $x^0, x^1 \in \supp u$, by combining the preceding inequality with the coefficient bounds \eqref{eq:Coeff_Holder_continuity_higherorder} and the interpolation inequality \eqref{eq:Du_C0_Interpolation}, we see that
\begin{equation}
\label{eq:bbx0Du_Calphaseminorm_Interpolation}
[(b-b(x^0))\cdot Du]_{C^{\alpha}_s(\bar S)} \leq \Lambda\left(\eps\|u\|_{C^{2+\alpha}_s(\bar S)} + C_1\eps^{-m}\|u\|_{C(\bar S)}
+ \delta^{\alpha/2} \|u\|_{C^{2+\alpha}_s(\bar S)}\right).
\end{equation}
Therefore, by combining \eqref{eq:bbx0Du_C0_Interpolation} and \eqref{eq:bbx0Du_Calphaseminorm_Interpolation}, we obtain
\begin{equation}
\label{eq:bbx0Du_Calphanorm_Interpolation}
\|(b-b(x^0))\cdot Du\|_{C^{\alpha}_s(\bar S)} \leq \Lambda(3\eps + \delta^{\alpha/2})\|u\|_{C^{2+\alpha}_s(\bar S)} + 2C_1\Lambda\eps^{-m}\|u\|_{C(\bar S)}.
\end{equation}
An identical analysis, just replacing the coefficient vector $b$ by the matrix $a$, and $Du$ by $x_dD^2u$, and the interpolation inequality \eqref{eq:InterpolationIneqS3} by \eqref{eq:InterpolationIneqS4}, yields
\begin{equation}
\label{eq:aax0xdD2u_Calphanorm_Interpolation}
\|\tr(x_d(a-a(x^0))D^2u)\|_{C^{\alpha}_s(\bar S)} \leq \Lambda(3\eps + \delta^{\alpha/2})\|u\|_{C^{2+\alpha}_s(\bar S)} + 2C_1\Lambda\eps^{-m}\|u\|_{C(\bar S)}.
\end{equation}
Similarly, replacing the coefficient vector $b$ by the function $c$, and $Du$ by $u$, and the interpolation inequality \eqref{eq:InterpolationIneqS3} by \eqref{eq:InterpolationIneqS1}, yields
\begin{equation}
\label{eq:ccx0u_Calphanorm_Interpolation}
\|(c-c(x^0))u\|_{C^{\alpha}_s(\bar S)} \leq \Lambda(3\eps + \delta^{\alpha/2})\|u\|_{C^{2+\alpha}_s(\bar S)} + 2C_1\Lambda\eps^{-m}\|u\|_{C(\bar S)}.
\end{equation}
We combine \eqref{eq:bbx0Du_Calphanorm_Interpolation}, \eqref{eq:aax0xdD2u_Calphanorm_Interpolation}, and \eqref{eq:ccx0u_Calphanorm_Interpolation} to give
$$
\|(A-A_{x_0})u\|_{C^{\alpha}_s(\bar S)} \leq 3\Lambda(3\eps + \delta^{\alpha/2})\|u\|_{C^{2+\alpha}_s(\bar S)} + 6C_1\Lambda\eps^{-m}\|u\|_{C(\bar S)}.
$$
We now choose $\eps>0$ such that $9C_0\Lambda\eps = 1/4$ and choose $\delta\in(0,\nu/2)$ (which we fix for the remainder of the proof) such that $3C_0\Lambda\delta^{\alpha/2} \leq 1/4$ and combine the preceding inequality with \eqref{eq:Ck2+alphasu_diffbound} to give
$$
\|u\|_{C^{2+\alpha}_s(\bar S)} \leq C_0\|Au\|_{C^{\alpha}_s(\bar S)} + \frac{1}{2}\|u\|_{C^{2+\alpha}_s(\bar S)} + C_2\|u\|_{C(\bar S)},
$$
for some constant $C_2$ with at most the dependencies stated for $C$ in our hypotheses. Rearrangement and the maximum principle estimate (Corollary \ref{cor:Maximum_principle_A_0}) for $\|u\|_{C(\bar S)}$ now give the conclusions \eqref{eq:Higher_order_estimate_global_slab_compactsupport} and \eqref{eq:Higher_order_estimate_global_slab_compactsupport_cnonneg} when $c\geq 0$ on $S$ in the case $k=0$.

Next, suppose $k\geq 1$ and let $\beta\in\NN^d$ be a multi-index with $|\beta|\leq k$. Because
$$
D^\beta (v w) = \sum_{\begin{subarray}{c}\beta'+\beta''=\beta \\ \beta', \beta''\in\NN^d \end{subarray}} D^{\beta'}v D^{\beta''}w.
$$
for any $v, w\in C^{k,\alpha}_s(\bar S)$, we may apply the preceding analysis virtually unchanged with $v = a-a(x^0)$, or $b-b(x^0)$, or $c-c(x^0)$ and $w=x_dD^2u$, or $Du$, or $u$, respectively, for each $\beta,\beta',\beta''\in\NN^d$ with $|\beta|\leq k$ and $\beta'+\beta''=\beta$. This completes the proof when $\supp u \subset B_\delta^+(y^0)$ for some $y^0\in\partial_0 S$.

Because $\supp u \subset B_{\delta/2}(x^*)$, for some $x^*\in\bar S$, the case $\dist(x^*, \partial_0 S)\leq\delta$ is covered by our analysis for half-balls, $B_{2\delta}^+(y^0)$, with $y^0\in\partial_0 S$. If $\dist(x^*, \partial_0 S)\geq\delta/2$, then the operator $A$ is strictly elliptic since $x_d \geq \delta/2$ and \cite[Theorem 6.6 and Problem 6.2]{GilbargTrudinger} imply that
$$
\|u\|_{C^{k, 2+\alpha}_s(\bar S)} \leq C_0\left(\|Au\|_{C^{k, \alpha}_s(\bar S))}+\|u\|_{C(\bar S)}\right),
$$
which is just \eqref{eq:Higher_order_estimate_global_slab_compactsupport}. Combining the preceding inequality with the maximum principle estimate (Corollary \ref{cor:Maximum_principle_A_0}) for $\|u\|_{C(\bar S)}$ again gives the conclusion \eqref{eq:Higher_order_estimate_global_slab_compactsupport_cnonneg} when $c\geq 0$ on $S$.
\end{proof}

Finally, we use Proposition \ref{prop:Higher_order_estimate_global_slab_compactsupport} to generalize Theorem \ref{thm:Holder_estimate_local_higherorder} to the case of variable coefficients to obtain the following analogue of \cite[Theorem I.1.3]{DaskalHamilton1998} (for a related boundary-degenerate parabolic operator with constant coefficients) and \cite[Corollary 6.3 and Problem 6.1]{GilbargTrudinger}.

\begin{thm}[Higher-order a priori interior local Schauder estimate when $A$ has variable coefficients]
\label{thm:Higher_order_estimate}
Let $\alpha\in(0,1)$ and let $b_0$, $\lambda_0$, $\Lambda$, $r_0$ be positive constants and let $k\in\NN$. Suppose the coefficients $a,b,c$ of $A$ in \eqref{eq:defnA} belong to $C^{k,\alpha}_s(\underline B_{r_0}^+(x^0))$, where $x^0 \in \partial\HH$, and obey \eqref{eq:Coeff_b_d}, \eqref{eq:Strict_ellipticity}, and
\begin{equation}
\label{eq:Coeff_Holder_continuity_higherorder}
\|a\|_{C^{k,\alpha}_s(\bar B_{r_0}^+(x^0))} + \|b\|_{C^{k,\alpha}_s(\bar B_{r_0}^+(x^0))} + \|c\|_{C^{k,\alpha}_s(\bar B_{r_0}^+(x^0))} \leq \Lambda.
\end{equation}
Then, for any $ r\in (0, r_0)$, there is a positive constant, $C=C(\alpha,b_0,d,k,\lambda_0,\Lambda,r_0,r)$, such that the following holds. If $u\in C^{k,2+\alpha}_s(\underline B_{r_0}^+(x^0))$ then
\begin{equation}
\label{eq:Higher_order_estimate}
\| u\|_{C^{k,2+\alpha}_s(\bar B^+_r(x^0))} \leq C \left(\|A u\|_{C^{k,\alpha}_s(\bar B_{r_0}^+(x^0))} + \|u\|_{C(\bar B_{r_0}^+(x^0))}\right).
\end{equation}
\end{thm}

\begin{proof}
We apply an induction argument. When $k=0$, the estimate \eqref{eq:Higher_order_estimate} follows from Theorem \ref{thm:Holder_estimate_local_variable_coeff} and so we may assume without loss of generality that $k \geq 1$. By induction, we may assume that the estimate \eqref{eq:Higher_order_estimate} holds with constant $C=C(l,*)$ when $k$ is replaced by $l\in\NN$ in the range $0\leq l\leq k-1$.

Let $r_1:=(r+r_0)/2$ and choose a cutoff function $\varphi \in C^\infty_0(\bar\HH)$ such that $0\leq\varphi\leq 1$ on $\bar\HH$ and $\varphi=1$ on $\bar B_r^+(x^0)$ while $\supp\varphi\subset \bar B_{r_1}^+(x^0)$ and note that, for $u \in C^{k, 2+\alpha}_s(\underline B_{r_0}^+(x^0))$ and thus $u_0 := \varphi u\in C^{k, 2+\alpha}_s(\bar S)$, we have
$$
\|u\|_{C^{k, 2+\alpha}_s(\bar B_r^+(x^0))} \leq \|u_0\|_{C^{k, 2+\alpha}_s(\bar S)},
$$
where $S = \RR^{d-1}\times(0, r_0)$ is the slab as in \eqref{eq:Slab}, and $u_0 = 0$ on $\partial_1 S$. By Proposition \ref{prop:Higher_order_estimate_global_slab_compactsupport}, we obtain
$$
\|u_0\|_{C^{k, 2+\alpha}_s(\bar S)} \leq C_0\left(\|Au_0\|_{C^{k, \alpha}_s(\bar S)} +  \|u_0\|_{C(\bar S)}\right),
$$
where we use $C_0$ to denote the constant $C$ in \eqref{eq:Higher_order_estimate_global_slab_compactsupport}, and hence, combining the preceding two inequalities,
\begin{equation}
\label{eq:Higher_order_estimate_Acutoff}
\|u\|_{C^{k, 2+\alpha}_s(\bar B_r^+(x^0))} \leq C_0\left(\|A(\varphi u)\|_{C^{k, \alpha}_s(\bar B_{r_1}^+(x^0))} + \|u\|_{C(\bar B_{r_0}^+(x^0))} \right).
\end{equation}
Notice that $A(\varphi u) = \varphi Au + [A, \varphi]u$ and that $[A, \varphi]$ is a first-order partial differential operator,
\begin{align*}
[A, \varphi]u &= A(\varphi u) - \varphi Au
\\
&= -\tr(x_d aD^2(\varphi u)) - b\cdot D(\varphi u) +
\varphi\tr(x_d aD^2 u) + \varphi b\cdot D u,
\end{align*}
and so
\begin{equation}
\label{eq:CommutatorAcutoff}
[A, \varphi]u = -\tr(x_d a((D^2\varphi)u + D\varphi\times Du)) - (b\cdot D\varphi)u,
\end{equation}
where $D\varphi\times Du$ denotes the $d\times d$ matrix with entries $\varphi_{x_i}u_{x_j}$. Observe that
$$
\|A(\varphi u)\|_{C^{k, \alpha}_s(\bar B_{r_1}^+(x^0))} \leq \|[A,\varphi] u\|_{C^{k, \alpha}_s(\bar B_{r_1}^+(x^0))} + \|\varphi Au\|_{C^{k, \alpha}_s(\bar B_{r_1}^+(x^0))}.
$$
Because of the structure \eqref{eq:CommutatorAcutoff} of $[A, \varphi]$ (with factor $x_d$ in the  coefficients of the first-order derivatives) and the fact that $C^{k, \alpha}_s(\bar B_{r_1}^+(x^0)) \subset C^{k-1, 2+\alpha}_s(\bar B_{r_1}^+(x^0))$ (by Definitions \ref{defn:DHspaces} and \ref{defn:DH2spaces}), we obtain
$$
\|[A, \varphi]u\|_{C^{k, \alpha}_s(\bar B_{r_1}^+(x^0))} \leq C\|u\|_{C^{k-1, 2+\alpha}_s(\bar B_{r_1}^+(x^0))},
$$
where $C$ has at most the dependencies stated for the constant in the estimate \eqref{eq:Higher_order_estimate}. By our induction hypothesis, we can apply the local Schauder estimate \eqref{eq:Higher_order_estimate} with $k$ replaced by $l=k-1$ to give
$$
\|u\|_{C^{k-1, 2+\alpha}_s(\bar B_{r_1}^+(x^0))} \leq C\left(\|Au\|_{C^{k-1, \alpha}_s(\bar B_{r_0}^+(x^0))} + \|u\|_{C(\bar B_{r_0}^+(x^0))}\right).
$$
Combining the preceding three bounds with \eqref{eq:Higher_order_estimate_Acutoff} yields the inequality,
\begin{equation}
\label{eq:Higher_order_estimate_smallenoughhalfball}
\|u\|_{C^{k, 2+\alpha}_s(\bar B_r^+(x^0))} \leq C\left(\|Au\|_{C^{k, \alpha}_s(\bar B_{r_0}^+(x^0))} + \|u\|_{C(\bar B_{r_0}^+(x^0))}\right),
\end{equation}
and this is \eqref{eq:Higher_order_estimate}.
\end{proof}

We can now prove the generalization of Corollary \ref{cor:Global_Schauder_estimate_ConstantCoefficients} to the case of variable coefficients.

\begin{proof}[Proof of Corollary \ref{cor:Global_Schauder_estimate_VariableCoefficients}]
The proof is virtually identical to that of Corollary \ref{cor:Global_Schauder_estimate_ConstantCoefficients} except that we replace appeals to Theorems \ref{thm:Holder_estimate_local} and \ref{thm:Holder_estimate_local_higherorder} (for the case of constant coefficients) by appeals to Theorems \ref{thm:Holder_estimate_local_variable_coeff} and \ref{thm:Higher_order_estimate} (for the case of variable coefficients).
\end{proof}

We can now give the

\begin{proof}[Proof of Theorem \ref{thm:APrioriSchauderInteriorDomain}]
Since we can apply Theorem \ref{thm:Higher_order_estimate} to half-balls, $B_{r_0}^+(x^0) = B_{r_0}(x^0)\cap \HH \subset \sO$ when $x_0\in\partial\sO$, and the standard a priori interior Schauder estimate for strictly elliptic operators \cite[Corollary 6.3 and Problem 6.1]{GilbargTrudinger} to balls, $B_{r_0}(x^0) \Subset \sO$ when $x_0\in\sO$, the remainder of the argument is very similar to the proof of Corollary \ref{cor:Global_Schauder_estimate_ConstantCoefficients}, when the open subset, $\sO$, is an infinite slab.
\end{proof}

\subsection{Regularity}
\label{subsec:Regularity}
We begin with the proof of our global existence result on slabs.

\begin{proof}[Proof of Theorem \ref{thm:Existence_uniqueness_slabs}]
The proof follows by the method of continuity. We denote
$$
\sA_0 := -x_d \sum_{i=1}^d \frac{\partial^2}{\partial x_i^2} - \frac{\partial}{\partial x_d}.
$$
Then Corollary \ref{cor:Existence_solutions_slab} implies that, for any $f \in C^{k, \alpha}_s(\bar S)$, there is a unique solution $u \in C^{k, 2+\alpha}_s(\bar S)$. We consider $A_t := (1-t) \sA_0 + t A$, for all $t \in [0,1]$. Given Corollary \ref{cor:Global_Schauder_estimate_VariableCoefficients} and the existence and uniqueness of solutions in $C^{k, 2+\alpha}_s(\bar S)$ for the operator $\sA_0$, the method of continuity \cite[Theorem 5.2]{GilbargTrudinger} applies and gives the result.
\end{proof}

We have the following analogue of \cite[Theorem 6.17]{GilbargTrudinger}, albeit with a quite different proof.

\begin{thm}[Higher-order interior local regularity of solutions when $A$ has variable coefficients]
\label{thm:Higher_order_regularity}
Assume the hypotheses of Theorem \ref{thm:Higher_order_estimate} for the operator $A$ in \eqref{eq:defnA}. If $u\in C^2(B_{r_0}^+(x^0))$ obeys
\begin{gather}
\label{eq:CondRegularityHalfball}
u, \ Du,\ x_d D^2 u \in C(\underline B^+_{r_0}(x^0)) \quad\hbox{and}\quad Au \in C^{k, \alpha}_s(\underline B^+_{r_0}(x^0)),
\\
\label{eq:VentcelHalfball}
x_d D^2 u =0 \quad\hbox{on } \partial_0 B^+_{r_0}(x^0),
\end{gather}
then $u \in C^{k, 2+\alpha}_s(\underline B_{r_0}^+(x^0))$.
\end{thm}

\begin{proof}
Let $r\in(0,r_0)$, and $r_1:=(r+r_0)/2$, and $r_2:=(r_1+r)/2$, so $r<r_1<r_2<r_0$. Let $\varphi\in C^\infty_0(\bar\HH)$ be a cutoff function such that $0\leq \varphi\leq 1$ on $\HH$ with $\varphi = 1$ on $\bar B^+_{r_1}(x^0)$
and $\supp \varphi \subset \bar B^+_{r_2}(x^0)$. Denoting $S=\RR^{d-1}\times (0,r_0)$ as in \eqref{eq:Slab} and $u_0 := u \varphi$ on $\bar S$, we see that $u_0 \in C^2(S)$ is a solution to \eqref{eq:Equation_slab}, \eqref{eq:Equation_slabBC} with $f$ replaced by
$$
f_0 := \varphi A u + [A,\varphi] u \quad\hbox{on } \bar S.
$$
By hypothesis, $\varphi Au \in C^{k, \alpha}_s(\underline B_{r_0}^+(x^0))$, while Lemma \ref{lem:Holder_continuity_x_d_Du} and \eqref{eq:CondRegularityHalfball} and \eqref{eq:VentcelHalfball} ensure that $[A,\varphi] u\in C^{\alpha}_s(\bar S)$, so
$$
f_0\in C^{\alpha}_s(\bar S),
$$
while the conditions \eqref{eq:CondRegularityHalfball} and \eqref{eq:VentcelHalfball} on $u$ imply that $u_0$ obeys
$$
u_0, \ Du_0, \ x_dD^2u_0 \in C(\underline S) \quad\hbox{and}\quad x_dD^2u_0 = 0 \quad\hbox{on } \partial_0 S.
$$
Corollary \ref{cor:Existence_solutions_slab} implies that there is a unique solution $v \in C^{2+\alpha}_s(\bar S)$ to \eqref{eq:Equation_slab}, \eqref{eq:Equation_slabBC} with $f$ replaced by $f_0$ and the weak maximum principle, Lemma \ref{lem:Comparison_principle_A_0}, implies that $u_0=v$ on $\bar S$. Thus, $u \in C^{2+\alpha}_s(\bar B^+_r(x^0))$.

When $k\geq 1$, we argue by induction and suppose that $u \in C^{k-1, 2+\alpha}_s(\bar B^+_r(x^0))$ as our induction hypothesis. But then $[A,\varphi] u \in C^{k, \alpha}_s(\bar B^+_r(x^0))$ by the proof of Theorem \ref{thm:Higher_order_estimate} and so $f_0\in C^{k, \alpha}_s(\bar S)$. Now Corollary \ref{cor:Existence_solutions_slab} implies that $v \in C^{k, 2+\alpha}_s(\bar S)$ by the preceding argument for $k=0$, and thus $u \in C^{k, 2+\alpha}_s(\bar B^+_r(x^0))$ since $u_0=v$ on $\bar S$.
\end{proof}

We can now complete the

\begin{proof}[Proof of Theorem \ref{thm:InteriorRegularityDomain}]
This is an immediate consequence of Theorem \ref{thm:Higher_order_regularity} and \cite[Theorem 6.17]{GilbargTrudinger} since we can apply those regularity results to any half-ball, $B_{r_0}^+(x^0) = B_{r_0}(x^0)\cap\HH \Subset \underline\sO$ when $x^0\in\partial\HH$, or ball, $B_{r_0}(x^0) \Subset \sO$ when $x^0\in\HH$, respectively.
\end{proof}

Finally, we complete the proofs of Theorem \ref{thm:ExistUniqueCk2+alphasHolderContinuityDomain} and Corollary \ref{cor:ExistUniqueCk2+alphasHolderContinuityDomain}.

\begin{proof}[Proof of Theorem \ref{thm:ExistUniqueCk2+alphasHolderContinuityDomain}]
The argument is very similar to the proof of Corollary \ref{cor:Existence_solutions_slab}, so we just highlight the differences. Because $f\in C^{k, \alpha}_s(\underline\sO)\cap C_b(\sO)$, we can apply the regularizing procedure
described in \cite[\S I.11]{DaskalHamilton1998} and \cite[Theorem I.11.3]{DaskalHamilton1998} to construct a sequence of functions $\{f_n\}_{n \in \NN} \subset C^{\infty}_0(\bar\HH)$ such that $f_n\to f$ in $C^{k, \alpha}_s(\bar U)\cap C_b(\sO)$ as $n\to\infty$, for all $U\Subset\underline\sO$, and
$$
\|f_n\|_{C^{k, \alpha}_s(\bar U')} \leq C'\|f\|_{C^{k, \alpha}_s(\bar U)}, \quad\forall\, n\in\NN,
$$
where $U'\Subset \underline U$ and $U\Subset \underline\sO$ and $C'$ may depend on $U$ and $U'$, and
$$
\|f_n\|_{C(\bar\sO)} \leq C\|f\|_{C(\bar\sO)},\quad\forall\, n \in \NN,
$$
for some positive constant, $C=C(d)$.

Let $\{u_n\}_{n\in\NN}\subset C^\infty(\underline\sO)\cap C_b(\sO\cup\partial_1\sO)$ be the corresponding (unique) sequence of solutions to \eqref{eq:HestonBoundaryValueProblem}, \eqref{eq:HestonBoundaryValueProblemBC}, with $f$ replaced by $f_n$, provided by \cite[Theorem 1.11]{Feehan_Pop_higherregularityweaksoln}. The maximum principle estimate (Corollary \ref{cor:Maximum_principle_A_0} for the case $c_0=0$ and \cite[Proposition 2.19 and Theorem 5.4]{Feehan_maximumprinciple} for the case $c_0>0$) implies that
$$
\|u_n\|_{C(\bar\sO)} \leq C_0\|f_n\|_{C(\bar\sO)}, \quad\forall\, n\in\NN,
$$
for a constant $C_0$ depending on the coefficients of $A$ and $\nu$ when $\height(\sO)=\nu$ and $c_0=0$ or $C_0=1/c_0$ when $c_0>0$ and $\height(\sO)=\infty$. The remainder of the argument is now the same as that of the proof of Corollary \ref{cor:Existence_solutions_slab}.
\end{proof}

\begin{proof}[Proof of Corollary \ref{cor:ExistUniqueCk2+alphasHolderContinuityDomain}]
The conclusion follows immediately from Theorem \ref{thm:ExistUniqueCk2+alphasHolderContinuityDomain} and \cite[Corollary 1.13]{Feehan_Pop_higherregularityweaksoln}, since the latter result ensures that $u\in C(\bar\sO)$.
\end{proof}

\appendix

\section{Weak maximum principle for boundary-degenerate elliptic operators on open subsets of finite height}
\label{sec:MaximumPrinciple}
In this appendix, we prove a weak maximum principle for operators which include those of the form $A$ in \eqref{eq:defnA} with $c \geq 0$ when the open subset, $\sO$, is unbounded.  Notice that when $c$ does not have a uniform positive lower bound, the weak maximum principle \cite[Theorem 5.4]{Feehan_maximumprinciple} does not immediately apply when $\sO$ is unbounded.

\begin{lem}[Weak maximum principle on a slab]
\label{lem:Comparison_principle_A_0}
Let $\sO\subset\HH$ be an open subset of finite height. Let\footnote{Note the more general definition of the coefficient $a(x)$ in Lemma \ref{lem:Comparison_principle_A_0} and Corollary \ref{cor:Maximum_principle_A_0}.}
$$
Av := -\tr(aD^2v) - b\cdot Dv + cv \quad\hbox{on }\sO, \quad v\in C^\infty(\sO),
$$
require that its coefficients, $a:\underline\sO\to\sS^+(d)$, and $b:\underline\sO\to\RR^d$, and $c:\underline\sO\to\RR$ obey
\begin{align*}
a(x) &=0 \quad\hbox{on }\partial_0\sO,
\\
\langle a\xi, \xi\rangle &\geq 0  \quad\hbox{on }\sO, \quad\forall\, \xi\in\RR^d,
\\
\sup_\sO a^{dd} &< \infty,
\\
\inf_\sO b^d &> 0,
\\
b^d &\geq 0 \quad\hbox{on } \partial_0\sO,
\\
c &\geq 0  \quad\hbox{on }\underline\sO,
\\
\tr(a(x)) + \langle x, b(x)\rangle &\leq K(1+|x|^2), \quad\forall\, x\in \sO,
\end{align*}
for some positive constant $K$. Suppose that $u \in C^2(\sO)\cap C(\sO\cup\partial_1\sO)$, and $\sup_\sO u <\infty$,
and $u, \ Du, \ \tr(aD^2u) \in C(\underline \sO)$, and
\[
\tr(aD^2u)=0 \quad\hbox{on }\partial_0 \sO.
\]
If $A u \leq 0$ on $\sO$ and $u \leq 0$ on $\partial_1 \sO$, then $u \leq 0$ on $\sO$.
\end{lem}

\begin{proof}
Define constants $b_0>0$ and $\Lambda > 0$ by
\begin{equation}
\label{eq:MaxPrincipleConstants}
\Lambda := \sup_\sO a^{dd} \quad\hbox{and}\quad b_0 := \inf_\sO b^d.
\end{equation}
Let $\sigma$ be a positive constant, to be fixed shortly, and define  $v\in C^2(\sO)\cap C(\sO\cup\partial_1\sO)$ with $\sup_\sO v <\infty$ by the transformation
\begin{equation}
\label{eq:Exponential_transformation}
u(x',x_d) = e^{-\sigma x_d} v(x',x_d), \quad\forall\, (x', x_d) \in \bar \sO,
\end{equation}
noting that $\sup_\sO v <\infty$ since $\sup_\sO u <\infty$ and $\height(\sO)<\infty$ by hypothesis. By direct calculation, we find that
$$
A u = e^{-\sigma x_d} \left(-a^{ij} v_{x_ix_j} -\left(b^i-2\sigma a^{id}\right) v_{x_i} +\left(c+\sigma b^d - \sigma^2 a^{dd}\right)v\right).
$$
We now define coefficients $\tilde a, \tilde b, \tilde c$ of an operator $\widetilde A$ by
\begin{align*}
\widetilde A v &:= -\tilde a^{ij} v_{x_ix_j} - \tilde b^i v_{x_i} + \tilde c v
\\
&:= -a^{ij} v_{x_ix_j} -\left(b^i-2\sigma a^{id}\right) v_{x_i} +\left(c+\sigma b^d - \sigma^2 a^{dd}\right)v,
\end{align*}
and we notice, by our hypotheses on $u$ and definition of $v$, that $\widetilde A v \leq 0$ on $\sO$ and $v \leq 0$ on $\partial_1 \sO$. Since $a=0$ and $b^d\geq 0$ on $\partial_0 \sO$ by hypothesis, we have
$$
\tilde c \equiv c+\sigma b^d - \sigma^2 a^{dd} = c+\sigma b^d \geq 0 \quad\hbox{on } \partial_0 \sO.
$$
We now choose $\sigma := b_0/(2\Lambda)$, so that, using $c \geq 0$ and $a^{dd}\leq\Lambda$ on $\sO$, then
$$
\tilde c \geq \sigma b_0 - \sigma^2\Lambda = \sigma b_0/2  > 0 \quad\hbox{on } \sO.
$$
Thus \cite[Theorem 5.4]{Feehan_maximumprinciple} applies to $v$, and now the conclusion follows immediately for $u$ also.
\end{proof}

\begin{cor}[Maximum principle estimate]
\label{cor:Maximum_principle_A_0}
Let $\nu>0$ and let $\sO\subseteqq\RR^{d-1}\times(0,\nu)$ be an open subset, let $A$ be as in Lemma \ref{lem:Comparison_principle_A_0}, and let $f \in C_b(\sO)$, and $g \in C_b(\partial_1 \sO)$. If $u$ obeys the regularity properties on $\bar \sO$ in the hypotheses of Lemma \ref{lem:Comparison_principle_A_0} and
\begin{align*}
A u &=f \quad\hbox{on } \sO,
\\
u &=g \quad\hbox{on } \partial_1 \sO,
\end{align*}
then there is a positive constant, $C=C(b_0,\Lambda,\nu)$, with $b_0, \Lambda$ as in \eqref{eq:MaxPrincipleConstants}, such that
\begin{equation*}
\|u\|_{C(\bar \sO)} \leq C\left(\|f\|_{C(\bar \sO)} + \|g\|_{C(\overline{\partial_1 \sO})}\right).
\end{equation*}
\end{cor}

\begin{proof}
We define $v = e^{\sigma x_d}u$ as in \eqref{eq:Exponential_transformation}, where $\sigma$ is chosen as in the proof of Lemma \ref{lem:Comparison_principle_A_0}. Then,
\begin{align*}
\widetilde A v &= \tilde f \quad\hbox{on } \sO,
\\
v &= \tilde g \quad\hbox{on } \partial_1 \sO,
\end{align*}
where $\tilde f := e^{\sigma x_d}f$ on $\sO$ and $\tilde g := e^{\sigma x_d} g$ on $\partial_1 \sO$. Because $\tilde c \geq b_0^2/(4\Lambda) > 0$ on $\sO$ from the proof of Lemma \ref{lem:Comparison_principle_A_0}, we can apply \cite[Proposition 2.19]{Feehan_maximumprinciple} to give
$$
\|v\|_{C(\bar \sO)} \leq \frac{1}{\tilde c}\left(\|\tilde f\|_{C(\bar \sO)} + \|\tilde g\|_{C(\overline{\partial_1\sO})}\right).
$$
The conclusion follows since $x_d\in[0, \nu]$ for all $x\in\bar\sO$ and we can take $C := 4\Lambda e^{\sigma \nu}/b_0^2$.
\end{proof}

\section{Existence of solutions for boundary-degenerate elliptic operators with constant coefficients on half-spaces and slabs}
\label{sec:Existence_smooth_solution}
In this section, we prove existence of smooth solutions to $Au=f$ on the half-space $\HH$ or on slabs $S=\RR^{d-1}\times(0,\nu)$ as in \eqref{eq:Slab}, for some $\nu>0$, when the source function $f$ is assumed to be smooth with compact support in $\bar \HH$ or in $\RR^{d-1}\times[0,\nu)$, respectively, and under the assumption of Hypothesis \ref{hyp:ConstantCoefficients}, that \emph{the coefficients, $a,b,c$, of the operator $A$ in \eqref{eq:defnA} and so the coefficients, $a,b$, of the operator $A_0$ in \eqref{eq:defnA_0} are constant}. The method of the proof is similar to that of \cite[Theorem I.1.2]{DaskalHamilton1998} and it is based on taking the Fourier transform in the first $(d-1)$-variables. The problem is then reduced to the study of the Kummer ordinary differential equations whose solutions can be expressed in terms of the confluent hypergeometric functions, $M$ and $U$ \cite[\S 13]{AbramStegun}.

We begin by reviewing the properties of the confluent hypergeometric functions which will be used in the proofs of Theorems \ref{thm:Existence_smooth_solutions_half_space} and \ref{thm:Existence_smooth_solutions_slab}.

\begin{lem}[Properties of the confluent hypergeometric functions]
\label{lem:ConfluentHypergeometricFunctionProperties}
\cite{AbramStegun}
Let $a \in \CC$ be such that its real part is positive, $\Re(a)>0$, and let $b$ be a positive constant. Then the following holds, for all $y>0$.
\begin{enumerate}
\item Asymptotic behavior as $y\rightarrow +\infty$:
\begin{align}
\label{eq:Asymptotics_M}
M(a,b,y) &= \frac{\Gamma(b)}{\Gamma(a)} y^{a-b} e^y\left(1+O(y^{-1})\right) \quad\hbox{\cite[\S\,13.1.4]{AbramStegun},}
\\
\label{eq:Asymptotics_U}
U(a,b,y) &= y^{-a}\left(1+O(y^{-1})\right) \quad\hbox{\cite[\S\,13.1.8]{AbramStegun}.}
\end{align}
\item Asymptotic behavior as $y\rightarrow 0$:
\begin{align}
\label{eq:Value_at_0_M}
 M(a,b,0)=1\quad\hbox{\cite[\S\,13.1.2]{AbramStegun}}
 \quad\hbox{and}\quad
 M(a,b,y)=1+O(y) \quad\hbox{\cite[\S\,13.5.5]{AbramStegun},}
\end{align}
and
\begin{equation}
\label{eq:Value_at_0_U}
\begin{aligned}
U(a,b,y) &= \frac{\Gamma(b-1)}{\Gamma(a)} y^{1-b} + O(y^{b-2}) \quad\hbox{if $b>2$,} \quad\hbox{\cite[\S\,13.5.6]{AbramStegun},}
\\
U(a,b,y) &= \frac{\Gamma(b-1)}{\Gamma(a)} y^{1-b} + O(|\log y|) \quad\hbox{if $b=2$,} \quad\hbox{\cite[\S\,13.5.7]{AbramStegun},}
\\
U(a,b,y) &= \frac{\Gamma(b-1)}{\Gamma(a)} y^{1-b} + O(1) \quad\hbox{if $1<b<2$,} \quad\hbox{\cite[\S\,13.5.8]{AbramStegun},}
\\
U(a,b,y) &= -\frac{1}{\Gamma(a)} (\log y + \psi(a)+2\gamma) + O(y|\log y|) \quad\hbox{if $b=1$,} \quad\hbox{\cite[\S\,13.5.9]{AbramStegun},}
\\
U(a,b,y) &= \frac{\Gamma(1-b)}{\Gamma(1+a-b)} + O(y^{1-b}) \quad\hbox{if $0<b<1$,} \quad\hbox{\cite[\S\,13.5.10]{AbramStegun},}
\\
U(a,b,0) &= \frac{\Gamma(1-b)}{\Gamma(1+a-b)} \quad\hbox{if $0<b<1$,} \quad\hbox{\cite[\S\,13.1.2 and 13.1.3]{AbramStegun},}
\end{aligned}
\end{equation}
where $\psi(a)=\Gamma'(a)/\Gamma(a)$ and $\gamma\in\RR$ is Euler's constant \cite[\S\,6.1.3]{AbramStegun}.
\item Differential properties:
\begin{align}
\label{eq:Derivative_M}
M'(a,b,y) &= \frac{a}{b} M(a+1,b+1,y) \quad\hbox{\cite[\S\,13.4.8]{AbramStegun},}
\\
\label{eq:Derivative_U}
U'(a,b,y) &= -a U(a+1,b+1,y) \quad\hbox{\cite[\S\,13.4.21]{AbramStegun}}.
\end{align}
\item Recurrence relations:
\begin{align}
\label{eq:Useful_identity_M}
(b-1)M(a-1,b-1,y) &= (b-1-y) M(a,b,y) + y M'(a,b,y) \quad\hbox{\cite[\S\,13.4.14]{AbramStegun},}
\\
\label{eq:Useful_identity_U}
U(a-1,b-1,y) &= (1-b+y) U(a,b,y) - y U'(a,b,y) \quad\hbox{\cite[\S\,13.4.27]{AbramStegun}.}
\end{align}
\end{enumerate}
\end{lem}

We can now give the

\begin{proof}[Proof of Theorem \ref{thm:Existence_smooth_solutions_half_space}]
Uniqueness of the solution, $u \in C^{\infty}(\bar\HH)$, follows from the weak maximum principle \cite[Theorem 5.4]{Feehan_maximumprinciple}.
By simple changes of variables  described in the proof of \cite[Proposition A.1]{Feehan_Pop_mimickingdegen_pde}, which leave invariant any slab of the form $\RR^{d-1}\times (0, \nu)$, for $\nu>0$, we may assume without loss of generality that $a^{ij}=\delta^{ij}$ in \eqref{eq:defnA} and so the differential operator $A$ has the form
$$
A v = -x_d v_{x_ix_i} - b^i v_{x_i} +cv, \quad\forall\, v \in C^2(\HH),
$$
where $b^d$ and $c$ are again positive constants.

We adapt the method of the proof of \cite[Theorem I.1.2]{DaskalHamilton1998}. We fix $f\in C^{\infty}_0(\bar\HH)$.
If $u\in C^{\infty}(\bar\HH)$ is a solution to $Au=f$ on $\HH$, then we expect that its Fourier transform in the $x'=(x_1,\ldots,x_{d-1})$-variables,
\begin{equation*}
\tilde u (\xi;x_d):=\frac{1}{(2\pi)^{d/2}}\int_{\RR^{d-1}}u(x',x_d)e^{-ix'\xi} \,dx', \quad\forall\,\xi\in\RR^{d-1},
\end{equation*}
is a solution, for each $\xi\in\RR^{d-1}$, to the ordinary differential equation,
\begin{equation}
\label{eq:ODE_tilde_u}
-x_d \tilde u_{x_dx_d}(\xi;x_d)-b^d\tilde u_{x_d}(\xi;x_d) + \left(c+i \sum_{k=1}^{d-1} b^k\xi_k + |\xi|^2 x_d\right)\tilde u(\xi;x_d)
= \tilde f(\xi;x_d),
\end{equation}
for all $x_d \in (0,\infty)$, where $\tilde f(\xi;x_d)$ is the Fourier transform of $f(x',x_d)$ with respect to $x'\in\RR^d$.
We show that the ordinary differential equation \eqref{eq:ODE_tilde_u} has a smooth enough solution, $\tilde u$, in a sense to be specified, such that its inverse Fourier transform,
\begin{equation}
\label{eq:Inverse_Fourier_transform_tilde_u}
u(x',x_d):=\frac{1}{(2\pi)^{d/2}}\int_{\RR^{d-1}}\tilde u(\xi;x_d)e^{ix'\xi}\,d\xi,
\end{equation}
is a $C^{\infty}(\bar\HH)$ solution to the equation $Au=f$ on $\HH$.

Defining the function $v(\xi;y)$, for $y=2|\xi|x_d$ and each $\xi\in\RR^{d-1}\less\{0\}$, by
\begin{equation}
\label{eq:Definition_tilde_u}
\tilde u(\xi;x_d)=:e^{-|\xi|x_d}v(\xi;2|\xi|x_d),\quad\forall\, \xi\in\RR^{d-1}\less\{0\}, \quad\forall\, x_d \in \RR_+,
\end{equation}
we see that $v$ is a solution to the Kummer ordinary differential equation,
\begin{equation}
\label{eq:Kummer_ODE}
-y v_{yy}(\xi;y)-(b-y)v_y(\xi;y) +a(\xi)v(\xi;y)=g(\xi;y),\quad\forall\, y\in \RR_+,
\end{equation}
where we denote
\begin{equation}
\label{eq:Coeff_Kummer_ODE}
\begin{aligned}
b         &:= \frac{b^d}{2},\\
a(\xi)    &:= \frac{1} {2|\xi|}\left(c+b^d |\xi|+i\sum_{k=1}^{d-1}b^k \xi_k\right),\\
g(\xi; y) &:= \frac{e^{y/2}}{2|\xi|} \tilde f\left(\xi; \frac{y}{2|\xi|}\right).
\end{aligned}
\end{equation}
Because $b^d>0$ and $c > 0$ by hypothesis, we see that $b>0$ and $\Re (a(\xi))> 0$ when $\xi\neq0$. Since $f$ has compact support in $\bar \HH$, the function $g(\xi;\cdot)$ also has compact support in $\bar\RR_+$.

It suffices to study the solutions, $v(\xi;\cdot)$, to the Kummer equations for $\xi\in\RR^{d-1}\less\{0\}$, and so without loss of generality, we will assume in the sequel that $\xi\neq 0$. The remainder of the proof of Theorem \ref{thm:Existence_smooth_solutions_half_space} is completed in two steps.

\setcounter{step}{0}
\begin{step}[Solution to the Kummer ordinary differential equation]
\label{step:Solution_Kummer_ODE}
The general solution to the Kummer ordinary differential equation \eqref{eq:Kummer_ODE} can be written in the form $v = v^h + v^p$, where
\begin{align*}
v^h(\xi;y) &:= c_1 M(a(\xi),b;y) + c_2 U(a(\xi),b;y),\\
v^p(\xi;y) &:= -M(a(\xi),b;y) \int_y^{\infty} g(\xi;z) \frac{U(a(\xi),b;z)}{W(a(\xi),b;z)} \,dz
               -U(a(\xi),b;y) \int_0^y g(\xi;z) \frac{M(a(\xi),b;z)}{W(a(\xi),b;z)} \,dz,
\end{align*}
with $c_1, c_2\in\RR$, and
\begin{equation}
\label{eq:Definition_Wronskian}
W(a(\xi),b;y):=-\frac{\Gamma(b)y^{-b}e^y}{\Gamma(a(\xi))},\quad\forall\, \xi \in \RR^{d-1}\less\{0\}, \ \forall\, y \in \bar\RR_+,
\end{equation}
is the Wronskian of the Kummer function, $M(a(\xi),b,y)$, and the Tricomi function, $U(a(\xi),b,y)$, \cite[\S\,13.1.22]{AbramStegun}.
We want to find a solution, $v\in C^{\infty}(\bar\RR_+)$, to \eqref{eq:Kummer_ODE}.

From \eqref{eq:Asymptotics_M}, we see that the function $M(a(\xi),b;y)$ is unbounded as $y$ tends to $+\infty$, and so we choose the constant $c_1=0$, because we only consider bounded solutions. At $y=0$, we obtain from \eqref{eq:Derivative_U} and \eqref{eq:Value_at_0_U} that $U'(a(\xi),b,y)$ is unbounded, since $b>0$, and so we choose the constant $c_2=0$, because we only consider solutions to the Kummer equation which are smooth on $\bar \RR_+$. Thus, we obtain
\begin{equation}
\label{eq:Definition_v}
\begin{aligned}
v(\xi;y) &= -M(a(\xi),b;y) \int_y^{\infty} g(\xi;z) \frac{U(a(\xi),b;z)}{W(a(\xi),b;z)} \,dz\\
         &\quad  - U(a(\xi),b;y) \int_0^y g(\xi;z)\frac{M(a(\xi),b;z)}{W(a(\xi),b;z)} \,dz.
\end{aligned}
\end{equation}
Given $v$ defined as above, and $\tilde u$ defined as in \eqref{eq:Definition_tilde_u}, we will prove the following properties of the solution, $\tilde u$, to verify that $u$ defined by \eqref{eq:Inverse_Fourier_transform_tilde_u} is a $C^{\infty}(\bar \HH)$ solution to $Au=f$ on $\HH$, as asserted by Theorem \ref{thm:Existence_smooth_solutions_half_space}.

\begin{lem}[Properties of $\tilde u$]
\label{lem:Properties_tilde_u}
If $f\in C^{\infty}_0(\bar\HH)$, then the function $\tilde u$ defined by \eqref{eq:Definition_tilde_u} has the following properties.
\begin{enumerate}
\item\label{item:Limit_infty_tilde_u} For all $\xi\in\RR^{d-1}\less\{0\}$, we have
\begin{equation}
\label{eq:Limit_infty_tilde_u}
\lim_{x_d \nearrow \infty} \tilde u(\xi;x_d)=0.
\end{equation}
\item\label{item:Smoothness_tilde_u} The function $\tilde u(\xi;\cdot)$ belongs to $C^{\infty}(\bar\RR_+)$, for all $\xi\in\RR^{d-1}\less\{0\}$.
\item\label{item:Boundedness_tilde_u} The function $\tilde u(\xi;\cdot)$ obeys
\begin{equation}
\label{eq:Boundedness_tilde_u}
|\tilde u(\xi;x_d)| < \frac{1}{c} \sup_{y \geq 0} |\tilde f(\xi;y)|,\quad\forall\, \xi\in\RR^{d-1}\less\{0\},\ \forall\, x_d \in \bar\RR_+,
\end{equation}
where $c$ is the zeroth-order coefficient of $A$ in \eqref{eq:defnA}.
\item\label{item:Superpolynomialdecay_tilde_u} The function $\tilde u(\cdot;x_d)$ decays faster than any polynomial in $\xi$, for all $x_d\in \bar\RR_+$.
\item\label{item:Superpolynomialdecay_Dkxd_tilde_u} The functions $D^k_{x_d} \tilde u$ decay faster than any polynomial in $\xi$, for all $k \in \NN$.
\end{enumerate}
\end{lem}
\end{step}

\begin{step}[Existence of a solution, $u\in C^{\infty}(\bar \HH)$, to $Au=f$ on $\HH$]
\label{step:Solution_model_operator_particular_case}
From Lemma \ref{lem:Properties_tilde_u}, Items \eqref{item:Smoothness_tilde_u} and \eqref{item:Superpolynomialdecay_tilde_u}, we see that the function $u$ defined by \eqref{eq:Inverse_Fourier_transform_tilde_u} has an arbitrary number of derivatives in the first $(d-1)$-variables which are continuous on $\bar\HH$. From Lemma \ref{lem:Properties_tilde_u}, Item \eqref{item:Smoothness_tilde_u}, we see that $u$ also admits an arbitrary number of derivatives in the $x_d$-variable, and they are continuous on $\bar\HH$. Now we consider $D^{\beta_d e_d}u$, for $\beta_d \in\NN$, which satisfies
$$
D^{\beta_d e_d}u(x',x_d):=\frac{1}{(2\pi)^{d/2}}\int_{\RR^{d-1}}D^{\beta_d e_d}\tilde u(\xi;x_d)e^{ix'\xi}\,d\xi,\quad\forall\, x_d\in\RR_+.
$$
By Lemma \ref{lem:Properties_tilde_u}, Item \eqref{item:Superpolynomialdecay_Dkxd_tilde_u}, the function $D^{\beta_d e_d}\tilde u$ decays faster than any polynomial in $\xi$, and so $D^{\lambda}D^{\beta_d e_d}u$ exists and is continuous on $\bar\HH$, for all $\lambda\in\NN^d$ with $\lambda_d=0$. Thus, $u$ belongs to $C^{\infty}(\bar\HH)$. Since $\tilde u$ solves \eqref{eq:ODE_tilde_u}, we find that $u$ solves $Au=f$ on $\HH$ by taking the inverse Fourier transform of $\tilde u(\xi;x_d)$ in $\xi\in\RR^{d-1}$. From \cite[Theorem 5.4]{Feehan_maximumprinciple}, it follows that $u$ is the unique $C^{\infty}(\bar \HH)$ solution to $Au=f$ on $\HH$.
\end{step}
Aside from the proof of Lemma \ref{lem:Properties_tilde_u}, given below, this completes the proof of Theorem \ref{thm:Existence_smooth_solutions_half_space}.
\end{proof}

It remains to prove Lemma \ref{lem:Properties_tilde_u}.

\begin{proof}[Proof of Lemma \ref{lem:Properties_tilde_u}]
We organize the proof into several steps.

\setcounter{step}{0}
\begin{step}[Proof of Item \eqref{item:Limit_infty_tilde_u}]
\label{step:Limit_infty_tilde_u}
First, we verify that the function $v$ defined in Equation \eqref{eq:Definition_v} is well-defined. We write $v=v^1+v^2$, where we set
\begin{align*}
v^1(\xi;y)&:= -M(a(\xi),b;y) \int_y^{\infty} g(\xi;z) \frac{U(a(\xi),b;z)}{W(a(\xi),b;z)} \,dz,\\
v^2(\xi;y)&:= -U(a(\xi),b;y) \int_0^y g(\xi;z) \frac{M(a(\xi),b;z)}{W(a(\xi),b;z)} \,dz.
\end{align*}
Recall that $g(\xi;\cdot)$ has compact support in $\bar\RR_+$, and so for the function $v^1(\xi;y)$, we only need to verify that it is continuous up to $y=0$. From the property \eqref{eq:Value_at_0_M} in Lemma \ref{lem:ConfluentHypergeometricFunctionProperties}, we know that $M(a(\xi),b;y)$ is continuous in $y$ up to $y=0$ with $M(a(\xi),b;0)=1$. Identities \eqref{eq:Value_at_0_U} and definition \eqref{eq:Definition_Wronskian} of the Wronskian  imply
that\footnote{Let $f$, $g$ be real-valued functions defined on $(0,\infty)$, with $g$ non-negative. We say that $f(y)\lesssim g(y)$ near $y=0$, if there are positive constants, $C$ and $y_0$, such that $|f(y)| \leq C g(y)$, for all $y\in(0,y_0)$.}
\begin{align*}
\frac{U(a(\xi),b;y)}{W(a(\xi),b;y)} \lesssim
\begin{cases}
\max\{y,\ y^{2(b-2)}\},& \quad\hbox{if } b>2,\\
y \log y,& \quad\hbox{if } b=2,\\
y,& \quad\hbox{if } 1\leq b<2,\\
y^b,& \quad\hbox{if } 0<b<1,
\end{cases}
\end{align*}
and so this function is integrable near $y=0$. Since $g(\xi;\cdot)$ has compact support in $\bar\RR_+$, we see that $v^1(\xi,\cdot) \in C(\bar\RR_+)$ and
\begin{equation*}
\lim_{y \nearrow \infty} v^1(\xi;y)=0,\quad\forall\, \xi\in\RR^{d-1}\less\{0\}.
\end{equation*}
Next, we consider the behavior of the function $v^2(\xi;\cdot)$. Near $y=0$, the property \eqref{eq:Value_at_0_M} and definition \eqref{eq:Definition_Wronskian} of the Wronskian yield
$$
\frac{M(a(\xi),b;y)}{W(a(\xi),b;y)} \lesssim y^b.
$$
Combining this result with the asymptotic behavior \eqref{eq:Value_at_0_U} of $U$ as $y\rightarrow 0$, we find that the limit of $v^2(\xi;y)$
exists as $y\to 0$. The limit of the integral,
$$
\int_0^y g(\xi;z) \frac{M(a(\xi),b;z)}{W(a(\xi),b;z)} \,dz,
$$
as $y\rightarrow\infty$ obviously exists because the function $g(\xi;\cdot)$ has compact support in $\bar\RR_+$. Moreover, using the asymptotic behavior \eqref{eq:Asymptotics_U} of $U(a(\xi),b;y)$ as $y\rightarrow +\infty$, we obtain
\begin{equation*}
\lim_{y \nearrow \infty} v^2(\xi;y)=0,\quad\forall\, \xi\in\RR^{d-1}\less\{0\}.
\end{equation*}
Since $v=v^1+v^2$, we obtain the limit property \eqref{eq:Limit_infty_tilde_u} for $\tilde u$ as $y\rightarrow +\infty$ using \eqref{eq:Definition_tilde_u}.
\end{step}

\begin{step}[Proof of Item \eqref{item:Smoothness_tilde_u}]
\label{step:Smoothness_tilde_u}
The argument employed in Step \ref{step:Limit_infty_tilde_u} shows that $\tilde u(\xi;\cdot)\in C(\bar\RR_+)$, for all $\xi\in\RR^{d-1}\less\{0\}$. Next, we want to show that $D^k_{x_d} \tilde u(\xi;\cdot) \in C(\bar \RR_+)$, for all $k \in \NN$ and $\xi\in\RR^{d-1}\less\{0\}$, but for this it suffices to show that $D^k_{y} v(\xi;\cdot)\in C(\bar \RR_+)$, for all $k \in \NN$, by \eqref{eq:Definition_tilde_u}.

We first consider the case $k=1$. A direct calculation shows that
\begin{align*}
v_y(\xi;y) &= -M_y(a(\xi),b;y) \int_y^{\infty} g(\xi;z) \frac{U(a(\xi),b;z)}{W(a(\xi),b;z)} \,dz\\
           &\quad  -U_y(a(\xi),b;y) \int_0^y g(\xi;z)\frac{M(a(\xi),b;z)}{W(a(\xi),b;z)} \,dz.
\end{align*}
Using identities \eqref{eq:Derivative_M} and \eqref{eq:Derivative_U}, we obtain
\begin{equation}
\label{eq:Identity_v_y}
\begin{aligned}
v_y(\xi;y) &= -\frac{a(\xi)}{b} M(a(\xi)+1,b+1;y) \int_y^{\infty} g(\xi;z) \frac{U(a(\xi),b;z)}{W(a(\xi),b;z)} \,dz\\
           &\quad +a(\xi) U(a(\xi)+1,b+1;y) \int_0^y g(\xi;z)\frac{M(a(\xi),b;z)}{W(a(\xi),b;z)} \,dz,
\end{aligned}
\end{equation}
and the same argument as used in the beginning of the proof of Lemma \ref{lem:Properties_tilde_u} gives us immediately that $v_y(\xi;\cdot)\in C(\bar \RR_+)$. Hence, $v(\xi;\cdot) \in C^1(\bar \RR_+)$, for all
$\xi\in\RR^{d-1}\less\{0\}$.

We next show that $v_y(\xi;\cdot)$ in \eqref{eq:Identity_v_y} is the unique $C^1(\bar \RR_+)$ solution to the Kummer equation,
\[
-y w_{yy}(\xi;y) - (b+1-y)w_y(\xi;y) + (a(\xi)+1)w(\xi;y)=g_y(\xi;y),\quad\forall\, y\in \RR_+.
\]
Our goal is to show that $v_y=w$, where we define
\begin{align*}
w(\xi;y) &: = -M(a(\xi)+1,b+1;y) \int_y^{\infty} g_z(\xi;z) \frac{U(a(\xi)+1,b+1;z)}{W(a(\xi)+1,b+1;z)} \,dz
\\
&\qquad -U(a(\xi)+1,b+1;y) \int_0^y g_z(\xi;z)\frac{M(a(\xi)+1,b+1;z)}{W(a(\xi)+1,b+1;z)} \,dz,
\end{align*}
for $y\in \RR_+$, $\xi\in\RR^{d-1}\less\{0\}$.  Integrating by parts in the expression of $w$, we obtain
\begin{align*}
w(\xi;y) & = M(a(\xi)+1,b+1;y) \int_y^{\infty} g(\xi;z) \frac{U_z W - U W_z}{W^2}(a(\xi)+1,b+1;z) \,dz\\
         &\qquad + U(a(\xi)+1,b+1;y) \int_0^y g(\xi;z)\frac{M_z W - M W_z}{W^2}(a(\xi)+1,b+1;z) \,dz.
\end{align*}
The expression for $v_y$ in \eqref{eq:Identity_v_y} coincides with that of $w$ if
\begin{align*}
-\left(\frac{U_z W - U W_z}{W^2}\right)(a(\xi)+1,b+1;z) &= \frac{a(\xi)}{b} \left(\frac{U}{W}\right)(a(\xi),b;z),
\\
-\left(\frac{M_z W - M W_z}{W^2}\right)(a(\xi)+1,b+1;z) &= -a(\xi) \left(\frac{M}{W}\right)(a(\xi),b;z).
\end{align*}
But the preceding two identities follow from the definition of the Wronskian, $W$, in \eqref{eq:Definition_Wronskian}, from \eqref{eq:Derivative_M} and \eqref{eq:Derivative_U}, and from the recursion relations \eqref{eq:Useful_identity_M} and \eqref{eq:Useful_identity_U}. Hence, the function $v_y$ is the unique $C^1(\bar\RR_+)$ solution to the corresponding  Kummer equation, which obviously implies that $v_{yy}\in C(\bar\RR_+)$.

Inductively, it follows that, for any $k\in\NN$, the derivative $D^k_y v$ exists and is the unique $C^1(\bar\RR_+)$ solution to the Kummer equation,
\[
-y (D^k_y v)_{yy}(\xi;y) - (b+k-y)(D^k_y v)_y(\xi;y) + (a(\xi)+k)D^k_y v(\xi;y)=D^k_y g(\xi;y),\quad\forall\, y\in \RR_+.
\]
Thus, $v(\xi;\cdot)\in C^{\infty}(\bar\RR_+)$, and so $\tilde u(\xi;\cdot)\in C^{\infty}(\bar\RR_+)$ by \eqref{eq:Definition_tilde_u}, and $D^k \tilde u(\xi;\cdot)$ satisfies the ordinary differential equation, for all $x_d\in (0,\infty)$ and $k \in \NN$,
\begin{equation}
\label{eq:ODE_D_k_tilde_u}
\begin{aligned}
&-x_d (D^k_{x_d} \tilde u)_{x_dx_d}(\xi;x_d) - (b+k-2|\xi|x_d)(D^k_{x_d} \tilde u)_{x_d}(\xi;x_d) + (a(\xi)+k)D^k_{x_d} \tilde u(\xi;x_d)\\
&\quad=(2|\xi|)^{k+1} D^k_y g(\xi;2|\xi|x_d).
\end{aligned}
\end{equation}
Notice that the right-hand side in the preceding equation is a function with compact support in $\bar\RR_+$.
\end{step}

\begin{step}[Proof of Items \eqref{item:Boundedness_tilde_u} and \eqref{item:Superpolynomialdecay_tilde_u}]
\label{step:Decay_tilde_u}
We adapt the method of the proof of \cite[Theorem I.1.2]{DaskalHamilton1998}. We fix $\xi \neq 0$. We write
$\tilde u(\xi;x_d) = p(\xi;x_d)+iq(\xi;x_d)$ and $\tilde f(\xi;x_d)=\tilde g(\xi;x_d) +i \tilde h(\xi;x_d)$. Then, equation \eqref{eq:ODE_tilde_u} becomes
\begin{equation*}
\begin{cases}
-x_d p_{x_dx_d}(\xi;x_d) - b^d p_{x_d}(\xi;x_d) + (c+x_d|\xi|^2)p(\xi;x_d) - b \xi q(\xi;x_d)=\tilde g(\xi;x_d),\\
-x_d q_{x_dx_d}(\xi;x_d) - b^d q_{x_d}(\xi;x_d) + (c+x_d|\xi|^2)q(\xi;x_d) + b \xi p(\xi;x_d)=\tilde h(\xi;x_d),
\end{cases}
\end{equation*}
where we use $b\xi$ as an abbreviation for the inner product of $(b^1,\ldots,b^{d-1})$ with $\xi = (\xi_1,\ldots,\xi_{d-1})\in\RR^{d-1}$.
Defining $F(\xi;x_d):=|\tilde u(\xi;x_d)|^2=p^2(\xi;x_d)+q^2(\xi;x_d)$, we obtain (where now we omit the $(\xi;x_d)$-variables)
\begin{align*}
F_{x_d}    &= 2pp_{x_d}+2qq_{x_d}, \\
F_{x_dx_d} &= 2p_{x_d}^2+2pp_{x_dx_d}+2q_{x_d}^2+2qq_{x_dx_d},
\end{align*}
which gives us
\begin{align*}
{}&x_dF_{x_dx_d} + b^d F_{x_d} - 2c F \\
&\qquad=  2p (x_dp_{x_dx_d} + b^d p_{x_d} - c p) + 2q (x_dq_{x_dx_d} +b^d q_{x_d} - c q) + 2x_d(p_{x_d}^2+q_{x_d}^2)\\
&\qquad\geq 2p \left(-\tilde g +x_d|\xi|^2p-b \xi q\right) + 2q\left(-\tilde h +x_d|\xi|^2q+b \xi p\right)\\
&\qquad\geq -2p\tilde g - 2q\tilde h\\
&\qquad\geq -c F - \frac{1}{c} \left(\tilde g^2+\tilde h^2\right),
\end{align*}
where we recall that $c>0$ by hypothesis, and so it follows that
\[
x_dF_{x_dx_d}(\xi;x_d) +b^d F_{x_d}(\xi;x_d) - c F(\xi;x_d) \geq -\frac{1}{c} \sup_{x_d\in\bar\RR_+} |\tilde f(\xi;x_d)|^2,
\quad\forall\, x_d\in\RR_+.
\]
Now let
$$
G(\xi;x_d) := F(\xi;x_d) - \frac{1}{c^2} \sup_{x_d\in\bar\RR_+}|\tilde f(\xi;x_d)|^2.
$$
Then,
\begin{align*}
\begin{cases}
x_dG_{x_dx_d}(\xi;x_d) +b^d G_{x_d}(\xi;x_d) - c G(\xi;x_d) \geq 0,\\
\lim_{x_d\nearrow\infty}G(\xi;x_d) \leq 0,
\end{cases}
\end{align*}
where we used \eqref{eq:Limit_infty_tilde_u} to determine the behavior of $G(\xi;x_d)$ as $x_d\to\infty$.
Therefore, the function $G(\xi;x_d)$ is bounded, and the weak maximum principle \cite[Theorem 5.4]{Feehan_maximumprinciple} then implies that $G(\xi,x_d)\leq 0$, for all $x_d \in \bar\RR_+$, and so
\begin{equation*}
|\tilde u(\xi;x_d)|^2 \leq \frac{1}{c^2} \sup_{y\in\bar\RR_+}|\tilde f(\xi;y)|^2,\quad\forall\, x_d\in\bar\RR_+,
\end{equation*}
which is equivalent to \eqref{eq:Boundedness_tilde_u}.

Since $f$ belongs to $C^{\infty}_0(\bar\HH)$, the function $\sup_{y\in\bar\RR_+}|\tilde f(\xi;y)|$ decays faster than any polynomial in $\xi$ by \cite[Theorem 8.22 (e)]{Folland_realanalysis}. Therefore, from \eqref{eq:Boundedness_tilde_u} we see that the function $\tilde u$ also decays faster than any polynomial in $\xi$.
\end{step}

\begin{step}[Proof of Item \eqref{item:Superpolynomialdecay_Dkxd_tilde_u}]
By \eqref{eq:ODE_D_k_tilde_u} and the fact that the right-hand side in \eqref{eq:ODE_D_k_tilde_u} is a function with compact support in $\bar\RR_+$, we see that the preceding steps can be applied to $D^k_{x_d} \tilde u$ instead of $\tilde u$, for all $k\in \NN$. Therefore, we obtain that the functions $D^k_{x_d} \tilde u$ decay faster than any polynomial in $\xi$, for all $k \in \NN$.
\end{step}

This completes the proof of Lemma \ref{lem:Properties_tilde_u}.
\end{proof}

We now prove the existence and uniqueness of smooth solutions on slabs in the half-space. We fix $\nu>0$ and recall from \eqref{eq:Slab} that $S=\RR^{d-1}\times (0,\nu)$, so that $\partial_0 S = \RR^{d-1}\times \{0\}$ and $\partial_1 S = \RR^{d-1}\times \{\nu\}$. We have the following elliptic analogue of \cite[Theorem I.1.2]{DaskalHamilton1998} in the parabolic case, but for finite-height slabs rather than the half-space.

\begin{thm}[Existence and uniqueness of a $C^\infty(\bar S)$ solution on a slab when $A$ has constant coefficients]
\label{thm:Existence_smooth_solutions_slab}
Let $A$ be an operator of the form \eqref{eq:defnA} and require that the coefficients, $a,b,c$, are constant with $b^d>0$ and $c\geq 0$.
Then, for any function, $f\in C^{\infty}_0(\bar S)$, there is a unique solution, $u \in C^{\infty}(\bar S)$, to
\begin{equation}
\label{eq:Equation_on_slabs_constant_coeff}
\begin{cases}
A u=f&\quad\hbox{on } S,\\
u=0&\quad\hbox{on } \partial_1 S.
\end{cases}
\end{equation}
\end{thm}

\begin{proof}
The method of the proof is the same as that of Theorem \ref{thm:Existence_smooth_solutions_half_space}, so we only highlight the main differences. Uniqueness of the solution, $u \in C^{\infty}(\bar S)$, follows from the weak maximum principle, Lemma \ref{lem:Comparison_principle_A_0}, for $A$. By analogy with \eqref{eq:ODE_tilde_u}, for each $\xi\in\RR^{d-1}\less\{0\}$, we construct the function $\tilde u(\xi; \cdot)$ to be the unique solution in $C^{\infty}([0,x^0_d])$ to
\begin{align*}
-x_d \tilde u_{x_dx_d}(\xi;x_d)-b^d\tilde u_{x_d}(\xi;x_d) + \left(c+i \sum_{k=1}^{d-1} b^k\xi_k + |\xi|^2 x_d\right)\tilde u(\xi;x_d)
&= \tilde f(\xi;x_d), \quad\forall\, x_d\in (0,\nu),\\
\tilde u(\xi; \nu) &= 0,
\end{align*}
by defining the new function, $v(\xi;\cdot)$, by \eqref{eq:Definition_tilde_u} and proving that $v(\xi;\cdot)$ is the unique solution in $C^{\infty}([0,x^0_d])$ to the Kummer equation,
\begin{align*}
-y v_{yy}(\xi;y)-(b-y)v_y(\xi;y) +a(\xi)v(\xi;y) &= g(\xi;y),\quad\forall\, y\in (0,2|\xi|\nu),
\\
v(\xi; 2|\xi|\nu) &= 0,
\end{align*}
for each $\xi\in\RR^{d-1}\less\{0\}$, where the coefficients $b$ and $a(\xi)$, and the function $g$ are defined in the same way as in \eqref{eq:Coeff_Kummer_ODE}. The arguments employed in the proof of Theorem \ref{thm:Existence_smooth_solutions_slab} show now that the unique solution in $C^{\infty}(\bar S)$ to the preceding ordinary differential equation is given by
\begin{align*}
v(\xi;y) &:= C M(a(\xi),b,y) -M(a(\xi),b;y) \int_y^{\infty} g(\xi;z) \frac{U(a(\xi),b;z)}{W(a(\xi),b;z)} \,dz\\
         &\quad  - U(a(\xi),b;y) \int_0^y g(\xi;z)\frac{M(a(\xi),b;z)}{W(a(\xi),b;z)} \,dz,
\end{align*}
where the constant $C$ is chosen such that the boundary condition, $v(\xi, 2|\xi| \nu)=0$, is satisfied. The only remaining modification that we need lies in Step \ref{step:Decay_tilde_u} of the proof of Lemma \ref{lem:Properties_tilde_u}. The reason why this part of the proof does not adapt immediately is because we used the fact that the zeroth-order coefficient, $c$, of $A$ in \eqref{eq:defnA} is strictly positive to derive \eqref{eq:Boundedness_tilde_u}, while now we assume $c \geq 0$. To circumvent this issue, we apply the method of the proof of Step \ref{step:Decay_tilde_u} of Lemma \ref{lem:Properties_tilde_u} not to $F$, but to $e^{-\sigma x_d} F$, where we choose the positive constant, $\sigma$, small enough. Notice that this is the same as the approach we employed in the proof of Lemma \ref{lem:Comparison_principle_A_0} to overcome the fact that $c=0$.
\end{proof}

\begin{cor}[Existence and uniqueness of a $C^{k, 2+\alpha}_s$ solution on a slab when $A$ has constant coefficients]
\label{cor:Existence_solutions_slab}
Let $\alpha\in(0,1)$ and $k\in \NN$. Let $A$ be an operator as in \eqref{eq:defnA} and require that the coefficients, $a,b,c$, are constant with $b^d>0$ and $c\geq 0$. If $f \in C^{k, \alpha}_s(\bar S)$, then there is a unique solution $u \in C^{k, 2+\alpha}_s(\bar S)$ to the boundary problem \eqref{eq:Equation_on_slabs_constant_coeff}.
\end{cor}

\begin{proof}
Uniqueness of the solution, $u \in C^{k, 2+\alpha}_s(\bar S)$, follows from the weak maximum principle, Lemma \ref{lem:Comparison_principle_A_0}, for $A$ since any $u \in C^{k, 2+\alpha}_s(\bar S)$ has the property that $Du$ and $x_dD^2u$ are continuous on $\bar S$ by Definition \ref{defn:DH2spaces} and that $x_dD^2u = 0$ on $\partial_0S$ by Lemma \ref{lem:PropSecondOrderDeriv}. Let $\{f_n\}_{n \in \NN} \subset C^{\infty}_0(\bar S)$ be a sequence such that
$f_n\to f$ in $C^{k, 2+\alpha}_s(\bar S)$ as $n\to\infty$
and
$$
\|f_n\|_{C^{k, \alpha}_s(\bar S)} \leq C \|f\|_{C^{k, \alpha}_s(\bar S)}.
$$
Such a sequence can be constructed using \cite[Theorem I.11.3]{DaskalHamilton1998}.
Let $u_n \in C^{\infty}(\bar S)$ be the unique solution to \eqref{eq:Equation_on_slabs_constant_coeff}, with $f$ replaced by $f_n$, given by Theorem \ref{thm:Existence_smooth_solutions_slab}. In particular, each solution satisfies the global Schauder estimate \eqref{eq:Global_Schauder_estimate_VariableCoefficients_nonnegc} which, when combined with the preceding inequality, gives
$$
\|u_n\|_{C^{k, 2+\alpha}_s(\bar S)} \leq C \|f\|_{C^{k, \alpha}_s(\bar S)},\quad\forall\, n \in \NN.
$$
By applying the Arzel\`a-Ascoli Theorem, we can extract a subsequence, which we continue to denote by $\{u_n\}_{n\in\NN}$, which converges in $C^{k,2+\beta}_s(\bar S)$, for all $\beta<\alpha$, to a limit function  $u \in C^{k, 2+\alpha}_s(\bar S)$  as $n\to\infty$.
Since $\{f_n\}_{n \in \NN}$, and $\{u_n\}_{n \in \NN}$, and $\{D u_n\}_{n \in \NN}$, and $\{x_d D^2u_n\}_{n \in \NN}$  also converge uniformly on compact subsets of $\bar S$ to $f$, and $u$, and $Du$, and $x_d D^2u$, respectively, as $n\to\infty$, we see that $u$ solves \eqref{eq:Equation_on_slabs_constant_coeff}.
\end{proof}

\section{Interpolation inequalities and boundary properties of functions in weighted H{\"o}lder spaces}
\label{sec:InterpolationInequalities}
A parabolic version of the following result is included in \cite[Proposition I.12.1]{DaskalHamilton1998} when $d=2$ and proved in
\cite{Feehan_Pop_mimickingdegen_pde} when $d\geq 2$ for parabolic weighted H\"older spaces. For completeness, we restate the result here for the elliptic weighted H\"older spaces used in this article.

\begin{lem}[Boundary properties of functions in weighted H{\"o}lder spaces]
\label{lem:PropSecondOrderDeriv}
\cite[Lemma 3.1]{Feehan_Pop_mimickingdegen_pde}
If $u \in C^{2+\alpha}_s(\underline{\HH})$ then, for all $x^0 \in\partial \HH$,
\begin{equation}
\label{eq:PropSecondOrderDeriv}
\lim_{\HH \ni x \rightarrow x^0} x_d D^2u(x) = 0.
\end{equation}
\end{lem}

In \cite{Feehan_Pop_mimickingdegen_pde}, we also proved the following interpolation inequalities for parabolic weighted H\"older spaces analogous to those for standard parabolic H\"older spaces \cite{Krylov_LecturesHolder, Lieberman}. For completeness, we restate these interpolation inequalities below for elliptic weighted H\"older spaces, analogous to those for standard elliptic H\"older spaces in \cite[Lemmas 6.32 and 6.35]{GilbargTrudinger}, \cite[Theorem~3.2.1]{Krylov_LecturesHolder}.

\begin{lem} [Interpolation inequalities for weighted H{\"o}lder spaces]
\label{lem:InterpolationIneqS}
\cite[Lemma 3.2]{Feehan_Pop_mimickingdegen_pde}
Let $\alpha \in (0,1)$ and $r_0>0$. Then there are positive constants, $m=m(\alpha,d)$ and $C=C(\alpha,d, r_0,)$, such that the following holds. If $u \in C^{2+\alpha}_s(\bar B_{r_0}^+(x^0))$, where $x^0 \in \partial \HH$, and $\eps \in (0,1)$, then
\begin{align}
\label{eq:InterpolationIneqS1}
\|u\|_{C^{\alpha}_s(\bar B_{r_0}^+(x^0))} &\leq \eps \|u\|_{C^{2+\alpha}_s(\bar B_{r_0}^+(x^0))} + C \eps^{-m} \|u\|_{C(\bar B_{r_0}^+(x^0))},\\
\label{eq:InterpolationIneqS2}
\|Du\|_{C(\bar B_{r_0}^+(x^0))} &\leq \eps \|u\|_{C^{2+\alpha}_s(\bar B_{r_0}^+(x^0))} + C \eps^{-m} \|u\|_{C(\bar B_{r_0}^+(x^0))},\\
\label{eq:InterpolationIneqS3}
\|x_d Du\|_{C^{\alpha}_s(\bar B_{r_0}^+(x^0))} &\leq \eps \|u\|_{C^{2+\alpha}_s(\bar B_{r_0}^+(x^0))} + C \eps^{-m} \|u\|_{C(\bar B_{r_0}^+(x^0))},\\
\label{eq:InterpolationIneqS4}
\|x_d D^2u\|_{C(\bar B_{r_0}^+(x^0))} &\leq \eps \|u\|_{C^{2+\alpha}_s(\bar B_{r_0}^+(x^0))} + C \eps^{-m} \|u\|_{C(\bar B_{r_0}^+(x^0))}.
\end{align}
\end{lem}

We add here the following

\begin{lem}[H\"older continuity for $x_d Du$]
\label{lem:Holder_continuity_x_d_Du}
Let $r>0$, and assume that $u \in C^2(B_r)$ is such that $Du$ and $x_d D^2u$ belong to $C(\bar B^+_r)$. Then, $x_d Du \in C^{\alpha}_s(\bar B^+_r)$.
\end{lem}

\begin{proof}
For this we only need to show that for any $x^1, x^2 \in B^+_{r_2}$ such that all their coordinates coincide, except for the $i$-th one, we have
$$
\frac{|x^1_d Du(x^1)-x^2_d Du(x^2)|}{s^\alpha(x^1,x^2)} \leq C,
$$
for some positive constant, $C$. We show this for the case $i=d$, and all the other cases, $i=1,\ldots,d-1$, follow by a similar argument. We have
\begin{align*}
\frac{|x^1_d Du(x^1)-x^2_d Du(x^2)|}{s^\alpha(x^1,x^2)}
&\leq \frac{|x^1_d-x^2_d|}{s^\alpha(x^1,x^2)} |Du(x^1)| + x^2_d\frac{|Du(x^1)-Du(x^2)|}{s^\alpha(x^1,x^2)} \\
&\leq \left(\|Du\|_{C(\bar B^+_{r_2})}+ x^2_d |D^2u(x^3)| \right)\frac{|x^1-x^2|}{s^\alpha(x^1,x^2)},
\end{align*}
where $x^3 \in B_{r_3}$ is a point on the line connecting $x^1$ and $x^2$, and we apply the Mean Value Theorem. We may assume without loss of generality that $x^2_d<x^1_d$, and because $x^3_d \geq x^2_d$, we have that $x^2_d |D^2u(x^3)| \leq \|x_d D^2 u\|_{C(\bar B^+_{r_2})}$. Using the definition \eqref{eq:Cycloidal_distance} of the cycloidal distance function, we obtain
\begin{align*}
\frac{|x^1_d Du(x^1)-x^2_d Du(x^2)|}{s^\alpha(x^1,x^2)}
&\leq \left(\|Du\|_{C(\bar B^+_{r_2})}+ \|x_dD^2u\|_{C(\bar B^+_{r_2})}\right) r_2^{1-\alpha/2}.
\end{align*}
Therefore, $x_d Du$ belongs to $C^{\alpha}_s(\bar B^+_{r_2})$, for all $\alpha\in (0,1)$. This completes the proof of Lemma \ref{lem:Holder_continuity_x_d_Du}
\end{proof}

%
%

\bibliography{mfpde}

\def\cprime{$'$} \def\polhk#1{\setbox0=\hbox{#1}{\ooalign{\hidewidth
  \lower1.5ex\hbox{`}\hidewidth\crcr\unhbox0}}} \def\cprime{$'$}
  \def\cprime{$'$} \def\cprime{$'$}
  \def\lfhook#1{\setbox0=\hbox{#1}{\ooalign{\hidewidth
  \lower1.5ex\hbox{'}\hidewidth\crcr\unhbox0}}} \def\cprime{$'$}
  \def\cprime{$'$} \def\cprime{$'$} \def\cprime{$'$} \def\cprime{$'$}
\providecommand{\bysame}{\leavevmode\hbox to3em{\hrulefill}\thinspace}
\providecommand{\MR}{\relax\ifhmode\unskip\space\fi MR }
\providecommand{\MRhref}[2]{%
  \href{http://www.ams.org/mathscinet-getitem?mr=#1}{#2}
}
\providecommand{\href}[2]{#2}
\begin{thebibliography}{10}

\bibitem{AbramStegun}
M.~Abramovitz and I.~A. Stegun, \emph{Handbook of mathematical functions},
  Dover, New York, 1972.

\bibitem{Adams_1975}
R.~A. Adams, \emph{Sobolev spaces}, Academic Press, Orlando, FL, 1975.

\bibitem{Athreya_Barlow_Bass_Perkins_2002}
S.~R. Athreya, M.~T. Barlow, R.~F. Bass, and E.~A. Perkins, \emph{Degenerate
  stochastic differential equations and super-{M}arkov chains}, Probab. Theory
  Related Fields \textbf{123} (2002), 484--520.

\bibitem{Brandt_1969a}
A.~Brandt, \emph{Interior estimates for second-order elliptic differential (or
  finite-difference) equations via the maximum principle}, Israel J. Math.
  \textbf{7} (1969), 95--121.

\bibitem{Crandall_Ishii_Lions_1992}
M.~G. Crandall, H.~Ishii, and P.-L. Lions, \emph{User's guide to viscosity
  solutions of second order partial differential equations}, Bull. Amer. Math.
  Soc. (N.S.) \textbf{27} (1992), 1--67.

\bibitem{Daskalopoulos_Feehan_statvarineqheston}
P.~Daskalopoulos and P.~M.~N. Feehan, \emph{Existence, uniqueness, and global
  regularity for variational inequalities and obstacle problems for degenerate
  elliptic partial differential operators in mathematical finance},
  arXiv:1109.1075.

\bibitem{DaskalHamilton1998}
P.~Daskalopoulos and R.~Hamilton, \emph{{$C^\infty$}-regularity of the free
  boundary for the porous medium equation}, J. Amer. Math. Soc. \textbf{11}
  (1998), 899--965.

\bibitem{Daskalopoulos_Rhee_2003}
P.~Daskalopoulos and E.~Rhee, \emph{Free-boundary regularity for generalized
  porous medium equations}, Commun. Pure Appl. Anal. \textbf{2} (2003),
  481--494.

\bibitem{DeSimone_Knupfer_Otto_2006}
A.~De~Simone, H.~Kn{\"u}pfer, and F.~Otto, \emph{2-d stability of the {N}\'eel
  wall}, Calc. Var. Partial Differential Equations \textbf{27} (2006),
  233--253.

\bibitem{DuffiePanSingleton2000}
D.~Duffie, J.~Pan, and K.~Singleton, \emph{Transform analysis and asset pricing
  for affine jump diffusions}, Econometrica \textbf{68} (2000), 1343--1376.

\bibitem{Epstein_Mazzeo_annmathstudies}
C.~L. Epstein and R.~Mazzeo, \emph{Degenerate diffusion operators arising in
  population biology}, Annals of Mathematics Studies, Princeton University
  Press, Princeton, NJ, 2013, arXiv:1110.0032.

\bibitem{Fabes_1982}
E.~B. Fabes, \emph{Properties of nonnegative solutions of degenerate elliptic
  equations}, Rend. Sem. Mat. Fis. Milano \textbf{52} (1982), 11--21.
  \MR{802990 (87d:35057)}

\bibitem{Fabes_Kenig_Serapioni_1982a}
E.~B. Fabes, C.~E. Kenig, and R.~P. Serapioni, \emph{The local regularity of
  solutions of degenerate elliptic equations}, Comm. Partial Differential
  Equations \textbf{7} (1982), 77--116.

\bibitem{Feehan_classical_perron_elliptic}
P.~M.~N. Feehan, \emph{A classical {P}erron method for existence of smooth
  solutions to boundary value and obstacle problems for degenerate-elliptic
  operators via holomorphic maps}, arXiv:1302.1849.

\bibitem{Feehan_maximumprinciple}
\bysame, \emph{Maximum principles for boundary-degenerate linear elliptic
  differential operators}, Communications in Partial Differential Equations, to
  appear, arXiv:1204.6613, doi:10.1080/03605302.2013.831446.

\bibitem{Feehan_Pop_regularityweaksoln}
P.~M.~N. Feehan and C.~A. Pop, \emph{Degenerate elliptic operators in
  mathematical finance and {H\"o}lder continuity for solutions to variational
  equations and inequalities}, arXiv:1110.5594.

\bibitem{Feehan_Pop_higherregularityweaksoln}
\bysame, \emph{Higher-order regularity for solutions to degenerate elliptic
  variational equations in mathematical finance}, arXiv:1208.2658.

\bibitem{Feehan_Pop_mimickingdegen_pde}
\bysame, \emph{A {S}chauder approach to degenerate-parabolic partial
  differential equations with unbounded coefficients}, Journal of Differential
  Equations \textbf{254} (2013), 4401--4445, arXiv:1112.4824.

\bibitem{Fichera_1956}
G.~Fichera, \emph{Sulle equazioni differenziali lineari ellittico-paraboliche
  del secondo ordine}, Atti Accad. Naz. Lincei. Mem. Cl. Sci. Fis. Mat. Nat.
  Sez. I. (8) \textbf{5} (1956), 1--30.

\bibitem{Fichera_1960}
\bysame, \emph{On a unified theory of boundary value problems for
  elliptic-parabolic equations of second order}, Boundary problems in
  differential equations, Univ. of Wisconsin Press, Madison, 1960, pp.~97--120.

\bibitem{Folland_realanalysis}
G.~B. Folland, \emph{Real analysis}, second ed., Wiley, New York, 1999.

\bibitem{Giacomelli_Knupfer_2010}
L.~Giacomelli and H.~Kn{\"u}pfer, \emph{A free boundary problem of fourth
  order: classical solutions in weighted {H}\"older spaces}, Comm. Partial
  Differential Equations \textbf{35} (2010), 2059--2091.

\bibitem{Giacomelli_Knupfer_Otto_2008}
L.~Giacomelli, H.~Kn{\"u}pfer, and F.~Otto, \emph{Smooth zero-contact-angle
  solutions to a thin-film equation around the steady state}, J. Differential
  Equations \textbf{245} (2008), 1454--1506.

\bibitem{GilbargTrudinger}
D.~Gilbarg and N.~Trudinger, \emph{Elliptic partial differential equations of
  second order}, second ed., Springer, New York, 1983.

\bibitem{Heston1993}
S.~Heston, \emph{A closed-form solution for options with stochastic volatility
  with applications to bond and currency options}, Review of Financial Studies
  \textbf{6} (1993), 327--343.

\bibitem{Koch}
H.~Koch, \emph{Non-{E}uclidean singular integrals and the porous medium
  equation}, Habilitation Thesis, University of Heidelberg, 1999,
  \url{www.mathematik.uni-dortmund.de/lsi/koch/publications.html}.

\bibitem{Kohn_Nirenberg_1967}
J.~J. Kohn and L.~Nirenberg, \emph{Degenerate elliptic-parabolic equations of
  second order}, Comm. Pure Appl. Math. \textbf{20} (1967), 797--872.

\bibitem{Krylov_LecturesHolder}
N.~V. Krylov, \emph{Lectures on elliptic and parabolic equations in {H}\"older
  spaces}, American Mathematical Society, Providence, RI, 1996.

\bibitem{LevendorskiDegenElliptic}
S.~Z. Levendorski{\u\i}, \emph{Degenerate elliptic equations}, Kluwer,
  Dordrecht, 1993.

\bibitem{Lieberman}
G.~M. Lieberman, \emph{Second order parabolic differential equations}, World
  Scientific Publishing Co. Inc., River Edge, NJ, 1996.

\bibitem{Murthy_Stampacchia_1968}
M.~K.~V. Murthy and G.~Stampacchia, \emph{Boundary value problems for some
  degenerate elliptic operators}, Ann. Mat. Pura Appl. \textbf{80} (1968),
  1--122.

\bibitem{Murthy_Stampacchia_1968corr}
\bysame, \emph{Errata corrige: ``{B}oundary value problems for some
  degenerate-elliptic operators''}, Ann. Mat. Pura Appl. (4) \textbf{90}
  (1971), 413--414.

\bibitem{Oleinik_Radkevic}
O.~A. Ole{\u\i}nik and E.~V. Radkevi{\v{c}}, \emph{Second order equations with
  nonnegative characteristic form}, Plenum Press, New York, 1973.

\bibitem{Radkevich_2009a}
E.~V. Radkevi{\v{c}}, \emph{Equations with nonnegative characteristic form.
  {I}}, J. Math. Sci. \textbf{158} (2009), 297--452.

\bibitem{Radkevich_2009b}
\bysame, \emph{Equations with nonnegative characteristic form. {II}}, J. Math.
  Sci. \textbf{158} (2009), 453--604.

\bibitem{RheeThesis}
E.~Rhee, \emph{Free boundary regularity in quasi linear degenerated diffusion
  equations}, {Ph.D.} thesis, University of California, Irvine, 2000.

\end{thebibliography}

\bibliographystyle{amsplain}
\end{document}